\documentclass[11pt,reqno]{amsart}
\usepackage{amsmath, color}
\usepackage{amssymb}
\usepackage{a4wide}
\usepackage{graphics}
\usepackage{epsfig}

\newtheorem{theorem}{Theorem}[section]
\newtheorem{lemma}[theorem]{Lemma}
\newtheorem{proposition}[theorem]{Proposition}

\newtheorem{remark}[theorem]{Remark}
\newtheorem{corollary}[theorem]{Corollary}
\newtheorem{example}[theorem]{Example}

\newcommand{\R}{\mathbb{R}}
\newcommand{\E}{\mathrm{E}}
\renewcommand{\P}{\mathrm{P}}

\renewcommand{\Re}{\operatorname{Re}}

\makeatletter
 \@addtoreset{equation}{section}
 \makeatother
 
\allowdisplaybreaks
\usepackage{appendix}
\usepackage{hyperref}
\title[Local asymptotic properties for a jump-type CIR process]{Local asymptotic properties for the growth rate of a jump-type CIR process}
\date{\today}
\author[Mohamed Ben Alaya, Ahmed Kebaier, Gyula Pap and Ngoc Khue Tran]{Mohamed Ben Alaya, Ahmed Kebaier, Gyula Pap and Ngoc Khue Tran}

\address{Mohamed Ben Alaya, Univ Rouen Normandie, CNRS, Normandie Univ, LMRS UMR 6085, F-76000 Rouen, France}
\email{mohamed.ben-alaya@univ-rouen.fr}

\address{Ahmed Kebaier, Laboratoire de Math\'ematiques et Mod\'elisation d'Evry, CNRS, UMR 8071, Universit\'e d'Evry, Universit\'e Paris-Saclay, 91037, Evry, France}
\email{ ahmed.kebaier@univ-evry.fr}

\address{Gyula Pap, Bolyai Institute, University of Szeged,
	Aradi v\'ertan\'uk tere 1, H--6720 Szeged, Hungary}
\email{papgy@math.u-szeged.hu}

\address{Ngoc Khue Tran, Faculty of Mathematics and Informatics, Hanoi University of Science and Technology, 1 Dai Co Viet, Hai Ba Trung, Hanoi, Vietnam}
\email{khue.tranngoc@hust.edu.vn}

\thanks{This research benefited from the support of the ANR project ``Efficient inference for large and high-frequency data" (ANR-21-CE40-0021). Ahmed Kebaier benefited from the support of the chair Risques Financiers, Fondation du Risque. This research is funded by the Vietnam Ministry of Education and Training under grant number B2024-CTT-05. \textcolor{black}{The authors sincerely thank the anonymous reviewer and the editor for their valuable comments and suggestions, which have significantly improved the quality of this manuscript.}}


\begin{document}
\maketitle
\begin{small}
	\begin{center}
		{\it In memoriam  Gyula Pap (1954 - 2019)}\footnote{  \url{http://www.math.u-szeged.hu/~barczy/Pap_Gyula_In_Memoriam_1954_2019.pdf}}	\end{center}
\end{small}

\begin{abstract} 
In this paper, we consider a one-dimensional jump-type Cox-Ingersoll-Ross process driven by a Brownian motion and a subordinator, whose growth rate is an unknown parameter. Considering the process observed continuously or discretely at high frequency, we derive the local asymptotic properties for the growth rate in both ergodic and non-ergodic cases. Local asymptotic normality (LAN) is proved in the subcritical case, local asymptotic quadraticity (LAQ) is derived in the critical case, and local asymptotic mixed normality (LAMN) is shown in the supercritical case. To obtain these results, techniques of Malliavin calculus and a subtle analysis on the jump structure of the subordinator involving the amplitude of jumps and number of jumps are essentially used.
\end{abstract}

\section{Introduction}

On a complete probability space $(\widehat{\Omega}, \widehat{\mathcal{F}},  \widehat{\P})$, we consider a one-dimensional jump-type Cox-Ingersoll-Ross (CIR) process $Y^b=(Y_t^b)_{t\in[0,\infty)}$ driven by a subordinator
\begin{equation}\label{eq1}
dY_t^b=(a-bY_t^b)dt +\sigma\sqrt{Y_t^b}dW_t+dJ_t,
\end{equation}
where $Y_0^b=y_0\in [0, \infty)$ is a given initial condition, $a\in [0, \infty)$, $b\in\R$ and $\sigma\in (0, \infty)$. Here, $W=(W_t)_{t\in[0,\infty)}$ is a one-dimensional standard Brownian motion, and $J=(J_t)_{t\in[0,\infty)}$ is a subordinator (an increasing L\'evy process) with zero drift and with L\'evy measure $m$ concentrated on $(0,\infty)$ and satisfying the following condition
\begin{list}{labelitemi}{\leftmargin=1cm}
	\item[\bf(A1)]\; \; $\int_{0}^{\infty}zm(dz)\in [0, \infty)$.
\end{list}
The Laplace transform of $J$ is given by
\begin{equation}\label{LK}
\widehat{\E}\left[e^{uJ_t}\right]=\exp\left\{t\int_0^{\infty}\left(e^{uz}-1\right)m(dz)\right\},
\end{equation}
for any $t \in [0, \infty)$ and for any complex number $u$ with
$\Re(u) \in (-\infty, 0]$, see, e.g. Sato \cite[proof of Theorem 24.11]{Sat}, where $\widehat{\E}$ denotes the expectation w.r.t. $\widehat{\P}$. We suppose that the processes $W$ and $J$ are independent. Note that the moment condition {\bf(A1)} implies that $m$ is a L\'evy measure since $\min(1, z^2) \leq z$ for $z \in (0, \infty)$. Moreover, the subordinator $J$ has sample paths of bounded variation on every compact time interval almost surely, see e.g. Sato \cite[Theorem 21.9]{Sat}. We point out the assumptions which assure the existence of a pathwise unique strong solution to the stochastic differential equation (SDE) \eqref{eq1} with $\widehat{\P}(Y_t^b \in [0, \infty) \ \textnormal{for all} \ t \in [0, \infty)) = 1$ (see Proposition \ref{Pro_jump_CIR}). In fact, $Y^b$ is a special continuous state and continuous time branching process with immigration (CBI process), see Proposition \ref{Pro_jump_CIR} below and \cite{DFS03, KM12}. Let $\R$, $\R_+$, $\R_{++}$, $\R_-$, $\R_{--}$ and $\mathbb{C}$ denote the sets of real numbers, non-negative real numbers, positive real numbers, non-positive real numbers, negative real numbers and complex numbers, respectively.  

In this paper, we will consider the jump-type CIR process $Y^b$ solution to equation \eqref{eq1} with known $a \in \R_+$, $\sigma \in \R_{++}$, $y_0 \in \R_+$ and L\'evy measure $m$ satisfying condition {\bf(A1)}, and we will consider $b \in \R$ as an unknown parameter to be estimated. Let $\{\widehat{\mathcal{F}}_t\}_{t\in\R_+}$ denote the natural filtration generated by two processes $W$ and $J$. We denote by $\widehat{\P}^{b}$ the probability measure induced by the process $Y^b$ on the canonical space $(D(\R_{+},\R),\mathcal{B}(D(\R_{+},\R)))$ endowed with the natural filtration $\{\widehat{\mathcal{F}}_t\}_{t\in\R_+}$. Here $D(\R_{+},\R)$ denotes the set of $\R$-valued c\`adl\`ag functions defined on $\R_{+}$, and $\mathcal{B}(D(\R_{+},\R))$ is its Borel $\sigma$-algebra. We denote by $\widehat{\E}^{b}$ the expectation with respect to (w.r.t.) $\widehat{\P}^{b}$. Let $\overset{\widehat{\P}^{b}}{\longrightarrow}$, $\overset{\mathcal{L}(\widehat{\P}^{b})}{\longrightarrow}$, $\widehat{\P}^{b}$-a.s., $\overset{\widehat{\P}}{\longrightarrow}$, and $\overset{\mathcal{L}(\widehat{\P})}{\longrightarrow}$ denote the convergence in $\widehat{\P}^{b}$-probability, in $\widehat{\P}^{b}$-law, in $\widehat{\P}^{b}$-almost surely, in $\widehat{\P}$-probability, and in $\widehat{\P}$-law, respectively. 

The L\'evy-It\^o decomposition of $J$ takes the form $J_t=\int_0^t\int_0^{\infty}zN(ds,dz)$ for any $t\in\R_+$, where $N(dt,dz):=\sum_{0\leq s\leq t}{\bf 1}_{\{\Delta J_s\neq 0\}}\delta_{(s,\Delta J_s)}(ds,dz)$ is an interger-valued Poisson random measure in $(\R_{+}\times\R_{+},\mathcal{B}(\R_{+}\times\R_{+}))$ with intensity measure $m(dz)dt$ associated with $J$. Here, the jump amplitude of $J$ is defined as $\Delta J_s:=J_s-J_{s-}$ for any $s\in \R_{++}$, $\Delta J_0:=0$, $\delta_{(s,z)}$ denotes the Dirac measure at the point $(s,z)\in\R_{+}\times\R_{+}$, and $\mathcal{B}(\R_{+}\times\R_{+})$ denotes the Borel $\sigma$-algebra on $\R_{+}\times\R_{+}$. Hence, \eqref{eq1} can be rewritten in the following integral form
\begin{equation}\label{eqintegral}
\begin{split}
Y_t^b=y_0+\int_0^t(a-bY_s^b)ds +\sigma\int_0^t\sqrt{Y_s^b}dW_s+\int_0^t\int_0^{\infty}zN(ds,dz),
\end{split}
\end{equation}
for any $t\in\R_+$. Observe that $\widehat{\E}\left[J_t\right]=t\int_0^{\infty}zm(dz)$, for any $t\in\R_+$.

The next proposition shows the existence and uniqueness of a strong solution of the SDE \eqref{eq1}, and also states that $Y^b$ is a CBI process.
\begin{proposition}\cite[Proposition 2.1]{BBKP17}\label{Pro_jump_CIR}
	For all $a \in \R_+$, $b \in \R$, $\sigma \in \R_{++}$, $y_0 \in \R_+$ and L\'evy measure $m$ on $\R_{++}$ satisfying condition {\bf(A1)}, there is a pathwise unique strong solution $Y^b=(Y_t^b)_{t\in\R_+}$ of the SDE \eqref{eq1} such that $\widehat{\P}(Y_0^b = y_0) = 1$ and $\widehat{\P}(\text{$Y_t^b \in \R_+$ for all $t \in \R_+$}) = 1$. Moreover, $Y^b$ is a CBI process having branching mechanism
	$$
	R(u)= \frac{\sigma^2}{2} u^2 - bu , \qquad
	\text{$u \in \mathbb{C}$ \ with \ $\Re(u) \leq 0$,}
	$$
	and immigration mechanism
	$$
	F(u) = a u + \int_0^\infty (e^{uz} - 1) m(dz) ,
	\qquad \text{$u \in \mathbb{C}$ \ with \ $\Re(u) \leq 0$.}
	$$
	Furthermore, the infinitesimal generator of $Y^b$ takes the form
	\begin{align*}
	\begin{split}
	(\mathcal{A} f)(y) = (a - by) f'(y) + \frac{\sigma^2}{2} y f''(y)	+ \int_0^\infty (f(y+z) - f(y)) m(dz),
	\end{split}
	\end{align*}
	where $y \in \R_+$, $f \in C^2_c(\R_+, \R)$, $f'$ and $f''$ denote the first and second order partial derivatives of $f$, and $C^2_c(\R_+, \R)$ denotes the set of twice continuously differentiable real-valued functions on $\R_+$ with compact support.
	
	If, in addition, $y_0 + a \in \R_{++}$, then
	$\widehat{\P}(\int_0^t Y_s^b ds \in \R_{++}) = 1$ for all $t \in \R_{++}$.
\end{proposition}
Next, we present the first moment of $Y^b$.
\begin{lemma}\cite[Proposition 2.2]{BBKP17}\label{Pro_moments}
	Assume condition {\bf(A1)}.	Then, for any $t \in \R_+$,
	\begin{align*}
	\widehat{\E}[Y_t^b]
	= \begin{cases}
	y_0 e^{-bt}	+ (a + \int_0^\infty z m(dz)) \frac{1-e^{-bt}}{b}
	& \text{if \ $b \ne 0$,} \\[1mm]
	y_0 + (a + \int_0^\infty z m(dz)) t & \text{if \ $b = 0$.}
	\end{cases}
	\end{align*}
	Consequently,
	\begin{align*}
	&\lim_{t\to\infty} \widehat{\E}[Y_t^b]
	= (a + \int_0^\infty z m(dz)) \frac{1}{b},&\textnormal{if}\quad b \in \R_{++},\\
	&\lim_{t\to\infty} t^{-1} \widehat{\E}[Y_t^b] = a + \int_0^\infty z m(dz),&\textnormal{if}\quad b = 0,\\
	&\lim_{t\to\infty} e^{bt} \widehat{\E}[Y_t^b]
	= y_0 - (a + \int_0^\infty z m(dz)) \frac{1}{b},&\textnormal{if}\quad b \in \R_{--}.
	\end{align*}	
\end{lemma}
Based on the asymptotic behavior of $\widehat{\E}[Y_t^b]$, one can introduce a classification of the jump-type CIR process \eqref{eq1}. We call $Y^b$ subcritical, critical or supercritical if $b \in \R_{++}$, $b = 0$ or $b \in \R_{--}$, respectively.

In mathematical finance, the CIR process and jump-type CIR process are used to describe the stochastic volatility of a price process of an asset in the Heston model or to model the evolution of short-term interest rates. They belong to a class of affine processes that can be found in the work of Teichmann et {\it al.} \cite{CFMT11},	Duffie, Filipovi{\'c} and Schachermayer \cite{DFS03}, Kallsen \cite{K06},  Keller-Ressel \cite{K08}, Keller-Ressel and Mijatovi{\'c} \cite{KM12},\ldots.

The concept of LAN, LAQ and LAMN can be found, e.g., in Le Cam and Yang \cite{CY90}, Jeganathan \cite{JP82} or in Subsection $7.1$ of H{\"{o}}pfner \cite{H14}. Most  research works on statistics for CIR process, CIR process with jumps and Heston models mainly focus on continuous observations. In case of the diffusion-type CIR process, Overbeck \cite{O98} examined local asymptotic properties for the drift parameters. Kutoyants \cite{K04} investigated statistical inference for one-dimensional ergodic diffusion processes. The LAMN property was addressed by Luschgy in \cite{L92} for semimartingale. Later, Ben Alaya and Kebaier \cite{BK12, BK13} show the asymptotic properties of maximum likelihood estimator (MLE) for the drift parameters of the diffusion-type CIR process in both ergodic and non-ergodic cases. More recently, Barczy et {\it al.} \cite{BBKP17} have investigated the asymptotic properties of MLE for the growth rate $b$ of the jump-type CIR process \eqref{eq1}, which provides the main inspiration for our current work. Barczy and Pap \cite{BP16}, Barczy et {\it al.} \cite{BBKP16, BBKP18} have studied the asymptotic properties of MLE for Heston models, jump-type Heston models, and stable CIR process, respectively. 

In the case of discrete observations, Malliavin calculus techniques have been applied to study diffusion processes and jump diffusion processes. Concretely, Gobet \cite{G01, G02} obtained the LAMN and LAN properties for multidimensional elliptic diffusions and ergodic diffusions. In the presence of jumps, several SDEs have been largely investigated, see e.g. Kawai \cite{K13}, Cl\'ement et {\it al.} \cite{CDG14, CG15, CGN19}, Kohatsu-Higa et {\it al.} \cite{KNT14, KNT15}, and Tran \cite{T15}. Notice that all these results deal with the SDEs whose coefficients are continuously differentiable and satisfy a global Lipschitz condition. The case where the coefficients of the SDEs do not satisfy these standard assumptions is less investigated. The first contribution in this direction can be found in Ben Alaya et {\it al.} \cite{BKT17} where the local asymptotic properties for the global drift parameters $(a,b)$ of the diffusion-type CIR process are studied.

The LAN, LAMN and LAQ properties for jump-type CIR processes on the basis of both continuous and discrete observations have never been addressed in the literature. Motivated by this fact and inspired by the recent paper \cite{BBKP17}, the main objective of this paper is to study the local asymptotic properties for the parameter $b$ of the jump-type CIR process \eqref{eq1} in both ergodic and non-ergodic cases. Precisely, LAN will be proved in the subcritical case, LAQ will be derived in the critical case and LAMN will be shown in the supercritical case. Let us mention here that the study of the parameter $a$ requires the asymptotic behavior of $\int_0^t\frac{1}{Y_s^b}ds$ which still remains an open problem.

For our purpose, in the case of continuous observations, we use an explicit expression of the Radon-Nikodym derivative, the asymptotic behaviors \eqref{ergo1}, \eqref{criticalconvergence1} and \eqref{supercriticalconvergence2} together with the central limit theorem for continuous local martingales. The jump-type CIR process does not possess an explicit expression for its transition density and then proving the local asymptotic properties based on discrete observations becomes challenging. To overcome this difficulty in the case of discrete observations, our strategy is first to prove the existence of a positive transition density, which is of class $C^1$ w.r.t. $b$  (see Proposition \ref{smoothdensity}) and then to use the Malliavin calculus approach developed by Gobet \cite{G01, G02} for regular SDEs in order to derive an explicit expression for the logarithm derivative of the transition density in terms of a conditional expectation of a Skorohod integral (see Proposition \ref{c2prop1}). This allows to obtain an appropriate stochastic expansion of the log-likelihood ratio (see \eqref{decompo}, \eqref{firstdeco} and \eqref{decompo2}). The main terms of the expansion can be treated by writing them in terms of the log-likelihood ratio based on the continuous observations (see \eqref{m1}).

To treat the negligible terms of the expansion, the first difficulty comes from the fact that the conditional expectations are taken under the measure $\widetilde{\P}_{t_k,x}^{b(\ell)}$ with the application of the Malliavin calculus whereas the convergence is evaluated under $\widehat{\P}_{t_k,x}^{b_0}$ with $\widehat{\P}_{t_k,x}^{b_0}\neq \widetilde{\P}_{t_k,x}^{b(\ell)}$ (see e.g. Lemma \ref{lemma1}). In \cite{G01, G02}, the authors use a change of transition density functions and the upper and lower bounds of Gaussian type of the transition density functions. For $Y^b$, the transition density estimates of Gaussian type may not exist. To overcome this difficulty, a change of measures is essentially used (see Lemma \ref{change}). Then a technical Lemma \ref{deviation1} is established to measure the deviation of the change of measure when the drift parameters change. The second difficulty is to deal with the jumps of the subordinator in the subcritical case since with the rate $\frac{1}{\sqrt{n\Delta_n}}$ in this case, the Burkholder-Davis-Gundy's inequality cannot be used in order to get the enough decreasing rate (see Lemma \ref{lemma4}). To resolve this problem, a subtle analysis on the jump structure of the subordinator involving the amplitude of jumps and number of jumps is mainly used. More concretely, as in \cite[Lemma A.14]{KNT15}, we condition on the number of big jumps outside and inside the conditional expectation. When the number of big jumps outside and inside the conditional expectation is different, one may use a large deviation principle in the estimate (see \eqref{M0} and \eqref{M01}) and otherwise, we rely our analysis on using the complementary set (see \eqref{M11}). All these arguments used together with the usual moment estimates lead us to get the exact large deviation type estimates where the decreasing rate is determined by the intensity of the big jumps and the asymptotic behavior of the small ones (see Lemma \ref{jumpestimate2}). It is worth noting that in \cite[Lemma A.14]{KNT15}, the authors needed to use lower bounds estimates for the transition density and prove upper bounds estimates for the transition density conditioned on the jump structure in order to obtain the large deviation type estimates. Besides, they only studied the finite L\'evy measure. In our study, both finite and infinite L\'evy measures are considered and the approach based on deriving estimates for the transition density conditioned on the jump structure is not needed. The techniques for analyzing the jumps in this paper can be viewed as an improvement on the approach developed in \cite{KNT15}. In the critical and supercritical cases, there is no difficulty in treating the jumps thanks to the better rates  $\frac{1}{n\Delta_n}$ and $e^{b_0\frac{n\Delta_n}{2}}$ for $b_0\in \R_{--}$, respectively.

To be able to apply the Malliavin calculus, some crucial estimates on the flow process in Lemma \ref{estimates} are needed to show an expression for the score function in Proposition \ref{c2prop1} and useful estimates in Lemma \ref{estimate}. In addition, to prove the convergence of the negligible terms, some positive and negative moment estimates of the jump-type CIR process in Lemma \ref{moment} are useful. As a consequence, condition {\bf(A3)} below turns out to be crucial. When following our aforementioned strategy, in the subcritical (ergodic) case, some additional assumptions on the decreasing rate of $\Delta_n$ such as $n\Delta_n^{p}\to 0$ for some $p>1$ are not required. However, in the nonergodic cases, we require $n\Delta_n^{\delta}\to 0$ for the critical case where $\delta>1$ is arbitrarily large, and $\Delta_n^{2-\varepsilon}e^{-b_0n\Delta_n}\to 0$ for the supercritical case where $\varepsilon>0$ is arbitrarily small.

The paper is organized as follows. In Section \ref{results}, we state our main results in Theorems \ref{Thm_subcritical}, \ref{Thm_critical}, \ref{Thm_supercritical}, \ref{theorem1}, \ref{theorem2}, and \ref{theorem3}. Furthermore, several examples on the subordinators are given. Section 3 is devoted to the proof of the main results. Section 4 presents technical results to obtain an explicit expression for the score function by using the Malliavin calculus. In Section 5, we focus on the study of the convergence of the remainder terms. To maintain the flow of the exposition, the proofs of some technical results are presented in Appendix A whereas some useful results are recalled in Appendix B.

\section{Main results}
\label{results}

In this section, we give a statement of our main results divided into three cases: subcritical, critical and supercritical for both continuous and discrete observations.
\subsection{Continuous observations}
For all  $T \in \R_{++}$, let
$\widehat{\P}_{T}^b := \widehat{\P}^{b}\vert_{\widehat{\mathcal{F}}_T}$ be the restriction of $\widehat{\P}^b$ on $\widehat{\mathcal{F}}_T$, and consider the continuous observation $Y^{T,b}:=(Y_t^b)_{t\in[0,T]}$ of the process $Y^b$ on the time interval $[0,T]$. The Radon--Nikodym derivative is given in the next lemma.
\begin{lemma}\cite[Proposition 4.1]{BBKP17}\label{RNb}
	Let $b, \widetilde{b}\in \R$. Then for all $T \in \R_{++}$, the probability measures $\widehat{\P}_{T}^b$ and $\widehat{\P}_{T}^{\widetilde{b}}$ are absolutely continuous w.r.t. each other, and
	\begin{equation*}\label{RNformulab}
	\begin{aligned}
	\log\frac{d\widehat{\P}_{T}^{\widetilde{b}}}{d\widehat{\P}_{T}^b}(Y^{T,b})
	&= - \frac{\widetilde{b}-b}{\sigma^2} (Y_T^b - y_0 - a T - J_T)
	- \frac{\widetilde{b}^2-b^2}{2\sigma^2} \int_0^T Y_s^b  d s \\
	&= - \frac{\widetilde{b}-b}{\sigma} \int_0^T \sqrt{Y_s^b}  dW_s
	- \frac{(\widetilde{b}-b)^2}{2\sigma^2} \int_0^T Y_s^b ds .
	\end{aligned}
	\end{equation*}
	Moreover, the process  $\Bigl(\frac{d\widehat{\P}_{T}^{\widetilde{b}}}{d\widehat{\P}_{T}^b}(Y^{T,b})\Bigr)_{T\in\R_+}$ is a martingale.
\end{lemma}
The martingale property of the process $\Bigl(\frac{d\widehat{\P}_{T}^{\widetilde{b}}}{d\widehat{\P}_{T}^b}(Y^{T,b})\Bigr)_{T\in\R_+}$ is a consequence of Theorem 3.4 in Chapter III of Jacod and Shiryaev \cite{JS03}.
\begin{corollary}\label{RN_Cor}
	For any $b \in \R$, $T \in \R_{++}$, $\varphi_T(b)\in\R$ and $u\in \R$, we have 
	$$
	\log\frac{d\widehat{\P}_{T}^{b+\varphi_T(b)u}}{d\widehat{\P}_{T}^b}(Y^{T,b})= uU_{T}(b) - \frac{u^2}{2}I_{T}(b),
	$$
	where
	\begin{align*}
	U_{T}(b)&:= - \frac{\varphi_T(b)}{\sigma^2}
	\biggl(Y_T^b - y_0 - a T + b \int_0^T Y_s^b ds - J_T\biggr)
	= - \frac{\varphi_T(b)}{\sigma} \int_0^T \sqrt{Y_s^b} dW_s , \\
	I_{T}(b)&:= \frac{(\varphi_T(b))^2}{\sigma^2} \int_0^T Y_s^b ds.
	\end{align*}
\end{corollary}
Observe that for any $t\in \R_{+}$, using equation \eqref{eq1}, we have
\begin{align*}
J_t=\sum_{s\in [0,t]}\Delta J_s=\sum_{s\in [0,t]}\Delta Y_s^b,
\end{align*}	
where the jump amplitude of $Y^b$ is defined as $\Delta Y_s^b:=Y_s^b-Y_{s-}^b$ for any $s\in \R_{++}$, $\Delta Y_0^b:=0$. Consequently, for all $t\in[0,T]$, $J_t$ is a measurable function of $(Y_u^b)_{u\in[0,T]}$. Therefore, $U_{T}(b)$ and $I_{T}(b)$ are measurable functions of $(Y_u^b)_{u\in[0,T]}$.

We now recall the asymptotic behaviors \eqref{ergo1}, \eqref{criticalconvergence1} and \eqref{supercriticalconvergence2} of $Y^b$ which follow from the joint Laplace transform of $Y_t^b$ and $\int_0^tY_s^bds$ in \cite[Theorem 3.1]{BBKP17} defined by
\begin{align}\label{Laplace}
	\widehat{\E}\Big[e^{uY_t^b+v\int_0^tY_s^bds}\Big]=\exp\Big\{\psi_{u,v}(t)y_0+\int_0^t\Big(a\psi_{u,v}(s)+\int_0^{\infty}\big(e^{z\psi_{u,v}(s)}-1\big)m(dz)\Big)ds\Big\},
\end{align}
for all $u, v\in\R_{-}$, $t\in\R_+$, where the function $\psi_{u,v}: \R_+\to \R_{-}$ takes the form
\begin{align*}
	\psi_{u,v}(t)=\begin{cases}
		\dfrac{u}{1-\frac{\sigma^2u}{2}t} &\textnormal{if}\quad v=0\quad\textnormal{and}\quad  b=0,\\
		\dfrac{u\gamma_v\cosh(\frac{\gamma_vt}{2})+(-ub+2v)\sinh(\frac{\gamma_vt}{2})}{\gamma_v\cosh(\frac{\gamma_vt}{2})+(-\sigma^2u+b)\sinh(\frac{\gamma_vt}{2})} &\textnormal{if}\quad v\in\R_{--}\quad\textnormal{or}\quad b\neq 0,
	\end{cases}
\end{align*}
with $\gamma_v:=\sqrt{b^2-2\sigma^2v}$. In particular, by \cite[Corollary 3.3]{BBKP17},
\begin{align}\label{Laplace2}
	&\widehat{\E}\left[e^{uY_t^b}\right]=\exp\left\{\psi_{u,0}(t)y_0+\int_0^t\left(a\psi_{u,0}(s)+\int_0^{\infty}\left(e^{z\psi_{u,0}(s)}-1\right)m(dz)\right)ds\right\}\notag\\
	&=\left(1-\frac{\sigma^2 u}{2}\zeta_b(t)\right)^{-\frac{2a}{\sigma^2}} \exp\left\{\psi_{u,0}(t)y_0\right\}\exp\left\{\int_0^t\int_0^{\infty}\left(e^{z\psi_{u,0}(s)}-1\right)m(dz)ds\right\},
\end{align}
for $t\in\R_+$, where the function $\psi_{u,0}: \R_+\to \R_{-}$ takes the form
\begin{align*}
	\psi_{u,0}(t)=\begin{cases}
		\frac{u}{1-\frac{\sigma^2u}{2}t} &\textnormal{if}\quad b=0,\\
		\frac{ue^{-bt}}{1-\frac{\sigma^2u}{2b}(1-e^{-bt})} &\textnormal{if}\quad b\neq 0,
	\end{cases}
\qquad\textnormal{and}\quad \zeta_b(t)=\begin{cases}
	t &\textnormal{if}\quad b=0,\\
	\frac{1}{b}(1-e^{-bt}) &\textnormal{if}\quad b\neq 0.
\end{cases}
\end{align*}
Indeed, following the same argument as in the proof of \cite[Proposition 1.2.4]{A15}, \eqref{Laplace2} is valid for $u<\frac{2}{\sigma^2\zeta_b(t)}$.
\subsubsection{Subcritical case}
\label{section_subcritical}

Let $b\in \R_{++}$. We recall the existence of a unique stationary distribution and the strong law of large numbers of the process $Y^b$.
\begin{proposition}\label{stationarydis} \cite[Theorems 2.4 and 5.1]{BBKP17}, \cite[Theorem 1.2 and Remark 5.1]{JKR17}
Assume $b \in \R_{++}$ and condition {\bf(A1)}. Then\\
	\textnormal{(i)}\; $Y^b$ has a unique stationary distribution denoted by $\pi_{b}(dy)$ which is given by
	$$
	\int_0^{\infty}e^{uy}\pi_{b}(dy)=\exp\left\{\int_u^0\dfrac{F(v)}{R(v)}dv\right\}=\exp\left\{\int_u^0\dfrac{av+\int_0^{\infty}(e^{vz}-1)m(dz)}{\frac{\sigma^2}{2}v^2-bv}dv\right\},
	$$
	for all $u\in \R_-$. Moreover, $\pi_{b}(dy)$ has a finite expectation given by
	\begin{align*}
	\int_{0}^{\infty}y\pi_{b}(dy)=\dfrac{1}{b}\left(a+\int_0^{\infty}zm(dz)\right)\in\R_+.
	\end{align*}	
	\textnormal{(ii)}\; As $t\to\infty$,
	\begin{align}\label{ergo1}
	&\dfrac{1}{t}\int_{0}^{t}Y_s^bds\overset{\widehat{\P}^{b}}{\longrightarrow}\int_{0}^{\infty}y\pi_{b}(dy)=\dfrac{1}{b}\left(a+\int_0^{\infty}zm(dz)\right)\in\R_+.
	\end{align}	
	\textnormal{(iii)}\; As $t\to\infty$,
	\begin{equation}\label{ergo3}
	\dfrac{1}{t}\int_{0}^{t}h(Y_s^b)ds\longrightarrow\int_{0}^{\infty}h(y)\pi_{b}(dy),\quad \widehat{\P}^{b}\textnormal{-a.s.},
	\end{equation}
for any $\pi_{b}(dy)$-integrable function $h: \R_{++} \to \R$. 
\end{proposition}

In \cite[Theorem 2.4]{BBKP17}, Barczy et {\it al.} show the strong law of large numbers \eqref{ergo3} for $Y^b$ which is obtained from the exponential ergodicity of $Y^b$ under two conditions on the L\'evy measure: {\bf(A1)} and $\int_{0}^{1}z\log(\frac{1}{z}) m(dz)<\infty$. More recently, in \cite[Theorem 1.2 and Remark 5.1]{JKR17}, Jin et {\it al.} show the validity of \eqref{ergo3} under the weaker condition: $\int_{1}^{\infty}\log zm(dz)<\infty$. This condition relaxes significantly the above two conditions used by Barczy et {\it al.} in \cite[Theorem 2.4]{BBKP17}. Notice that condition $\int_{1}^{\infty}\log zm(dz)<\infty$ is then satisfied under assumption {\bf(A1)}. 

For fixed $b_0\in \R_{++}$, we consider $Y^{T,b_0}=(Y_t^{b_0})_{t\in[0,T]}$ of $Y^{b_0}$ on $[0,T]$. The first result of this paper is the following LAN property in the subcritical case.
\begin{theorem}\label{Thm_subcritical}
	Assume $b_0 \in \R_{++}$, condition {\bf(A1)}, and $a \in \R_{++}$ or $m \ne 0$. Then, the LAN holds at $b_0$ with $\varphi_T(b_0)= \frac{1}{\sqrt{T}}$ and 
	\begin{equation}\label{conv_subcritical}
	(U_{T}(b_0),I_{T}(b_0))\overset{\mathcal{L}(\widehat{\P}^{b_0})}{\longrightarrow} (U(b_0),I(b_0)),
	\end{equation}
	as $T \to \infty$, where $U(b_0)=\mathcal{N}(0, I(b_0))$ is a centered $\R$-valued Gaussian random variable with variance
	$$
	I(b_0):= \frac{1}{\sigma^2b_0} \biggl(a + \int_0^\infty zm(dz)\biggr) \in \R_{++}. 
	$$
	\end{theorem}
\subsubsection{Critical case}
\label{section_critical}

Let $b=0$. We recall the asymptotic behavior of $Y^0=(Y_t^0)_{t\in\R_+}$.
\begin{proposition} \cite[Theorem 6.1]{BBKP17}
Assume $b =0$ and condition {\bf(A1)}. Then, as $t\to\infty$,
\begin{align}\label{criticalconvergence1}
\left(\dfrac{Y_t^0}{t},\dfrac{1}{t^2}\int_0^tY_s^0ds\right)\overset{\mathcal{L}(\widehat{\P}^{0})}{\longrightarrow} \left(\mathcal{Y}_1,\int_0^1\mathcal{Y}_sds\right).
\end{align}
Here, $\mathcal{Y}=(\mathcal{Y}_t)_{t\in\R_+}$ is the unique strong solution of a critical diffusion-type CIR process starting from $0$ defined by 
\begin{equation}\label{cirr0}
d\mathcal{Y}_t=\left(a+\int_0^{\infty}zm(dz)\right)dt+\sigma\sqrt{\mathcal{Y}_t}d\mathcal{W}_t,
\end{equation}
where $\mathcal{Y}_0=0$, and $\mathcal{W}=(\mathcal{W}_t)_{t\in\R_+}$ is a one-dimensional standard Brownian motion. Moreover, the Laplace transform of $(\mathcal{Y}_1,\int_0^1\mathcal{Y}_sds)$ takes the form
\begin{align*}
\widehat{\E}\left[e^{u\mathcal{Y}_1+v\int_0^1\mathcal{Y}_sds}\right]=\begin{cases}
\left(\cosh\frac{\gamma_v}{2}-\frac{\sigma^2 u}{\gamma_v}\sinh\frac{\gamma_v}{2}\right)^{-\frac{2}{\sigma^2}\left(a+\int_0^{\infty}zm(dz)\right)}\quad & \textnormal{if}\ \ v\in \R_{--},\\
\left(1-\frac{\sigma^2 u}{2}\right)^{-\frac{2}{\sigma^2}\left(a+\int_0^{\infty}zm(dz)\right)}\quad & \textnormal{if}\ \ v=0,
\end{cases}
\end{align*}
for all $u, v\in \R_{-}$, where $\gamma_v:=\sqrt{-2\sigma^2v}$.
\end{proposition}
Now consider the diffusion-type CIR process $\mathcal{Y}=(\mathcal{Y}_t)_{t\in\R_+}$ starting from $0$ defined by 
\begin{equation}\label{eq2}
d\mathcal{Y}_t=\left(a+\int_0^{\infty}zm(dz)-b\mathcal{Y}_t\right)dt+\sigma\sqrt{\mathcal{Y}_t}d\mathcal{W}_t,
\end{equation}
where $\mathcal{Y}_0=0$ and $b\in\R$. We denote by $\P_{\mathcal{Y}}^{b}$ the probability measure induced by the CIR process $\mathcal{Y}$ solution to equation \eqref{eq2} on the canonical space $(C(\R_{+},\R),\mathcal{B}(C(\R_{+},\R))$ endowed with the natural filtration $\{\mathcal{G}_t\}_{t\in\R_+}$ generated by the Brownian motion $\mathcal{W}$. Here $C(\R_{+},\R)$ denotes the set of $\R$-valued continuous functions defined on $\R_{+}$, and $\mathcal{B}(C(\R_{+},\R))$ is its Borel $\sigma$-algebra. For any $T\in\R_{++}$, let $\P_{\mathcal{Y},T}^{b}$ be the restriction of $\P_{\mathcal{Y}}^{b}$ on $\mathcal{G}_T$. As a consequence of \cite[Proposition 4.1]{BBKP17}, under condition {\bf(A1)}, for any $b\in\R$, the probability measures $\P_{\mathcal{Y},T}^{b}$ and $\P_{\mathcal{Y},T}^{0}$ are absolutely continuous w.r.t. each other and 
\begin{equation}\label{RN}
\dfrac{d\P_{\mathcal{Y},T}^{b}}{d\P_{\mathcal{Y},T}^{0}}\left((\mathcal{Y}_s)_{s\in[0,T]}\right)=\exp\left\{-\dfrac{b}{\sigma}\int_0^T\sqrt{\mathcal{Y}_s}d\mathcal{W}_s-\dfrac{b^2}{2\sigma^2}\int_0^T\mathcal{Y}_sds\right\},
\end{equation}
where $\mathcal{Y}=(\mathcal{Y}_t)_{t\in\R_+}$ is the CIR process solution to equation \eqref{eq2} corresponding to the parameter $b=0$ (i.e., solution to equation \eqref{cirr0}).	As a consequence of Theorem 3.4 in Chapter III of Jacod and Shiryaev \cite{JS03}, the process $\Bigl(\frac{d\P_{\mathcal{Y},T}^{b}}{d\P_{\mathcal{Y},T}^{0}}\left((\mathcal{Y}_s)_{s\in[0,T]}\right)\Bigr)_{T\in\R_+}$
is a martingale w.r.t. the filtration $(\mathcal{G}_T)_{T\in\R_+}$.

For fixed $b_0=0$, we consider the continuous observation $Y^{T,0}=(Y_t^{0})_{t\in[0,T]}$ of $Y^{0}$. The second result of this paper is the following LAQ property in the critical case.
\begin{theorem}\label{Thm_critical}
	Assume $b_0 = 0$, $y_0+a \in \R_{++}$ and condition {\bf(A1)}. Then, the LAQ holds at $b_0 = 0$ with $\varphi_T(0)= \frac{1}{T}$ and 
	\begin{equation}\label{conv_critical}
	(U_{T}(0),I_{T}(0))\overset{\mathcal{L}(\widehat{\P}^{0})}{\longrightarrow}(U(0), I(0)),
	\end{equation}
	as $T\to\infty$, and
	\begin{align}\label{conv_critical2}
	\widehat{\E}\Bigl[e^{uU(0)-\frac{u^2}{2}I(0)}\Bigr]=1,
	\end{align}
	where
	$$
	U(0)
	:= \frac{a + \int_0^\infty  z m(dz) - \mathcal{Y}_1}{\sigma^2}
	= - \frac{1}{\sigma} \int_0^1 \sqrt{\mathcal{Y}_s} d \mathcal{W}_s, \qquad
	I(0):= \frac{1}{\sigma^2} \int_0^1 \mathcal{Y}_s ds ,
	$$
	and $(\mathcal{Y}_t)_{t\in\R_+}$ is the unique strong solution of the SDE \eqref{cirr0} with initial condition  $\mathcal{Y}_0 = 0$. 
\end{theorem}

\subsubsection{Supercritical case}
\label{section_supercritical}

Let $b\in \R_{--}$. We recall the asymptotic behavior of $Y^b$.
\begin{proposition}\cite[Theorem 7.1]{BBKP17}\label{superbehavior} Assume $b \in \R_{--}$ and condition {\bf(A1)}. Then, there exists a random variable $V$ with $\widehat{\P}(V\in \R_+)=1$ such that as $t\to\infty$,
\begin{align}
&e^{bt}Y_t^b\longrightarrow V,\quad \widehat{\P}^{b}\textnormal{-a.s.}\label{supercriticalconvergence1}\\
&e^{bt}\int_0^tY_s^bds\longrightarrow-\dfrac{V}{b},\quad \widehat{\P}^{b}\textnormal{-a.s.}\label{supercriticalconvergence2}
\end{align}
Moreover, the Laplace transform of $V$ takes the form
\begin{equation}\label{Vlimit}
\begin{split}
\widehat{\E}\left[e^{uV}\right]&=\exp\left\{\dfrac{uy_0}{1+\frac{\sigma^2u}{2b}}\right\}\left(1+\dfrac{\sigma^2u}{2b}\right)^{-\frac{2a}{\sigma^2}}\\
&\qquad\times\exp\left\{\int_0^{\infty}\left(\int_0^{\infty}\left(\exp\left\{\dfrac{zue^{by}}{1+\frac{\sigma^2u}{2b}e^{by}}\right\}-1\right)m(dz)\right)dy\right\},
\end{split}
\end{equation}
for all $u\in \R_-$, and consequently $V\overset{\mathcal{L}}{=}\widetilde{\mathcal{V}}+\widetilde{\widetilde{\mathcal{V}}}$, where $\widetilde{\mathcal{V}}$ and $\widetilde{\widetilde{\mathcal{V}}}$ are independent random variables such that $e^{bt}\widetilde{\mathcal{Y}}_t\overset{\textnormal{a.s.}}{\longrightarrow}\widetilde{\mathcal{V}}$ and $e^{bt}\widetilde{\widetilde{\mathcal{Y}}}_t\overset{\textnormal{a.s.}}{\longrightarrow}\widetilde{\widetilde{\mathcal{V}}}$ as $t\to\infty$, where $\widetilde{\mathcal{Y}}=(\widetilde{\mathcal{Y}}_t)_{t\in\R_+}$ and $\widetilde{\widetilde{\mathcal{Y}}}=(\widetilde{\widetilde{\mathcal{Y}}}_t)_{t\in\R_+}$ are the pathwise unique strong solutions of the supercritical CIR models
$$
d\widetilde{\mathcal{Y}}_t=\left(a-b\widetilde{\mathcal{Y}}_t\right)dt +\sigma\sqrt{\widetilde{\mathcal{Y}}_t}d\widetilde{\mathcal{W}}_t, \quad \textnormal{with} \ \widetilde{\mathcal{Y}}_0=y_0,
$$
and 
$$
d\widetilde{\widetilde{\mathcal{Y}}}_t=-b\widetilde{\widetilde{\mathcal{Y}}}_tdt +\sigma\sqrt{\widetilde{\widetilde{\mathcal{Y}}}_t}d\widetilde{\widetilde{\mathcal{W}}}_t+dJ_t, \quad \textnormal{with} \ \widetilde{\widetilde{\mathcal{Y}}}_0=0,
$$
respectively, where $\widetilde{\mathcal{W}}=(\widetilde{\mathcal{W}}_t)_{t\in\R_+}$ and $\widetilde{\widetilde{\mathcal{W}}}=(\widetilde{\widetilde{\mathcal{W}}}_t)_{t\in\R_+}$ are independent one-dimensional standard Brownian motions. Furthermore, $\widetilde{\mathcal{V}}\overset{\mathcal{L}}{=}\mathcal{Z}_{-\frac{1}{b}}$, where $\mathcal{Z}=(\mathcal{Z}_t)_{t\in\R_+}$ is the pathwise unique strong solution of the critical CIR model
$$
d\mathcal{Z}_t=adt+\sigma\sqrt{\mathcal{Z}_t}d\mathcal{W}_t, \quad \textnormal{with} \ \mathcal{Z}_0=y_0,
$$
where $\mathcal{W}=(\mathcal{W}_t)_{t\in\R_+}$ is a one-dimensional standard Brownian motion.

If, in addition, $a \in \R_{++}$, then $\widehat{\P}(V \in \R_{++}) = 1$.
\end{proposition}
For fixed $b_0\in \R_{--}$, we consider  $Y^{T,b_0}=(Y_t^{b_0})_{t\in[0,T]}$ of $Y^{b_0}$. The third result of this paper is the following LAMN property in the supercritical case.
\begin{theorem}\label{Thm_supercritical}
	Assume $b_0 \in \R_{--}$, $a \in \R_{++}$ and condition {\bf(A1)}. Then, the LAMN holds at $b_0$ with  $\varphi_T(b_0)=e^{\frac{b_0T}{2}}$ and 
	\begin{equation}\label{conv_supercritical}
	(U_{T}(b_0),I_{T}(b_0))\overset{\mathcal{L}(\widehat{\P}^{b_0})}{\longrightarrow} (U(b_0), I(b_0)),
	\end{equation}
	as $T\to\infty$, where $U(b_0):= \sqrt{I(b_0)}Z$ and the asymptotic random Fisher information is given by
	$$
	 I(b_0):= - \frac{V}{\sigma^2b_0}.
	$$
	Here, $V$ is the positive random variable whose Laplace transform is given by \eqref{Vlimit}, and $Z$ is a standard normally distributed random variable, independent of $V$. 
\end{theorem}

\subsection{Discrete observations}
Let us first add two following assumptions on equation \eqref{eq1} we shall work with. 
\begin{list}{labelitemi}{\leftmargin=1cm}
	\item[\bf(A2)] For any $p>1$, \; $\int_{1}^{\infty}z^pm(dz)<\infty$.
	\vskip 7pt
	\item[\bf(A3)] \; \; $\dfrac{a}{\sigma^2}>\dfrac{15+\sqrt{185}}{4}$.
\end{list}
Note that conditions {\bf(A1)} and {\bf(A2)} imply that $\int_{0}^{\infty}z^pm(dz)<\infty$ for any $p\geq 1$.

In order to avoid confusion with the observed process $Y^{b}$ driven by the Brownian motion $W$ and the subordinator $J$, we introduce a new probability space $(\widetilde{\Omega}, \widetilde{\mathcal{F}}, \widetilde{\P})$ as the extension of $(\widehat{\Omega}, \widehat{\mathcal{F}},\{\widehat{\mathcal{F}}_t\}_{t \geq 0}, \widehat{\P})$ on which we define an independent copy $X^{b}=(X_t^{b})_{t \geq 0}$ of $Y^{b}$ solution to 
\begin{equation}\label{c2eq1rajoute}
\begin{split}
X_t^{b}=y_0+\int_0^t (a-bX_s^{b})ds +\sigma\int_0^t\sqrt{X_s^{b}}dB_s+\widetilde{J}_t,
\end{split}
\end{equation}
where $B=(B_t)_{t\in\R_+}$ is a standard Brownian motion and $\widetilde{J}=(\widetilde{J}_t)_{t\in\R_+}$ is a subordinator with L\'evy measure $m$. Let  $M(dt,dz)$ be the Poisson random measure  with intensity measure $m(dz)dt$ associated with the subordinator $\widetilde{J}=(\widetilde{J}_t)_{t\in\R_+}$, i.e., $\widetilde{J}_t=\int_0^t\int_{0}^{\infty}zM(ds,dz)$. 
	
Now, let $\widetilde{\E}$ denote the expectation w.r.t. $\widetilde{\P}$. We denote by $\widetilde{\P}^{b}$ the probability measure induced by the process $X^b$ on the canonical space $(D(\R_{+},\R),\mathcal{B}(D(\R_{+},\R)))$ endowed with the natural filtration $\{\widetilde{\mathcal{F}}_t\}_{t\in\R_+}$. For all $t\geq 0$ and $x\in\R_{++}$, we denote by $\widetilde{\P}_{t,x}^{b}$ the probability law of $X^{b}$ starting at $x$ at time $t$, i.e., $\widetilde{\P}_{t,x}^{b}(A)=\widetilde{\E}[{\bf 1}_{A}\vert X_{t}^{b}=x]$ for all $A\in \widetilde{\mathcal{F}}$, and denote by $\widetilde{\E}_{t,x}^{b}$ the expectation w.r.t. $\widetilde{\P}_{t,x}^{b}$. That is, for all $\widetilde{\mathcal{F}}$-measurable random variables $V$, we have that $\widetilde{\E}_{t,x}^{b}[V]=\widetilde{\E}[V\vert X_{t}^{b}=x]$.

Given the process $Y^{b}=(Y_t^{b})_{t\in\R_+}$ and $n\geq 1$, we consider a discrete observation scheme of $Y^{b}$ at deterministic and equidistant times $t_k=k \Delta_n$ with $k \in \{0,\ldots,n\}$, which is denoted by $Y^{n,b}=(Y_{t_0}^{b}, Y_{t_1}^{b},\ldots,Y_{t_n}^{b})$ and satisfies the high-frequency and infinite horizon conditions. That is, $\Delta_n\rightarrow 0$ and $n\Delta_n\rightarrow\infty$ as $n\rightarrow\infty$. Let $\mathcal{F}_n=\sigma(Y_{t_0}^{b}, Y_{t_1}^{b},\ldots,Y_{t_n}^{b})$ be the $\sigma$-field of observable events, which are induced by the random vector $Y^{n,b}=(Y_{t_0}^{b}, Y_{t_1}^{b},\ldots,Y_{t_n}^{b})$.  We denote by $\P^{b}_n$ the probability law of the random vector $Y^{n,b}$.

 \begin{proposition}\label{measurabledecom}
For any $b_0 \in \R$, let $Y^{n,b_0}=(Y_{t_0}^{b_0}, Y_{t_1}^{b_0},\ldots,Y_{t_n}^{b_0})$ be the discrete observation of $Y^{b_0}$, for $\varphi_{n\Delta_n}(b_0)\in\R$ and $u\in \R$, we have 
\begin{align*}
\log\dfrac{d\P_n^{b_0+\varphi_{n\Delta_n}(b_0)u}}{d\P^{b_0}_n}(Y^{n,b_0})=uV_n(b_0)-\dfrac{u^2}{2}T_n(b_0)+o_{\widehat{\P}}(1),
\end{align*}
where $b(\ell):=b_0+\ell \varphi_{n\Delta_n}(b_0)u$ and
\begin{equation}\begin{split}\label{firstdeco}
V_n(b_0):&=-\frac{\varphi_{n\Delta_n}(b_0)}{\sigma^2}\sum_{k=0}^{n-1}\left(Y_{t_{k+1}}^{b_0}-Y_{t_{k}}^{b_0}-(a-b_0Y_{t_k}^{b_0})\Delta_n+U_{k,n}(b_0)\right),\\
U_{k,n}(b_0):&=-\int_0^1\widetilde{\E}_{t_k,Y_{t_{k}}^{b_0}}^{b(\ell)}\Big[X_{t_{k+1}}^{b(\ell)}-X_{t_{k}}^{b(\ell)}-(a-b(\ell)X_{t_k}^{b(\ell)})\Delta_n\\
&\qquad-\sigma\sqrt{X_{t_k}^{b(\ell)}}(B_{t_{k+1}}-B_{t_{k}})\vert X_{t_{k+1}}^{b(\ell)}=Y_{t_{k+1}}^{b_0}\Big]d\ell,\\
T_n(b_0):&=\dfrac{\Delta_n}{\sigma^2 }(\varphi_{n\Delta_n}(b_0))^2\sum_{k=0}^{n-1}Y_{t_k}^{b_0}.
\end{split}
\end{equation}
\end{proposition}
Observe that $V_n(b_0)$ and $T_n(b_0)$ are $\mathcal{F}_n$-measurable random variables.
\subsubsection{Subcritical case}
For fixed $b_0\in \R_{++}$, the next result of this paper is the following LAN property.
\begin{color}{black}
\begin{theorem}\label{theorem1} Assume $b_0 \in \R_{++}$, {\bf(A1)}-{\bf(A3)}, and $a \in \R_{++}$ or $m \ne 0$. Then, the LAN property holds at $b_0$ with rate of convergence $\varphi_{n\Delta_n}(b_0)=\frac{1}{\sqrt{n\Delta_n}}$ and 
	\begin{equation*}
	(V_n(b_0),T_n(b_0))\overset{\mathcal{L}(\widehat{\P}^{b_0})}{\longrightarrow} (U(b_0),I(b_0)),
	\end{equation*}
	as $n \to \infty$, where $U(b_0)=\mathcal{N}(0, I(b_0))$ and the asymptotic Fisher information is given by
	$$I(b_0)=\dfrac{1}{\sigma^2b_0}\left(a+\int_0^{\infty}zm(dz)\right).$$
\end{theorem}
\end{color}
\subsubsection{Critical case}
For fixed $b_0=0$, the next result of this paper is the following LAQ property.
\begin{theorem}\label{theorem2} 
	Assume $b_0 = 0$, $y_0+a \in \R_{++}$, {\bf(A1)}-{\bf(A3)} and that $n\Delta_n^{\delta}\to 0$ as $n\to\infty$ where $\delta>1$ is arbitrarily large. Then, the LAQ property holds at $b_0=0$ with rate of convergence $\varphi_{n\Delta_n}(0)=\frac{1}{n\Delta_n}$ and 
	\begin{equation*}
	(V_n(0),T_n(0))\overset{\mathcal{L}(\widehat{\P}^{0})}{\longrightarrow}(U(0), I(0)),
	\end{equation*}
	as $n\to\infty$, and
	$$\widehat{\E}[e^{uU(0)-\frac{u^2}{2}I(0)}]=1,$$
	where 
	$$
	U(0)=\dfrac{a+\int_0^{\infty}zm(dz)-\mathcal{Y}_1}{\sigma^2}	= - \frac{1}{\sigma} \int_0^1 \sqrt{\mathcal{Y}_s} d \mathcal{W}_s,\qquad I(0)=\frac{1}{\sigma^2}\int_0^1\mathcal{Y}_sds.
	$$
	Here, $\mathcal{Y}=(\mathcal{Y}_t)_{t\in\R_+}$ is a critical diffusion-type CIR process starting from $0$ defined by \eqref{cirr0}.
\end{theorem}

\subsubsection{Supercritical case}

For fixed $b_0\in \R_{--}$, the last result of this paper is the following LAMN property.
\begin{theorem}\label{theorem3}
	Assume $b_0 \in \R_{--}$, $a \in \R_{++}$, {\bf(A1)}-{\bf(A3)} and that $\Delta_n^{2-\varepsilon}e^{-b_0n\Delta_n}\to 0$ as $n\to\infty$ where $\varepsilon>0$ is arbitrarily small. Then, the LAMN property holds at $b_0$ with rate of convergence $\varphi_{n\Delta_n}(b_0)=e^{b_0\frac{n\Delta_n}{2}}$ and 
	\begin{equation*}
	(V_n(b_0),T_n(b_0))\overset{\mathcal{L}(\widehat{\P}^{b_0})}{\longrightarrow} (U(b_0), I(b_0)),
	\end{equation*}
	as $n\to\infty$, 
	where $U(b_0)=\sqrt{I(b_0)}Z$ and the asymptotic random Fisher information $I(b_0)$ is given by
	$$
	I(b_0)=-\frac{V}{\sigma^2 b_0}.
	$$
	Here, $V$ is the positive random variable whose Laplace transform is given by \eqref{Vlimit}, and $Z$ is a standard normal random variable, independent of $V$.
\end{theorem}
\begin{remark}
Theorems \ref{theorem1}, \ref{theorem2} and \ref{theorem3} extend the results in Theorems 1, 2 and 3 of \cite{BKT17} when $b$ is the only unknown parameter in the drift coefficient. However, condition {\bf(A3)} in \cite{BKT17} is less restrictive than the one in this paper (see Remark \ref{A3diffusionjumps}).
\end{remark}
\begin{remark}
	The technical condition {\bf(A3)} required in Theorems \ref{theorem1}, \ref{theorem2} and \ref{theorem3} comes from the techniques of Malliavin calculus used in this paper. The lower bound $\frac{15+\sqrt{185}}{4}$ appearing in {\bf(A3)} is fixed in an optimal way in the sense that the H\"older weights used to control the remainder terms in Lemma \ref{lemma1} are carefully chosen in order that the final condition {\bf (A3)}: $\frac{a}{\sigma^2}>\frac{15+\sqrt{185}}{4}$ still remains the same as the one required for the representation of the score function in Proposition \ref{c2prop1}.
\end{remark}
\begin{remark} Combining our main results Theorems \ref{Thm_subcritical} and \ref{Thm_supercritical} together with Theorems $5.2$ and $7.3$ in \cite{BBKP17}, the MLE $\widehat{b}_T$ of $b$ in \cite[Proposition 4.2]{BBKP17} based on continuous time observations $(Y_t^b)_{t\in[0,T]}$ defined by
\begin{align*}
\widehat{b}_T=-\dfrac{Y_T^b - y_0 - a T - J_T}{\int_0^T Y_s^b  d s},
\end{align*}
satisfies (see page 1152 and page 1159 of \cite{BBKP17})
\begin{align*}
\varphi_T(b_0)^{-1}\left(\widehat{b}_T-b_0\right)=G_T(b_0)^{-1} \Lambda_T(b_0),
\end{align*}
where
\begin{align*}
\Lambda_T(b_0)=-\dfrac{1}{\sigma}\varphi_T(b_0) \int_0^T\sqrt{Y_t^{b_0}}dW_t,\;\;\textnormal{and}\;
G_T(b_0) =\dfrac{1}{\sigma^2}\varphi_T^2(b_0)\int_0^TY_t^{b_0}dt.
\end{align*}
Therefore, due to (2.35) of Remark 3 on page 160 of \cite{L92}, $\widehat{b}_T$ is regular and asymptotically efficient in the sense of H{\'a}jek-Le Cam convolution theorem.
\end{remark}
\begin{example} The following subordinators whose L\'evy measure can be finite or infinite satisfy conditions {\bf(A1)} and {\bf(A2)}.\\
\noindent\textnormal{1)} $J$ is a Poisson process.\\
\noindent\textnormal{2)} $J$ is a compound Poisson process with exponentially distributed jump sizes. That is, $m(dz)=C\lambda e^{-\lambda z}{\bf 1}_{(0,\infty)}(z)dz$ for some constants $C\in (0,\infty)$ and $\lambda \in (0,\infty)$.\\
\noindent\textnormal{3)} $J$ is a Gamma process with L\'evy measure $m(dz)=\gamma z^{-1} e^{-\lambda z}{\bf 1}_{(0,\infty)}(z)dz$ where $\gamma$ and $\lambda$ are positive constants.\\
\noindent\textnormal{4)} $J$ is a subordinator whose L\'evy measure is given by the gamma probability distribution. That is, $m(dz)=\frac{\lambda^{\alpha}}{\Gamma(\alpha)} z^{\alpha-1} e^{-\lambda z}{\bf 1}_{(0,\infty)}(z)dz$ where $\alpha\in(-1,\infty)$ and $\lambda$ is a positive constant.\\
\noindent\textnormal{5)} $J$ is an inverse Gaussian process with $m(dz)=\frac{\delta}{\sqrt{2\pi z^3}} e^{-\frac{\gamma^2 z}{2}}{\bf 1}_{(0,\infty)}(z)dz$ for a positive constant $\delta$.
\end{example}
As usual, positive constants will be denoted by $C$ and will always be independent of time and $\Delta_n$. They may change of value from one line to the next.
\section{Proof of main results}
This section is devoted to the proof of the main results.
\subsection{Proof of Theorem \ref{Thm_subcritical}}
\label{con_subcritical}
\begin{proof}
	By Proposition \ref{stationarydis} \textnormal{(i)}, $(Y_t^{b_0})_{t\in\R_+}$ has a unique stationary distribution $\pi_{b_0}(dy)$ with
	$\int_0^\infty y \pi_{b_0}(dy)
	= \bigl(a + \int_0^\infty z m(dz)\bigr) \frac{1}{b_0} \in \R_{++}$.
	By Proposition \ref{stationarydis} \textnormal{(ii)}, we have
	$\frac{1}{T} \int_0^T Y_s^{b_0} ds\overset{\widehat{\P}^{b_0}}{\longrightarrow} \int_0^\infty y \pi_{b_0}(dy)$ as $T \to \infty$. Thus, as $T \to \infty$,
	\begin{equation}\label{conv_subcritical_J}
	I_{T}(b_0)= \frac{1}{\sigma^2T} \int_0^T Y_s^{b_0} ds\overset{\widehat{\P}^{b_0}}{\longrightarrow}
	\frac{1}{\sigma^2} \int_0^\infty y \pi_{b_0}(d y)= \frac{1}{\sigma^2b_0} \biggl(a + \int_0^\infty zm(dz)\biggr)	= I(b_0).
	\end{equation}
	The quadratic variation process of the square integrable martingale
	$\bigl(\int_0^T \sqrt{Y_s^{b_0}} dW_s\bigr)_{T\in\R_+}$  takes the form
	$\bigl(\int_0^T Y_s^{b_0} ds\bigr)_{T\in\R_+}$. Hence, applying Lemma \ref{THM_Zanten} with $\eta := \left(\int_0^\infty y \, \pi_{b_0}(dy)\right)^{1/2}$, we	obtain that as $T \to \infty$,
	\begin{align*}
	U_{T}(b_0)
	= - \frac{1}{\sigma\sqrt{T}} \int_0^T \sqrt{Y_s^{b_0}} dW_s&\overset{\mathcal{L}(\widehat{\P}^{b_0})}{\longrightarrow}
	- \frac{1}{\sigma} \left(\int_0^\infty y\pi_{b_0}(dy)\right)^{1/2}\mathcal{N}(0,1)\\
	&\qquad	= \mathcal{N}\biggl(0, \frac{1}{\sigma^2} \int_0^\infty y\pi_{b_0}(d y)\biggr)
	=\mathcal{N}(0, I(b_0))=U(b_0).
	\end{align*}
	Consequently, by \eqref{conv_subcritical_J} and Theorem 2.7 (v) of van der Vaart
	\cite{Vaart}, we obtain \eqref{conv_subcritical}. This, combined with Corollary \ref{RN_Cor}, finishes the proof.
\end{proof}
\subsection{Proof of Theorem \ref{Thm_critical}}
\label{con_critical}
\begin{proof}
	By strong law of large numbers for the L\'evy process $(J_t)_{t\in\R_+}$ (see, e.g., Kyprianou \cite[Exercise 7.2]{K14}), we have
	\begin{align}\label{help_Levy_SLLN}
	\widehat{\P}\left(\lim_{T\to\infty} \frac{J_T}{T} = \widehat{\E}[J_1] = \int_0^\infty z m(dz)\right)
	= 1 .
	\end{align}
	Hence, using \eqref{criticalconvergence1}, we obtain that as $T \to \infty$,
	\begin{align*}
	(U_{T}(0), I_{T}(0))
	&= \biggl(-\frac{1}{\sigma^2T} (Y_T^0 - y_0 - a T - J_T),
	\frac{1}{\sigma^2T^2} \int_0^T Y_s^0 d s\biggr) \\
	&\overset{\mathcal{L}(\widehat{\P}^{0})}{\longrightarrow}\biggl(-\frac{1}{\sigma^2} \biggl(\mathcal{Y}_1- a - \int_0^\infty z m(dz)\biggr),
	\frac{1}{\sigma^2} \int_0^1 \mathcal{Y}_s ds\biggr)
	= (U(0), I(0)),
	\end{align*}
	which deduces \eqref{conv_critical}. Proposition \ref{Pro_jump_CIR} implies
	$\widehat{\P}(I(0) \in \R_{++}) = \widehat{\P}(\int_0^1 \mathcal{Y}_s ds \in \R_{++}) = 1$.
	
	Finally, using equation \eqref{cirr0} and \eqref{RN}, we have that
	\begin{equation}\label{LAQpr3}
	\begin{split}
	&\widehat{\E}\left[e^{uU(0)-\frac{u^2}{2}I(0)}\right]=\widehat{\E}\left[\exp\left\{\dfrac{u}{\sigma^2}\left(a+\int_0^{\infty}zm(dz)-\mathcal{Y}_1\right)-\dfrac{u^2}{2\sigma^2}\int_0^1\mathcal{Y}_sds\right\}\right]\\
	&=\widehat{\E}\left[\exp\left\{-\dfrac{u}{\sigma}\int_0^1\sqrt{\mathcal{Y}_s}d\mathcal{W}_s-\dfrac{u^2}{2\sigma^2}\int_0^1\mathcal{Y}_sds\right\}\right]=\widehat{\E}\left[\dfrac{d\P_{\mathcal{Y},1}^{u}}{d\P_{\mathcal{Y},1}^{0}}\left((\mathcal{Y}_s)_{s\in[0,1]}\right)\right].
	\end{split}
	\end{equation}
	Then, using the fact that the Radon-Nikodym derivative process $	\Bigl(\frac{d\P_{\mathcal{Y},T}^{u}}{d\P_{\mathcal{Y},T}^{0}}\left((\mathcal{Y}_s)_{s\in[0,T]}\right)\Bigr)_{T\in\R_+}$
	is a martingale w.r.t. the filtration $(\mathcal{G}_T)_{T\in\R_+}$, we get that
	$$
	\widehat{\E}\left[\dfrac{d\P_{\mathcal{Y},1}^{u}}{d\P_{\mathcal{Y},1}^{0}}\left((\mathcal{Y}_s)_{s\in[0,1]}\right)\right]=\widehat{\E}\left[\dfrac{d\P_{\mathcal{Y},0}^{u}}{d\P_{\mathcal{Y},0}^{0}}\left(\mathcal{Y}_0\right)\right]=1.
	$$
	This, together with \eqref{LAQpr3}, concludes \eqref{conv_critical2}. Thus, the result follows thanks to Corollary \ref{RN_Cor}.
\end{proof}
\subsection{Proof of Theorem \ref{Thm_supercritical}}
\label{con_supercritical}
\begin{proof}
	By \eqref{supercriticalconvergence2},
	$e^{b_0T} \int_0^T Y_s^{b_0} ds\to-\frac{V}{b_0},\; \widehat{\P}^{b_0}\textnormal{-a.s.}$ as $T \to \infty$, and using Lemma \ref{THM_Zanten} with $\eta:= \sqrt{-\frac{V}{b_0}}$, we obtain that as $T \to \infty$,
	$$
	(U_T(b_0), I_T(b_0)) 
	= \biggl(-\frac{e^{b_0T/2}}{\sigma} \int_0^T \sqrt{Y_s^{b_0}} dW_s,
	\frac{e^{b_0T}}{\sigma^2} \int_0^T Y_s^{b_0}ds\biggr)\overset{\mathcal{L}(\widehat{\P}^{b_0})}{\longrightarrow} (U(b_0), I(b_0)),
	$$
	which deduces \eqref{conv_supercritical}. By Proposition \ref{superbehavior}, we have $\widehat{\P}(V \in \R_{++}) = 1$ since $a \in \R_{++}$. Thus, the result follows thanks to Corollary \ref{RN_Cor}.
\end{proof}
Let us now treat the case of discrete observations. For this, we first give the moment estimates for the jump-type CIR process \eqref{eq1}. 
\begin{lemma}\label{moment} Assume condition {\bf(A1)}. \\
	\noindent\textnormal{(i)} For all $0<p<\frac{2a}{\sigma^2}$ and $b>0$, we have  
	$$\sup_{t\geq 0}\widehat{\E}^{b}\Big[\frac{1}{(Y_t^{b})^{p}}\Big]<\infty.$$
	Assume further {\bf(A2)}. Then, for all $p>0$, $b>0$ and $t>0$, we have $\widehat{\E}^{b}[(Y_t^{b})^{p}]<\infty$.\\
	\noindent\textnormal{(ii)} For all $0<p<\frac{2a}{\sigma^2}$ and  $b=0$, we have 
	$$\sup_{0\leq t\leq 1}\widehat{\E}^{0}\Big[\frac{1}{(Y_t^{0})^{p}}\Big]<\infty\quad\textnormal{and}\quad \sup_{t\geq 1}\frac{\widehat{\E}^{0}[\frac{1}{(Y_t^{0})^{p}}]}{t^{^{-p}}}<\infty.$$
	\noindent\textnormal{(iii)} For all $0<p<\frac{2a}{\sigma^2}$ and $b<0$, we have 
	$$
	\sup_{0\leq t\leq 1}\widehat{\E}^{b}\Big[\frac{1}{(Y_t^{b})^{p}}\Big]<\infty\quad\textnormal{and}\quad\sup_{t\geq 1}\frac{\widehat{\E}^{b}[\frac{1}{(Y_t^{b})^{p}}]}{t^{^{-p}}}<\infty.$$
	\noindent\textnormal{(iv)} Assume further {\bf(A2)}. Then, for all $p>0$,
	\begin{equation*}\begin{split}
	&\sup_{t\geq 0}\dfrac{\widehat{\E}^{b}[(Y_t^{b})^{p}]}{(1+t)^{p}}<\infty,\quad\textnormal{if}\; b=0,\\
	&\sup_{t\geq 0}\widehat{\E}^{b}[(e^{bt}Y_t^{b})^{p}]<\infty,\quad\textnormal{if}\; b<0.
	\end{split}
	\end{equation*}
\end{lemma}
To control the negligible terms in the critical and supercritical cases, the following lemma will be needed.
	\begin{lemma}\label{auxi} Assume conditions {\bf(A1)} and {\bf(A2)}. \\ 
		\noindent\textnormal{(i)} Let $b_0= 0$. Then, for $n$ large enough there exist two positive constants $C_1, C_2$ such that
		\begin{align*}
		&\Delta_n\sum_{k=0}^{n-1}\widehat{\E}^{0}\Big[\frac{1}{(Y_{t_k}^{0})^{p}}\Big]\leq
		\sum_{k=0}^{n-1}\int_{t_k}^{t_{k+1}}\sup_{u\in[t_k,s]}\widehat{\E}^{0}\Big[\frac{1}{(Y_u^{0})^{p}}\Big]ds\leq C_1+C_2f(n\Delta_n),
		\end{align*}
		where $f(x)=\log x$ if $p=1$, $f(x)=\frac{1}{x^{p-1}}$ if $p\neq 1$ and $p<\frac{2a}{\sigma^2}$.\\
		\noindent\textnormal{(ii)} Let $b_0< 0$. Then, for $n$ large enough there exist two positive constants $C_1, C_2$ such that
		\begin{align*}
		&\Delta_n\sum_{k=0}^{n-1}\widehat{\E}^{b_0}\Big[\frac{1}{(Y_{t_k}^{b_0})^{p}}\Big]\leq
		\sum_{k=0}^{n-1}\int_{t_k}^{t_{k+1}}\sup_{u\in[t_k,s]}\widehat{\E}^{b_0}\Big[\frac{1}{(Y_u^{b_0})^{p}}\Big]ds\leq C_1+C_2f(n\Delta_n),
		\end{align*}
		where $f(x)=\log x$ if $p=1$, $f(x)=\frac{1}{x^{p-1}}$ if $p\neq 1$ and $0<p<\frac{2a}{\sigma^2}$, and $f(x)=e^{pb_0x}$ if $p< 0$.
	\end{lemma}
Next, we prove the existence and the smoothness of the density by using the affine structure of the jump-type CIR process and the inverse Fourier transform. 
\begin{proposition}\label{smoothdensity}
	Assume condition {\bf(A1)} and $2a> \sigma^2$. Then for any $t>0$ and $b\in \mathbb R$, the law of $Y_t^{b}$ admits a strictly positive density function $p^{b}(t,y_0,y)$. Moreover, assume further $\frac{a}{\sigma^2}>1+\frac{\sqrt{2}}{2}$, then  $p^{b}(t,y_0,y)$ is of class $C^1$ w.r.t. $b$ for all $t>0$ and $b\in \mathbb R$.
\end{proposition}
We denote by $p_n(\cdot;b)$ the density of the random vector $Y^{n,b}=(Y_{t_0}^{b}, Y_{t_1}^{b},\ldots,Y_{t_n}^{b})$. By Proposition \ref{smoothdensity}, for any $t>s$, the law of $Y_t^{b}$ conditioned on $Y_s^{b}=x$ admits a positive transition density $p^{b}(t-s,x,y)$ which is of class $C^1$ w.r.t. $b$. Then, for fixed $b_0\in \R$ and $\frac{a}{\sigma^2}>1+\frac{\sqrt{2}}{2}$, using the Markov property, the chain rule for Radon-Nikodym derivatives and the mean value theorem on the parameter space, the log-likelihood ratio can be written as follows
\begin{align}\label{decompo} 
	&\log\dfrac{d\P_n^{b_0+\varphi_{n\Delta_n}(b_0)u}}{d\P^{b_0}_n}(Y^{n,b_0})=\log\dfrac{p_n(Y^{n,b_0};b_0+\varphi_{n\Delta_n}(b_0)u)}{p_n(Y^{n,b_0};b_0)}\notag\\
	&=\sum_{k=0}^{n-1}\log\dfrac{p^{b_0+\varphi_{n\Delta_n}(b_0)u}}{p^{b_0}}(\Delta_n,Y_{t_k}^{b_0},Y_{t_{k+1}}^{b_0}) \notag\\
	&=\sum_{k=0}^{n-1}\varphi_{n\Delta_n}(b_0)u\int_0^1\dfrac{\partial_{b}p^{b(\ell)}}{p^{b(\ell)}}(\Delta_n,Y_{t_k}^{b_0},Y_{t_{k+1}}^{b_0})d\ell \notag\\ 
	&\textcolor{black}{=uV_n(b_0)-\dfrac{u^2}{2}T_n(b_0)+R_n(b_0),}
\end{align}
where $b(\ell):=b_0+\ell \varphi_{n\Delta_n}(b_0)u$, the random vector $(V_n(b_0),T_n(b_0))$ are the main terms whereas $R_n(b_0)$ are the remainder terms in the expansion. The triplet $(V_n(b_0),T_n(b_0),R_n(b_0))$ will appear from the stochastic analysis of the corresponding score function $\frac{\partial_{b}p^{b(\ell)}}{p^{b(\ell)}}(\Delta_n,Y_{t_k}^{b_0},Y_{t_{k+1}}^{b_0})$. The main terms are given by \eqref{firstdeco}
with $\varphi_{n\Delta_n}(b_0)=\frac{1}{\sqrt{n\Delta_n}}$ in the subcritical case ($b_0>0$), $\varphi_{n\Delta_n}(b_0)=\frac{1}{n\Delta_n}$ in the critical case ($b_0=0$) and $\varphi_{n\Delta_n}(b_0)=e^{b_0\frac{n\Delta_n}{2}}$ in the supercritical case ($b_0<0$). The remainder terms \textcolor{black}{$R_n(b_0)$} are implicitly given by
the decomposition \eqref{decompo} and they will be explicitly determined in the proof of Lemma \ref{negligibleterms}. Our
aim will be to study the convergence of the main terms \textcolor{black}{$(V_n(b_0),T_n(b_0))$}. Note that the convergence of the main terms requires a weaker condition than {\bf(A3)}, that is \textcolor{black}{$2a> 3\sigma^2$},  whereas the convergence of the
remainder term \textcolor{black}{$R_n(b_0)$} will require condition {\bf(A3)}.
\begin{lemma}\label{negligibleterms} Assume conditions {\bf(A1)}-{\bf(A3)} and the setting $\Delta_n\rightarrow 0$ and $n\Delta_n\rightarrow\infty$ as $n\rightarrow\infty$. Then, for the subcritical case, as $n\to\infty$,
	\begin{align}\label{remainder}
		\textcolor{black}{R_n(b_0)}\overset{\widehat{\P}^{b_0}}{\longrightarrow}0.
	\end{align}	
Furthermore, if we assume the additional condition $n\Delta_n^{\delta}\to 0$ for the critical case where $\delta>1$ is arbitrarily large and $\Delta_n^{2-\varepsilon}e^{-b_0n\Delta_n}\to 0$ for the supercritical case where $\varepsilon>0$ is arbitrarily small, then \eqref{remainder} remains valid for both cases.
\end{lemma}
The proof of this lemma will be given in Section $5$. 

\textcolor{black}{The proof of Proposition \ref{measurabledecom} follows from \eqref{decompo}  and Lemma \ref{negligibleterms}.}

\subsection{Proof of Theorems \ref{theorem1}, \ref{theorem2} and \ref{theorem3}}
\begin{proof}
	From the decomposition \eqref{decompo} and Lemma \ref{negligibleterms} \textcolor{black}{(Proposition \ref{measurabledecom})}, to prove Theorems \ref{theorem1}, \ref{theorem2} and \ref{theorem3}, it suffices to show that as $n\to\infty$,
	\begin{align}\label{main}
		\textcolor{black}{(V_n(b_0),T_n(b_0))\overset{\mathcal{L}(\widehat{\P}^{b_0})}{\longrightarrow} (U(b_0), I(b_0))},\qquad\textnormal{and}\quad \widehat{\E}\Bigl[e^{uU(0)-\frac{u^2}{2}I(0)}\Bigr]=1.
	\end{align}
\begin{color}{black}
For this, using equation \eqref{eq1}, 
\begin{align*} 
Y_{t_{k+1}}^{b_0}-Y_{t_{k}}^{b_0}&=(a-b_0Y_{t_k}^{b_0})\Delta_n+\sigma\sqrt{Y_{t_k}^{b_0}}(W_{t_{k+1}}-W_{t_{k}})-b_0\int_{t_k}^{t_{k+1}}(Y_s^{b_0}-Y_{t_k}^{b_0})ds\\
&\qquad+\sigma\int_{t_k}^{t_{k+1}}\big(\sqrt{Y_s^{b_0}}-\sqrt{Y_{t_k}^{b_0}}\big)dW_s+\int_{t_k}^{t_{k+1}}\int_{0}^{\infty}zN(ds,dz).
\end{align*}
Therefore, we write
\begin{align}\label{m1}
(V_n(b_0),T_n(b_0))=\left(U_{t_n}(b_0),I_{t_n}(b_0)\right)+\left(\sum_{k=0}^{n-1}H_{10}^{b_0}+L_n(b_0),\sum_{k=0}^{n-1}H_{11}^{b_0}\right),
\end{align}
where $t_n=n\Delta_n$ and 
\begin{align*}
&U_{t_n}(b_0)=- \frac{\varphi_{n\Delta_n}(b_0)}{\sigma} \int_0^{t_n} \sqrt{Y_s^{b_0}} dW_s,\qquad I_{t_n}(b_0)= \frac{(\varphi_{n\Delta_n}(b_0))^2}{\sigma^2} \int_0^{t_n} Y_s^{b_0} ds,\\
&H_{10}^{b_0}=\dfrac{\varphi_{n\Delta_n}(b_0)}{\sigma}\int_{t_k}^{t_{k+1}}\big(\sqrt{Y_s^{b_0}}-\sqrt{Y_{t_k}^{b_0}}\big)dW_s,\\
&H_{11}^{b_0}=\dfrac{1}{\sigma^2 }(\varphi_{n\Delta_n}(b_0))^2\int_{t_k}^{t_{k+1}}(Y_s^{b_0}-Y_{t_k}^{b_0})ds,\\
&L_n(b_0)=\frac{\varphi_{n\Delta_n}(b_0)}{\Delta_n}\sum_{k=0}^{n-1}\Big(H_7^{b_0}+H_8^{b_0}+H_9^{b_0}-\dfrac{\Delta_n}{\sigma^2}U_{k,n}(b_0)\Big),\\
&H_7^{b_0}=\frac{\Delta_n}{\sigma^2}b_0\int_{t_k}^{t_{k+1}}(Y_s^{b_0}-Y_{t_k}^{b_0})ds,\qquad H_8^{b_0}=-\frac{\Delta_n}{\sigma}\int_{t_k}^{t_{k+1}}(\sqrt{Y_s^{b_0}}-\sqrt{Y_{t_k}^{b_0}})dW_s,\\
&H_9^{b_0}=-\frac{\Delta_n}{\sigma^2}\int_{t_k}^{t_{k+1}}\int_{0}^{\infty}zN(ds,dz).
\end{align*}	
\end{color}
First, from \eqref{conv_subcritical}, \eqref{conv_critical} and \eqref{conv_critical2}, \eqref{conv_supercritical}, as $n\to \infty$,
	\begin{equation}\label{m0}
		(U_{t_n}(b_0),I_{t_n}(b_0))\overset{\mathcal{L}(\widehat{\P}^{b_0})}{\longrightarrow} (U(b_0),I(b_0)),\qquad\textnormal{and}\quad \widehat{\E}\Bigl[e^{uU(0)-\frac{u^2}{2}I(0)}\Bigr]=1,
	\end{equation}
	for three cases. Now, we apply Lemma \ref{zero} to $H_{10}^{b_0}$. Clearly, $\sum_{k=0}^{n-1}\widehat{\E}^{b_0}[H_{10}^{b_0}\vert \widehat{\mathcal{F}}_{t_k}]=0$. Next, applying It\^o's formula, we get
	\begin{align*}
		&\widehat{\E}^{b_0}\Big[\Big\vert\sum_{k=0}^{n-1}\widehat{\E}^{b_0}[(H_{10}^{b_0})^2\vert \widehat{\mathcal{F}}_{t_k}]\Big\vert\Big]=\sum_{k=0}^{n-1}\widehat{\E}^{b_0}[(H_{10}^{b_0})^2]\\
		&=\dfrac{1}{\sigma^2 }(\varphi_{n\Delta_n}(b_0))^2\sum_{k=0}^{n-1}\int_{t_k}^{t_{k+1}}\widehat{\E}^{b_0}\Big[\Big(\sqrt{Y_s^{b_0}}-\sqrt{Y_{t_k}^{b_0}}\Big)^2\Big]ds\\
		&=\dfrac{1}{\sigma^2 }(\varphi_{n\Delta_n}(b_0))^2\sum_{k=0}^{n-1}\int_{t_k}^{t_{k+1}}\widehat{\E}^{b_0}\Big[\Big(\int_{t_k}^{s}\Big((\frac{a}{2}-\frac{\sigma^2}{8})\frac{1}{\sqrt{Y_u^{b_0}}}-\frac{b_0}{2}\sqrt{Y_u^{b_0}}\Big)du+\frac{\sigma}{2}\int_{t_k}^{s}dW_u\\
		&\qquad+\int_{t_k}^{s}\int_0^{\infty}\textcolor{black}{(\sqrt{Y_{u-}^{b_0}+z}-\sqrt{Y_{u-}^{b_0}})}N(du,dz)\Big)^2\Big]ds\\
		&\leq C(\varphi_{n\Delta_n}(b_0))^2\sum_{k=0}^{n-1}\int_{t_k}^{t_{k+1}}(s-t_k)\Big((s-t_k)\big(\sup_{u\in[t_k,s]}\widehat{\E}^{b_0}\big[\frac{1}{Y_u^{b_0}}\big]+\vert b_0\vert \sup_{u\in[t_k,s]}\widehat{\E}^{b_0}[Y_u^{b_0}]\big)\\
		&\qquad+1+\int_0^{\infty}z^2m(dz)\sup_{u\in[t_k,s]}\widehat{\E}^{b_0}\big[\frac{1}{Y_u^{b_0}}\big]\Big)ds.
	\end{align*}
	Thus, for subcritical case $b_0>0$ with $\varphi_{n\Delta_n}(b_0)=\frac{1}{\sqrt{n\Delta_n}}$ (respectively critical case $b_0=0$ with $\varphi_{n\Delta_n}(0)=\frac{1}{n\Delta_n}$, respectively supercritical case $b_0<0$ with $\varphi_{n\Delta_n}(b_0)=e^{b_0\frac{n\Delta_n}{2}}$), using Lemma \ref{moment} \textnormal{(i)} (respectively \textcolor{black}{Lemma \ref{auxi} \textnormal{(i)}}, respectively \textcolor{black}{Lemma \ref{auxi} \textnormal{(ii)}}), Lemma \ref{Pro_moments} and standard calculations, under $2a> \sigma^2$ we get $\widehat{\E}^{b_0}[\vert\sum_{k=0}^{n-1}\widehat{\E}^{b_0}[(H_{10}^{b_0})^2\vert \widehat{\mathcal{F}}_{t_k}]\vert]\leq C\Delta_n$ for $b_0>0$, $b_0=0$ and $b_0<0$, which tends to zero. Thus, $\sum_{k=0}^{n-1}\widehat{\E}^{b_0}[(H_{10}^{b_0})^2\vert \widehat{\mathcal{F}}_{t_k}]\overset{\widehat{\P}^{b_0}}{\longrightarrow}0$. Therefore, as $n\to\infty$, 
	\begin{align}\label{m2}
	\sum_{k=0}^{n-1}H_{10}^{b_0}\overset{\widehat{\P}^{b_0}}{\longrightarrow}0.
	\end{align}	
	Finally, using equation \eqref{eqintegral}, we obtain
	\begin{align*}
		\widehat{\E}^{b_0}\big[\big\vert\sum_{k=0}^{n-1}H_{11}^{b_0}\big\vert\big]&\leq \dfrac{1}{\sigma^2 }(\varphi_{n\Delta_n}(b_0))^2\sum_{k=0}^{n-1}\int_{t_k}^{t_{k+1}}\widehat{\E}^{b_0}\big[\big\vert Y_s^{b_0}-Y_{t_k}^{b_0}\big\vert\big]ds\\
		&=\dfrac{1}{\sigma^2 }(\varphi_{n\Delta_n}(b_0))^2\sum_{k=0}^{n-1}\int_{t_k}^{t_{k+1}}\widehat{\E}^{b_0}\Big[\Big\vert \int_{t_k}^s(a-b_0Y_u^{b_0})du +\sigma\int_{t_k}^s\sqrt{Y_u^{b_0}}dW_u\\
		&\qquad+\int_{t_k}^s\int_0^{\infty}zN(du,dz)\Big\vert\Big]ds\\
		&\leq C(\varphi_{n\Delta_n}(b_0))^2\sum_{k=0}^{n-1}\int_{t_k}^{t_{k+1}}\Big(\int_{t_k}^s(a+\vert b_0\vert\widehat{\E}^{b_0}[Y_u^{b_0}])du\\
		&\qquad +\sigma\big(\int_{t_k}^s\widehat{\E}^{b_0}[Y_u^{b_0}]du\big)^{\frac{1}{2}}+(s-t_k)\int_0^{\infty}zm(dz)\Big)ds.
	\end{align*}
	Thus, for subcritical case $b_0>0$ with $\varphi_{n\Delta_n}(b_0)=\frac{1}{\sqrt{n\Delta_n}}$ (respectively critical case $b_0=0$ with $\varphi_{n\Delta_n}(0)=\frac{1}{n\Delta_n}$, respectively supercritical case $b_0<0$ with $\varphi_{n\Delta_n}(b_0)=e^{b_0\frac{n\Delta_n}{2}}$), using Lemma \ref{Pro_moments} and standard calculations, we get $\widehat{\E}^{b_0}[\vert\sum_{k=0}^{n-1}H_{11}^{b_0}\vert]\leq C\sqrt{\Delta_n}$ for $b_0>0$, $b_0=0$ and $b_0<0$, which tends to zero. Thus, we have shown that as $n\to\infty$,
	\begin{equation}\begin{split}\label{m3}	
			\sum_{k=0}^{n-1}H_{11}^{b_0}\overset{\widehat{\P}^{b_0}}{\longrightarrow}0.	
		\end{split}
	\end{equation}
	\begin{color}{black}
	Now, note that from equation \eqref{c2eq1rajoute}, we write
	\begin{equation}\begin{split}\label{splitequ}
	&X_{t_{k+1}}^{b}-X_{t_{k}}^{b}-(a-bX_{t_k}^{b})\Delta_n-\sigma\sqrt{X_{t_k}^{b}}(B_{t_{k+1}}-B_{t_{k}})\\
	&=-b\int_{t_k}^{t_{k+1}}(X_s^{b}-X_{t_k}^{b})ds+\sigma\int_{t_k}^{t_{k+1}}\big(\sqrt{X_s^{b}}-\sqrt{X_{t_k}^{b}}\big)dB_s+\int_{t_k}^{t_{k+1}}\int_{0}^{\infty}zM(ds,dz),
	\end{split}
	\end{equation}
	which implies that
	\begin{align*}
	L_n(b_0)&=\frac{\varphi_{n\Delta_n}(b_0)}{\Delta_n}\sum_{k=0}^{n-1}\Big(H_7^{b_0}+H_8^{b_0}+H_9^{b_0}\notag\\
	&\qquad-\int_0^1\widetilde{\E}_{t_k,Y_{t_{k}}^{b_0}}^{b(\ell)}\left[H_4^{b(\ell)}+H_5^{b(\ell)}+H_6^{b(\ell)}\big\vert X_{t_{k+1}}^{b(\ell)}=Y_{t_{k+1}}^{b_0}\right]d\ell\Big),
	\end{align*}
	where	
	\begin{align*}
	H_4^{b}&=\frac{\Delta_n}{\sigma^2}b\int_{t_k}^{t_{k+1}}(X_s^{b}-X_{t_k}^{b})ds,\qquad H_5^{b}=-\frac{\Delta_n}{\sigma}\int_{t_k}^{t_{k+1}}(\sqrt{X_s^{b}}-\sqrt{X_{t_k}^{b}})dB_s,\\
	H_6^{b}&=-\frac{\Delta_n}{\sigma^2}\int_{t_k}^{t_{k+1}}\int_{0}^{\infty}zM(ds,dz).
	\end{align*}
	\end{color}
	Consequently, from \eqref{m1}, \eqref{m0}, \eqref{m2},  \eqref{m3}, and \textcolor{black}{Lemmas \ref{lemma2}, \ref{lemma3} and \ref{lemma4}}, we conclude \eqref{main} for the subcritical, critical and supercritical cases. Thus, the result follows.
\end{proof}

\section{Technical results}
\label{sec:prelim}
The goal of this section is to use some tools of Malliavin calculus to represent the remainder terms $\textcolor{black}{R_n(b_0)}$ in \eqref{decompo} in terms of conditional expectation of Skorohod integral that will be analyzed in Section \ref{proofnegligible}. On the complete probability space $(\widetilde{\Omega}, \widetilde{\mathcal{F}}, \widetilde{\P})$, we consider the flow process $X^{b}(s,x)=(X_t^{b}(s,x), t\geq s)$, $x\in\R_{++}$, \textcolor{black}{associated to the solution of \eqref{c2eq1rajoute}}, on the time interval $[s,\infty)$ and with initial condition $X_{s}^{b}(s,x)=x$ satisfying
\begin{equation}\label{flow}
X_t^{b}(s,x)=x+\int_s^t (a-bX_u^{b}(s,x))du +\sigma\int_s^t\sqrt{X_u^{b}(s,x)}dB_u+\int_{s}^t\int_{0}^{\infty}z
M(du,dz),
\end{equation}
for any $t\geq s$. 

For any $k \in \{0,...,n-1\}$, the process $(X_t^{b}(t_k,x), t\in [t_k,t_{k+1}])$ satisfies
\begin{equation}\label{flowk}
X_t^{b}(t_k,x)=x+\int_{t_k}^t (a-bX_u^{b}(t_k,x))du +\sigma\int_{t_k}^t\sqrt{X_u^{b}(t_k,x)}dB_u+\int_{t_k}^t\int_{0}^{\infty}zM(du,dz).
\end{equation}
Condition $2a\geq \sigma^2$ and the fact that the subordinator admits only positive jumps imply that the jump-type CIR process $X^b=(X_t^b)_{t\in\R_+}$ remains almost surely strictly
positive. By \cite[Theorem V.39]{P05}, the process $(X_t^{b}(t_k,x), t\in [t_k,t_{k+1}])$ is differentiable w.r.t. $x$ that we denote by $(\partial_{x}X_t^{b}(t_k,x), t\in [t_k,t_{k+1}])$. Moreover, this process admits the derivative w.r.t. the parameter $b$ that we denote by $(\partial_{b}X_t^{b}(t_k,x), t\in [t_k,t_{k+1}])$ since this problem is similar to the derivative w.r.t. the initial condition (see e.g. \cite[Theorem 10.1 page 486]{P16}). Under condition $2a\geq \sigma^2$, these processes are solutions to the following equations
\begin{align}
&\partial_xX_t^{b}(t_k,x)=1-b\int_{t_k}^t \partial_xX_u^{b}(t_k,x)du +\frac{\sigma}{2}\int_{t_k}^t\frac{ \partial_xX_u^{b}(t_k,x)}{\sqrt{X_u^{b}(t_k,x)}}dB_u,\label{px}\\
&\partial_{b}X_t^{b}(t_k,x)=-\int_{t_k}^t\left(X_u^{b}(t_k,x)+b\partial_{b}X_u^{b}(t_k,x)\right)du +\frac{\sigma}{2}\int_{t_k}^t\frac{\partial_bX_u^{b}(t_k,x)}{\sqrt{X_u^{b}(t_k,x)}}dB_u.\label{pb}
\end{align}
Therefore, their solutions are explicitly given by
\begin{align}
\partial_xX_t^{b}(t_k,x)&=\exp{\left\{-b(t-t_k)-\dfrac{\sigma^2}{8}\int_{t_k}^t\dfrac{du}{X_u^{b}(t_k,x)}+\dfrac{\sigma}{2}\int_{t_k}^t\dfrac{dB_u}{\sqrt{X_u^{b}(t_k,x)}}\right\}},\label{dxe}\\
\partial_bX_t^{b}(t_k,x)&=-\int_{t_k}^{t}X_r^{b}(t_k,x)\exp{\bigg\{-b(t-r)-\frac{\sigma^2}{8}\int_{r}^t\frac{du}{X_u^{b}(t_k,x)}+\frac{\sigma}{2}\int_{r}^t\frac{dB_u}{\sqrt{X_u^{b}(t_k,x)}}\bigg\}}dr.\label{dxb}
\end{align}
Observe that from \eqref{dxe}, \eqref{dxb}, we can write
\begin{align}
\partial_bX_t^{b}(t_k,x)=-\int_{t_k}^{t}X_r^{b}(t_k,x)\partial_xX_t^{b}(t_k,x)(\partial_xX_r^{b}(t_k,x))^{-1}dr.\label{dxb2}
\end{align}
The following crucial estimates will be useful.
\begin{lemma}\label{estimates} \textnormal{(i)} Assume {\bf(A1)} and {\bf(A2)}. For any $p\geq 1$, there exists $C>0$ such that
	\begin{align}
		\widetilde{\E}_{t_k,x}^{b}\Big[\big(X_t^{b}(t_k,x)\big)^p\Big]\leq C\left(1+x^p\right).\label{e1}
	\end{align}		
		\noindent\textnormal{(ii)} Assume {\bf(A1)} and $2a\geq \sigma^2$. For any $p\in[1,\frac{2a}{\sigma^2}-1)$, there exists $C>0$ such that
		\begin{align}
		\widetilde{\E}_{t_k,x}^{b}\bigg[\dfrac{1}{(X_t^{b}(t_k,x))^p}\bigg]\leq\dfrac{C}{x^p}.\label{e2}
	\end{align}
		\noindent\textnormal{(iii)} Assume {\bf(A1)} and $2a\geq \sigma^2$. For any $p\geq -\frac{(\frac{2a}{\sigma^2}-1)^2}{2(\frac{2a}{\sigma^2}-\frac{1}{2})}$, there exists $C>0$ such that
		\begin{align}
		&\widetilde{\E}_{t_k,x}^{b}\Big[\big(\partial_xX_t^{b}(t_k,x)\big)^p\Big]\leq C\left(1+\dfrac{1}{x^{\frac{\frac{2a}{\sigma^2}-1+p}{2}}}\right),\label{e4}
	\end{align}		
	for any $k \in \{0,...,n-1\}$, $x\in\R_{++}$  and $t\in[t_k,t_{k+1}]$. 
\end{lemma}
Next, to derive a new representation for the score function, the Malliavin calculus on the Wiener space induced by $B$ will be applied. Let $D$ and $\delta$ denote the Malliavin derivative operator and the Skorohod integral w.r.t. $B$ on $[t_k,t_{k+1}]$. We denote by $\mathbb{D}^{1,2}$ the Sobolev space of random variables which are Malliavin differentiable w.r.t. $B$, and by $\textnormal{Dom}\ \delta$ the domain of $\delta$. Recall that the Malliavin calculus for CIR process is developed by Al\`os and Ewald \cite{AE08}, and Altmayer and Neuenkirch \cite{AN14} whereas the Malliavin calculus for CIR process with jumps is discussed e.g. in \cite[Example 1]{P08}, \cite[Section 4.4]{CM10} and \cite[page 419]{K12}. 
 
Under condition $2a\geq \sigma^2$, for any $t\in [t_k,t_{k+1}]$, the random variable $X_t^{b}(t_k,x)$ belongs to $\mathbb{D}^{1,2}$. From \eqref{flowk} and the chain rule, the Malliavin derivative satisfies
\begin{align}\label{MD}
D_sX_t^{b}(t_k,x)&=\sigma\sqrt{X_s^{b}(t_k,x)}-b\int_{s}^tD_sX_u^{b}(t_k,x)du +\frac{\sigma}{2}\int_{s}^t\frac{D_sX_u^{b}(t_k,x)}{\sqrt{X_u^{b}(t_k,x)}}dB_u,	
\end{align}
for $s\leq t$ a.e., and $D_sX_t^{b}(t_k,x)=0$ for $s>t$ a.e. Furthermore, proceeding as in \cite[(2.59)]{N} by using It\^o's formula and the uniqueness of the solution to equation \eqref{MD} (see also \cite[Proposition 7]{P08}), we obtain the following expression
\begin{equation}\label{expression1}
	D_sX_t^{b}(t_k,x)=\sigma\sqrt{X_s^{b}(t_k,x)}\partial_xX_t^{b}(t_k,x)(\partial_xX_s^{b}(t_k,x))^{-1}{\bf 1}_{[t_k,t]}(s).
\end{equation}
The following lemma shows the Malliavin differentiability of the flow process and its inverse.
\begin{lemma}\label{Malliderivable} Let $b\in\R$, $k \in \{0,...,n-1\}$, $t\in [t_k,t_{k+1}]$, and $x\in\R_{++}$.\\
		\noindent\textnormal{(i)} Assume condition $\frac{a}{\sigma^2}>\frac{5+3\sqrt{2}}{2}$. Then, $\partial_xX_t^{b}(t_k,x)\in \mathbb{D}^{1,2}$. Furthermore,
		\begin{align}\label{dmpx}
			D_s\big(\partial_xX_t^{b}(t_k,x)\big)&=\partial_xX_t^{b}(t_k,x)\bigg(\dfrac{\sigma}{2\sqrt{X_s^{b}(t_k,x)}}+\dfrac{\sigma^2}{8}\int_{s}^t\dfrac{1}{(X_u^{b}(t_k,x))^2}D_sX_u^{b}(t_k,x)du\notag\\
			&\qquad-\dfrac{\sigma}{4}\int_{s}^t\dfrac{1}{(X_u^{b}(t_k,x))^{\frac{3}{2}}}D_sX_u^{b}(t_k,x)dB_u\bigg){\bf 1}_{[t_k,t]}(s).
		\end{align} 
		\noindent\textnormal{(ii)} Assume condition $\frac{a}{\sigma^2}>\frac{7}{2}+\sqrt{10}$. Then, $(\partial_xX_t^{b}(t_k,x))^{-1}\in \mathbb{D}^{1,2}$ and $\frac{X_{t}^{b}(t_k,x)}{\partial_xX_{t}^{b}(t_k,x)}\in \mathbb{D}^{1,2}$. Furthermore,
		\begin{align}
			&D_s\Big(\dfrac{1}{\partial_xX_{t}^{b}(t_k,x)}\Big)=\frac{-1}{(\partial_xX_t^{b}(t_k,x))^2}D_s\big(\partial_xX_t^{b}(t_k,x)\big){\bf 1}_{[t_k,t]}(s),\label{Mallinveflow1}\\
			&D_s\Big(\dfrac{X_{t}^{b}(t_k,x)}{\partial_xX_{t}^{b}(t_k,x)}\Big)=\dfrac{1}{\partial_xX_{t}^{b}(t_k,x)}D_sX_{t}^{b}(t_k,x)+X_{t}^{b}(t_k,x)D_s\Big(\dfrac{1}{\partial_xX_{t}^{b}(t_k,x)}\Big).\label{Mallinveflow2}
		\end{align} 	
\end{lemma}
Following \cite[Proposition 4.1]{G01}, we have the following expression for the score function.
\begin{proposition} \label{c2prop1}
Assume condition {\bf(A1)} and $\frac{a}{\sigma^2}>\frac{7}{2}+\sqrt{10}$. Then, for all $k \in \{0,...,n-1\}$, $b\in\R$, and $x, y\in\R_{++}$,
\begin{align*}
\dfrac{\partial_{b}p^{b}}{p^{b}}\left(\Delta_n,x,y\right)=\dfrac{1}{\Delta_n}\widetilde{\E}_{t_k,x}^{b}\left[\delta\left(\partial_{b}X_{t_{k+1}}^{b}(t_k,x)U^{b}(t_k,x)\right)\big\vert X_{t_{k+1}}^{b}=y\right],
\end{align*}
where $U^{b}(t_k,x):=(U^{b}_t(t_k,x), t\in[t_k,t_{k+1}])$ with $U^{b}_t(t_k,x):=(D_tX_{t_{k+1}}^{b}(t_k,x))^{-1}$. 

Furthermore, under condition {\bf(A3)}: $\frac{a}{\sigma^2}>\frac{15+\sqrt{185}}{4}$, the Skorohod integral is decomposed as follows
	\begin{equation}\label{derib}\begin{split}
	\delta\left(\partial_{b}X_{t_{k+1}}^{b}(t_k,x)U^{b}(t_k,x)\right)=-\frac{\Delta_n}{\sigma}\sqrt{x} \left(B_{t_{k+1}}-B_{t_{k}}\right)+H_1^{b}+H_2^{b}+H_3^{b}.
	\end{split}
	\end{equation}
where 
\begin{align*}
&H_1^{b}=H_1^{b}(t_k,x)=-\Delta_n\dfrac{x}{\sigma}\int_{t_k}^{t_{k+1}}\left(\frac{\partial_{x}X_{s}^{b}(t_k,x)}{\sqrt{X_s^{b}(t_k,x)}}-\frac{\partial_{x}X_{t_k}^{b}(t_k,x)}{\sqrt{X_{t_k}^{b}(t_k,x)}}\right)dB_s,\\
&H_2^{b}=H_2^{b}(t_k,x)=-\int_{t_k}^{t_{k+1}}\left(\frac{X_{s}^{b}(t_k,x)}{\partial_{x}X_{s}^{b}(t_k,x)}-\frac{X_{t_k}^{b}(t_k,x)}{\partial_{x}X_{t_k}^{b}(t_k,x)}\right)ds\int_{t_k}^{t_{k+1}}\frac{\partial_{x}X_{s}^{b}(t_k,x)}{\sigma\sqrt{X_s^{b}(t_k,x)}}dB_s,\\
&H_3^{b}=H_3^{b}(t_k,x)=\int_{t_k}^{t_{k+1}}\int_{s}^{t_{k+1}}D_s\Big(\dfrac{X_{r}^{b}(t_k,x)}{\partial_xX_{r}^{b}(t_k,x)}\Big)dr\frac{\partial_{x}X_{s}^{b}(t_k,x)}{\sigma\sqrt{X_s^{b}(t_k,x)}}ds.
\end{align*}
\end{proposition}
The following crucial estimates will be needed.
\begin{lemma} \label{estimate} Let $k \in \{0,...,n-1\}$, $b\in\R$, and $x, y\in\R_{++}$. \\
	\noindent\textnormal{(i)} Assume condition {\bf(A1)} and  $\frac{a}{\sigma^2}>\frac{7}{2}+\sqrt{10}$. Then, we have
	\begin{align}
		\widetilde{\E}_{t_k,x}^{b}[H_1^{b}+H_2^{b}+H_3^{b}]=0.\label{es3}
	\end{align}	
	\noindent\textnormal{(ii)} \textcolor{black}{Let $q\geq 1$} and assume conditions {\bf(A1)}, {\bf(A2)} and  $\frac{a}{\sigma^2}>\frac{13q+2+\sqrt{169q^2+16q}}{4}$. Then, there exists a constant $C>0$ such that 
	\begin{align}
	\widetilde{\E}_{t_k,x}^{b}\Big[\big\vert H_1^{b}+H_2^{b}+H_3^{b}\big\vert^{2q}\Big]\leq C\Delta_n^{3q+\frac{1}{\overline{p}}}\Big(x^{q}+\dfrac{1}{x^{\frac{\frac{2a}{\sigma^2}-1}{2}+5q}}\Big),\label{es4}
	\end{align}	
	where $\overline{p}=\frac{5q+4+\sqrt{25q^2+8q}}{4(2q+1)}$.
\end{lemma}
\section{Proof of \textcolor{black}{negligible contributions}}
\label{proofnegligible}

First, using the decomposition \eqref{decompo} and Proposition \ref{c2prop1}, under condition {\bf(A3)}: $\frac{a}{\sigma^2}>\frac{15+\sqrt{185}}{4}$ , we obtain 
\begin{align*}
R_n(b_0)&=\sum_{k=0}^{n-1}\varphi_{n\Delta_n}(b_0)u\int_0^1\dfrac{\partial_{b}p^{b(\ell)}}{p^{b(\ell)}}(\Delta_n,Y_{t_k}^{b_0},Y_{t_{k+1}}^{b_0})d\ell\textcolor{black}{-uV_n(b_0)+\dfrac{u^2}{2}T_n(b_0)}\\
&=\sum_{k=0}^{n-1}\frac{\varphi_{n\Delta_n}(b_0)u}{\Delta_n}\int_0^1\widetilde{\E}_{t_k,Y_{t_{k}}^{b_0}}^{b(\ell)}\Big[-\frac{\Delta_n}{\sigma}\sqrt{Y_{t_k}^{b_0}} (B_{t_{k+1}}-B_{t_{k}})+H^{b(\ell)}\vert X_{t_{k+1}}^{b(\ell)}=Y_{t_{k+1}}^{b_0}\Big]d\ell\\
&\qquad\textcolor{black}{-uV_n(b_0)+\dfrac{u^2}{2}T_n(b_0),}
\end{align*}
where $H^{b(\ell)}=H_1^{b(\ell)}+H_2^{b(\ell)}+H_3^{b(\ell)}$. Then, from expression \eqref{splitequ}, we write
\begin{align}
R_n(b_0)&=\sum_{k=0}^{n-1}\frac{\varphi_{n\Delta_n}(b_0)u}{\Delta_n}\int_0^1\widetilde{\E}_{t_k,Y_{t_{k}}^{b_0}}^{b(\ell)}\bigg[-\frac{\Delta_n}{\sigma^2} \sqrt{\frac{Y_{t_k}^{b_0}}{X_{t_k}^{b(\ell)}}}\Big(X_{t_{k+1}}^{b(\ell)}-X_{t_{k}}^{b(\ell)}-(a-b(\ell)X_{t_k}^{b(\ell)})\Delta_n \notag\\
&\quad+b(\ell)\int_{t_k}^{t_{k+1}}(X_s^{b(\ell)}-X_{t_k}^{b(\ell)})ds-\sigma\int_{t_k}^{t_{k+1}}(\sqrt{X_s^{b(\ell)}}-\sqrt{X_{t_k}^{b(\ell)}})dB_s\notag\\
&\quad-\int_{t_k}^{t_{k+1}}\int_{0}^{\infty}zM(ds,dz)\Big)+H^{b(\ell)}\vert X_{t_{k+1}}^{b(\ell)}=Y_{t_{k+1}}^{b_0}\bigg]d\ell\textcolor{black}{-uV_n(b_0)+\dfrac{u^2}{2}T_n(b_0)} \notag\\
&=\sum_{k=0}^{n-1}\frac{\varphi_{n\Delta_n}(b_0)u}{\Delta_n}\int_0^1\bigg(-\frac{\Delta_n}{\sigma^2}\big(Y_{t_{k+1}}^{b_0}-Y_{t_{k}}^{b_0}-(a-b(\ell)Y_{t_k}^{b_0})\Delta_n\big)  \notag\\
&\quad+\widetilde{\E}_{t_k,Y_{t_{k}}^{b_0}}^{b(\ell)}\Big[-H_4^{b(\ell)}-H_5^{b(\ell)}-H_6^{b(\ell)}+H^{b(\ell)}\vert X_{t_{k+1}}^{b(\ell)}=Y_{t_{k+1}}^{b_0}\Big]\bigg)d\ell\textcolor{black}{-uV_n(b_0)+\dfrac{u^2}{2}T_n(b_0)} \notag\\
&=\frac{\varphi_{n\Delta_n}(b_0)u}{\Delta_n}\sum_{k=0}^{n-1}\int_0^1\widetilde{\E}_{t_k,Y_{t_{k}}^{b_0}}^{b(\ell)}\left[H^{b(\ell)}\big\vert X_{t_{k+1}}^{b(\ell)}=Y_{t_{k+1}}^{b_0}\right]d\ell. \label{decompo2}
\end{align}
This stochastic expansion, combined with Lemma \ref{lemma1} below, finishes the proof of Lemma \ref{negligibleterms}.

Now, to prove Lemmas \ref{lemma1}-\ref{lemma3} and \ref{lemma4} below, the two following preliminary results are needed. For this, we denote by $\widehat{\P}_{t_k,x}^{b}$ the probability law of $Y^{b}$ starting at $x$ at time $t_k$, i.e., $\widehat{\P}_{t_k,x}^{b}(A)=\widehat{\E}[{\bf 1}_{A}\vert Y_{t_{k}}^{b}=x]$ for all $A\in \widehat{\mathcal{F}}$, and by $\widehat{\E}_{t_k,x}^{b}$ the expectation w.r.t. $\widehat{\P}_{t_k,x}^{b}$. That is,  $\widehat{\E}_{t_k,x}^{b}[V]=\widehat{\E}[V\vert Y_{t_{k}}^{b}=x]$ for all $\widehat{\mathcal{F}}$-measurable random variables $V$. The change of measures on each interval $I_k:=[t_k, t_{k+1}]$ will be used. For all $b, b_1\in\R$, $x\in\R_{++}$ and $k\in\{0,...,n-1\}$, by \cite[Proposition 4.1]{BBKP17}, the probability measures $\widehat{\P}_{t_k,x}^{b_1}$ and $\widehat{\P}_{t_k,x}^{b}$ are absolutely continuous w.r.t. each other and its Radon-Nikodym derivative is given by
\begin{equation}\label{ratio2}
	\begin{split}
		\dfrac{d\widehat{\P}_{t_k,x}^{b_1}}{d\widehat{\P}_{t_k,x}^{b}}((Y_t^{b})_{t\in I_k})=\exp\left\{-\frac{b_1-b}{\sigma}\int_{t_k}^{t_{k+1}}\sqrt{Y_s^{b}}dW_s-\frac{(b_1-b)^2}{2\sigma^2}\int_{t_k}^{t_{k+1}}Y_s^{b}ds\right\}.
	\end{split}
\end{equation}
Let $V$ be a $\widetilde{\mathcal{F}}_{t_{k+1}}$-measurable random variable which will be $H^{b}$, $(H^{b})^2$, $H_4^{b}$, $(H_4^{b})^2$, $H_5^{b}$, $(H_5^{b})^2$, and let $\widehat{V}$ be a $\widehat{\mathcal{F}}_{t_{k+1}}$-measurable random variable which will be $\int_{t_k}^{t_{k+1}}(Y_s^{b_0}-Y_{t_k}^{b_0})ds$. As in \cite[Lemma 9]{BKT17}, we have the following lemma.
\begin{lemma}\label{change} Assume condition {\bf(A1)}. Then for any $k\in\{0,...,n-1\}$, $b\in\R$ and $x\in\R_{++}$,
	\begin{equation}\label{for1}\begin{split}
	\widehat{\E}_{t_k,x}^{b_0}\Big[\widetilde{\E}_{t_k,x}^{b}\big[V\vert X_{t_{k+1}}^{b}=
	Y_{t_{k+1}}^{b_0}\big]\Big]=\widetilde{\E}_{t_k,x}^{b}\left[V\right]+\widetilde{\E}_{t_k,x}^{b}\Big[\widetilde{\E}_{t_k,x}^{b}\big[V\vert X_{t_{k+1}}^{b}\big]\Big(\frac{d\widetilde{\P}_{t_k,x}^{b_0}}{d\widetilde{\P}_{t_k,x}^{b}}((X_t^{b})_{t\in I_k})-1\Big)\Big].
	\end{split}
	\end{equation}
	Similarly, we have
	\begin{equation}\label{for2}\begin{split}
	\widehat{\E}_{t_k,x}^{b_0}[\widehat{V}]=\widehat{\E}_{t_k,x}^{b}[\widehat{V}]+\widehat{\E}_{t_k,x}^{b}\Big[\widehat{V}\Big(\frac{d\widehat{\P}_{t_k,x}^{b_0}}{d\widehat{\P}_{t_k,x}^{b}}((Y_t^{b})_{t\in I_k})-1\Big)\Big].
	\end{split}
	\end{equation}	
\end{lemma}
The estimate for the second terms in Lemma \ref{change} is guaranteed by the following lemma.
\begin{lemma}\label{deviation1} Let $b=b(\ell)$. For any $q>1$, there exist constants $C_1, C_2>0$ such that for any $k\in\{0,...,n-1\}$, $x\in\R_{++}$ and $n$ large enough,
	\begin{equation}\label{for3}\begin{split}
	&\left\vert\widetilde{\E}_{t_k,x}^{b}\left[\widetilde{\E}_{t_k,x}^{b}\big[V\vert X_{t_{k+1}}^{b}\big]\Big(\frac{d\widetilde{\P}_{t_k,x}^{b_0}}{d\widetilde{\P}_{t_k,x}^{b}}((X_t^{b})_{t\in I_k})-1\Big)\right]\right\vert  \\
	&\leq C_1\sqrt{\Delta_n}\big(\widetilde{\E}_{t_k,x}^{b}[\vert V\vert^q]\big)^{\frac{1}{q}}e^{C_2(b_0-b)^2\Delta_nx}\vert b-b_0\vert\left(1+\sqrt{x}+\sqrt{\Delta_n}\vert b-b_0\vert x\right).
	\end{split}
	\end{equation}	
	Moreover,
	\begin{equation}\label{for4}\begin{split}
	&\left\vert\widehat{\E}_{t_k,x}^{b}\left[\widehat{V}\Big(\frac{d\widehat{\P}_{t_k,x}^{b_0}}{d\widehat{\P}_{t_k,x}^{b}}((Y_t^{b})_{t\in I_k})-1\Big)\right]\right\vert\\
	&\leq C_1\sqrt{\Delta_n}\big(\widehat{\E}_{t_k,x}^{b}[\vert \widehat{V}\vert^q]\big)^{\frac{1}{q}}e^{C_2(b_0-b)^2\Delta_nx} \vert b-b_0\vert\left(1+\sqrt{x}+\sqrt{\Delta_n}\vert b-b_0\vert x\right).
	\end{split}
	\end{equation}
\end{lemma}

\begin{lemma}\label{lemma1} Assume conditions {\bf(A1)}-{\bf(A3)}. Then, as $n\to\infty$,
\begin{align*} 
\sum_{k=0}^{n-1}\frac{\varphi_{n\Delta_n}(b_0)u}{\Delta_n}\int_0^1\widetilde{\E}_{t_k,Y_{t_k}^{b_0}}^{b(\ell)}\left[H^{b(\ell)}\big\vert X^{b(\ell)}_{t_{k+1}}=Y_{t_{k+1}}^{b_0}\right]d\ell\overset{\widehat{\P}^{b_0}}{\longrightarrow}0.
\end{align*}
\end{lemma}
\begin{proof}
We apply Lemma \ref{zero} with
$$
\zeta_{k,n}:=\frac{\varphi_{n\Delta_n}(b_0)u}{\Delta_n}\int_0^1\widetilde{\E}_{t_k,Y_{t_k}^{b_0}}^{b(\ell)}\left[H^{b(\ell)}\big\vert X^{b(\ell)}_{t_{k+1}}=Y_{t_{k+1}}^{b_0}\right]d\ell.
$$ 
\begin{color}{black}
First, using \eqref{for1} of Lemma \ref{change} with $V=H^{b(\ell)}$, the fact that  $\widetilde{\E}_{t_k,Y_{t_k}^{b_0}}^{b(\ell)}[H^{b(\ell)}]=0$ by \eqref{es3} of Lemma \ref{estimate}, \eqref{for3} of Lemma \ref{deviation1} with $q=2$, we obtain for $n$ large enough,	
\begin{align*} 
&\Big\vert\sum_{k=0}^{n-1}\widehat{\E}^{b_0}\big[\zeta_{k,n}\vert \widehat{\mathcal{F}}_{t_k}\big]\Big\vert=\bigg\vert\sum_{k=0}^{n-1}\frac{\varphi_{n\Delta_n}(b_0)u}{\Delta_n}\int_0^1\widehat{\E}_{t_k,Y_{t_k}^{b_0}}^{b_0}\left[\widetilde{\E}_{t_k,Y_{t_k}^{b_0}}^{b(\ell)}\left[H^{b(\ell)}\big\vert X^{b(\ell)}_{t_{k+1}}=Y_{t_{k+1}}^{b_0}\right]\right]d\ell\bigg\vert\\
&=\bigg\vert\sum_{k=0}^{n-1}\frac{\varphi_{n\Delta_n}(b_0)u}{\Delta_n}\int_0^1\bigg\{\widetilde{\E}_{t_k,Y_{t_k}^{b_0}}^{b(\ell)}\left[H^{b(\ell)}\right]\\
&\qquad+\widetilde{\E}_{t_k,Y_{t_k}^{b_0}}^{b(\ell)}\bigg[\widetilde{\E}_{t_k,Y_{t_k}^{b_0}}^{b(\ell)}\big[H^{b(\ell)}\vert X_{t_{k+1}}^{b(\ell)}\big]\Big(\frac{d\widetilde{\P}_{t_k,Y_{t_k}^{b_0}}^{b_0}}{d\widetilde{\P}_{t_k,Y_{t_k}^{b_0}}^{b(\ell)}}((X_t^{b(\ell)})_{t\in I_k})-1\Big)\bigg]\bigg\}d\ell\bigg\vert\\
&\leq\bigg\vert\sum_{k=0}^{n-1}\frac{\varphi_{n\Delta_n}(b_0)u}{\Delta_n}\int_0^1 \bigg\vert\widetilde{\E}_{t_k,Y_{t_k}^{b_0}}^{b(\ell)}\bigg[\widetilde{\E}_{t_k,Y_{t_k}^{b_0}}^{b(\ell)}\big[H^{b(\ell)}\vert X_{t_{k+1}}^{b(\ell)}\big]\Big(\frac{d\widetilde{\P}_{t_k,Y_{t_k}^{b_0}}^{b_0}}{d\widetilde{\P}_{t_k,Y_{t_k}^{b_0}}^{b(\ell)}}((X_t^{b(\ell)})_{t\in I_k})-1\Big)\bigg]\bigg\vert d\ell\bigg\vert\\
&\leq C_1\sum_{k=0}^{n-1}\dfrac{\vert u\vert\varphi_{n\Delta_n}(b_0)}{\Delta_n}\sqrt{\Delta_n}\int_0^1\big(\widetilde{\E}_{t_k,Y_{t_k}^{b_0}}^{b(\ell)}[\vert H^{b(\ell)}\vert^2]\big)^{\frac{1}{2}}e^{C_2u^2(\varphi_{n\Delta_n}(b_0))^2\Delta_nY_{t_k}^{b_0}} \\
&\qquad\times\vert u\vert\varphi_{n\Delta_n}(b_0)\big(1+\sqrt{Y_{t_k}^{b_0}}+\sqrt{\Delta_n}\vert u\vert\varphi_{n\Delta_n}(b_0) Y_{t_k}^{b_0}\big)d\ell.
\end{align*}
Next, using \eqref{es4} of Lemma \ref{estimate} with $q=1$, $\overline{p}=\frac{9+\sqrt{33}}{12}$, $3q+\frac{1}{\overline{p}}=3+\frac{12}{9+\sqrt{33}}$, condition {\bf(A3)}, and the decomposition ${\bf 1}_{\{Y_{t_k}^{b_0}\leq 1\}}+{\bf 1}_{\{Y_{t_k}^{b_0}>1\}}$, we obtain for $n$ large enough,	
\begin{align*} 
&\Big\vert\sum_{k=0}^{n-1}\widehat{\E}^{b_0}\big[\zeta_{k,n}\vert \widehat{\mathcal{F}}_{t_k}\big]\Big\vert\\
&\leq  Cu^2\frac{(\varphi_{n\Delta_n}(b_0))^2}{\sqrt{\Delta_n}}\Delta_n^{\frac{3}{2}+\frac{6}{9+\sqrt{33}}} \sum_{k=0}^{n-1}e^{C_2u^2(\varphi_{n\Delta_n}(b_0))^2\Delta_nY_{t_k}^{b_0}}\Big((Y_{t_k}^{b_0})^{\frac{1}{2}}+\dfrac{1}{(Y_{t_k}^{b_0})^{\frac{\frac{2a}{\sigma^2}-1}{4}+\frac{5}{2}}}\Big)\\
&\qquad\times  \big(1+\sqrt{Y_{t_k}^{b_0}}+\sqrt{\Delta_n}\vert u\vert\varphi_{n\Delta_n}(b_0) Y_{t_k}^{b_0}\big)\\
&\leq  Cu^2(\varphi_{n\Delta_n}(b_0))^2\Delta_n^{1+\frac{6}{9+\sqrt{33}}}\sum_{k=0}^{n-1}e^{C_2u^2(\varphi_{n\Delta_n}(b_0))^2\Delta_nY_{t_k}^{b_0}} \Big(Y_{t_k}^{b_0}+\sqrt{\Delta_n}\vert u\vert\varphi_{n\Delta_n}(b_0) (Y_{t_k}^{b_0})^{\frac{3}{2}}\\
&\qquad+\dfrac{1}{(Y_{t_k}^{b_0})^{\frac{\frac{2a}{\sigma^2}-1}{4}+\frac{5}{2}}}\Big),
\end{align*}
for some constant $C>0$. 
\end{color}

Then, using Young's inequality for products with $\frac{1}{\alpha}+\frac{1}{\beta}=1$ and $\beta$ close to $1$, we get
\begin{align*} 
	&\Big\vert\sum_{k=0}^{n-1}\widehat{\E}^{b_0}\big[\zeta_{k,n}\vert \widehat{\mathcal{F}}_{t_k}\big]\Big\vert\leq  Cu^2(\varphi_{n\Delta_n}(b_0))^2\Delta_n^{1+\frac{6}{9+\sqrt{33}}}\bigg(\frac{1}{\alpha}\sum_{k=0}^{n-1}e^{C_2\alpha u^2(\varphi_{n\Delta_n}(b_0))^2\Delta_nY_{t_k}^{b_0}} \\
	&\qquad+\frac{1}{\beta}\sum_{k=0}^{n-1} \Big((Y_{t_k}^{b_0})^{\beta}+(\sqrt{\Delta_n}\vert u\vert\varphi_{n\Delta_n}(b_0))^{\beta} (Y_{t_k}^{b_0})^{\frac{3}{2}\beta}+\dfrac{1}{(Y_{t_k}^{b_0})^{\beta(\frac{\frac{2a}{\sigma^2}-1}{4}+\frac{5}{2})}}\Big)\bigg).
\end{align*} 
Now using the Laplace transform \eqref{Laplace2} and conditions {\bf(A1)}, {\bf(A2)}, we get for $n$ large enough,
\begin{align}\label{boundedexp}
	R_n:=\widehat{\E}^{b_0}\Big[e^{C_2\alpha u^2(\varphi_{n\Delta_n}(b_0))^2\Delta_nY_{t_k}^{b_0}}\Big]\leq C,
\end{align}
for some constant $C>0$ for three cases. Indeed, we prove the convergence of $R_n$ towards $1$ as $n\to\infty$. For this, for subcritical case $b_0>0$ (respectively critical case $b_0=0$, respectively supercritical case $b_0<0$), we take $\varphi_{n\Delta_n}(b_0)=\frac{1}{\sqrt{n\Delta_n}}$ (respectively $\varphi_{n\Delta_n}(b_0)=\frac{1}{n\Delta_n}$, respectively   $\varphi_{n\Delta_n}(b_0)=e^{b_0\frac{n\Delta_n}{2}}$) and use that $1-e^{-b_0t_k}\leq b_0t_k$, $\frac{t_k}{n\Delta_n}\leq 1$ and $e^{-b_0t_k}\leq 1$ (respectively $\frac{t_k}{n\Delta_n}\leq 1$, respectively $e^{b_0(n\Delta_n-t_k)}\leq 1$ and $n\Delta_n\to\infty$).

\begin{color}{black}
Thus, for subcritical case $b_0>0$, taking $\varphi_{n\Delta_n}(b_0)=\frac{1}{\sqrt{n\Delta_n}}$, using \eqref{boundedexp}, Lemma \ref{moment} \textnormal{(i)} and standard calculations, we get $\widehat{\E}^{b_0}[\vert\sum_{k=0}^{n-1}\widehat{\E}^{b_0}[\zeta_{k,n}\vert \widehat{\mathcal{F}}_{t_k}]\vert]\leq C\Delta_n^{\frac{6}{9+\sqrt{33}}}$, which tends to zero. Here notice that $\beta(\frac{\frac{2a}{\sigma^2}-1}{4}+\frac{5}{2})<\frac{2a}{\sigma^2}$. 

For critical case $b_0=0$, taking $\varphi_{n\Delta_n}(b_0)=\frac{1}{n\Delta_n}$, using \eqref{boundedexp} and Lemma \ref{auxi} \textnormal{(i)} with $p\in\{-\beta, -\frac{3}{2}\beta, \beta(\frac{2a/\sigma^2-1}{4}+\frac{5}{2})\}$, we get 
\begin{align*}
&\widehat{\E}^{0}\Big[\big\vert\sum_{k=0}^{n-1}\widehat{\E}^{0}[\zeta_{k,n}\vert \widehat{\mathcal{F}}_{t_k}]\big\vert\Big]\\
&\leq  Cu^2\frac{1}{(n\Delta_n)^2}\Delta_n^{\frac{6}{9+\sqrt{33}}}\bigg(n\Delta_n+(n\Delta_n)^{\beta +1}+(\sqrt{\Delta_n}\frac{\vert u\vert}{n\Delta_n})^{\beta}(n\Delta_n)^{\frac{3}{2}\beta+1}+\dfrac{1}{(n\Delta_n)^{\beta(\frac{\frac{2a}{\sigma^2}-1}{4}+\frac{5}{2})-1}}\bigg)\\
&\leq  C\frac{1}{(n\Delta_n)^2}\Delta_n^{\frac{6}{9+\sqrt{33}}}\bigg(n\Delta_n+(n\Delta_n)^{\beta +1}\bigg)\\
&= C\left(\frac{1}{n\Delta_n}\Delta_n^{\frac{6}{9+\sqrt{33}}}+ (n\Delta_n^{1+\frac{6}{(9+\sqrt{33})(\beta-1)}})^{\beta-1}\right),
\end{align*}
which tends to zero by choosing $\beta$ close to $1$. This is because $n\Delta_n^{\delta}\to 0$ for the critical case where $\delta>1$ is arbitrarily large.

For supercritical case $b_0<0$, taking   $\varphi_{n\Delta_n}(b_0)=e^{b_0\frac{n\Delta_n}{2}}$, using \eqref{boundedexp} and Lemma \ref{auxi} \textnormal{(ii)} with $p\in\{-\beta, -\frac{3}{2}\beta, \beta(\frac{2a/\sigma^2-1}{4}+\frac{5}{2})\}$, we get 
\begin{align*}
&\widehat{\E}^{b_0}\Big[\big\vert\sum_{k=0}^{n-1}\widehat{\E}^{b_0}[\zeta_{k,n}\vert \widehat{\mathcal{F}}_{t_k}]\big\vert\Big]\\
&\leq Cu^2e^{b_0n\Delta_n}\Delta_n^{\frac{6}{9+\sqrt{33}}}\bigg(n\Delta_n+e^{-\beta b_0n\Delta_n}+(\sqrt{\Delta_n}\vert u\vert e^{b_0\frac{n\Delta_n}{2}})^{\beta}e^{-\frac{3}{2}\beta b_0n\Delta_n}+\dfrac{1}{(n\Delta_n)^{\beta(\frac{\frac{2a}{\sigma^2}-1}{4}+\frac{5}{2})-1}}\bigg)\\
&\leq C \Delta_n^{\frac{6}{9+\sqrt{33}}} e^{-(\beta-1)b_0n\Delta_n}\\
&= C(\Delta_n^{2+2\frac{12+\sqrt{33}-(9+\sqrt{33})\beta}{(9+\sqrt{33})(\beta-1)}}e^{-b_0n\Delta_n})^{\beta-1},
\end{align*}
which tends to zero by choosing $\beta\in (1,\frac{12+\sqrt{33}}{9+\sqrt{33}})$. This is because  $\Delta_n^{2-\varepsilon}e^{-b_0n\Delta_n}\to 0$ for the supercritical case where $\varepsilon>0$ is arbitrarily small. This completes the proof of condition (i) of Lemma \ref{zero}. 
\end{color}

Similarly, using Jensen's inequality, \eqref{for1} of Lemma \ref{change} with $V=(H^{b(\ell)})^2$, \eqref{for3} of Lemma \ref{deviation1} with $q=q_0>1$ and $q_0$ close to $1$, \eqref{es4} of Lemma \ref{estimate} with $q\in\{1,q_0\}$, and condition {\bf(A3)}, we get for $n$ large enough,
\begin{align*}
&\sum_{k=0}^{n-1}\widehat{\E}^{b_0}[\zeta_{k,n}^2\vert \widehat{\mathcal{F}}_{t_k}]\\
&\leq\sum_{k=0}^{n-1}\frac{u^2(\varphi_{n\Delta_n}(b_0))^2}{\Delta_n^2}\int_0^1\widehat{\E}_{t_k,Y_{t_k}^{b_0}}^{b_0}\left[\widetilde{\E}_{t_k,Y_{t_k}^{b_0}}^{b(\ell)}\left[(H^{b(\ell)})^2\vert X_{t_{k+1}}^{b(\ell)}=Y_{t_{k+1}}^{b_0}\right]\right]d\ell\\
&\leq \frac{u^2(\varphi_{n\Delta_n}(b_0))^2}{\Delta_n^2}\sum_{k=0}^{n-1}\int_0^1\bigg\{\widetilde{\E}_{t_k,Y_{t_k}^{b_0}}^{b(\ell)}\left[(H^{b(\ell)})^2\right]+C_1\sqrt{\Delta_n}\Big(\widetilde{\E}_{t_k,Y_{t_k}^{b_0}}^{b(\ell)}\big[\vert H^{b(\ell)}\vert^{2q_0}\big]\Big)^{\frac{1}{q_0}} \\
&\qquad\times e^{C_2u^2(\varphi_{n\Delta_n}(b_0))^2\Delta_nY_{t_k}^{b_0}} \vert u\vert\varphi_{n\Delta_n}(b_0)\big(1+\sqrt{Y_{t_k}^{b_0}}+\sqrt{\Delta_n}\vert u\vert\varphi_{n\Delta_n}(b_0) Y_{t_k}^{b_0}\big)\bigg\}d\ell\\
&\leq C\frac{u^2(\varphi_{n\Delta_n}(b_0))^2}{\Delta_n^2}\sum_{k=0}^{n-1}\bigg\{\Delta_n^{3+\frac{12}{9+\sqrt{33}}}\Big(Y_{t_k}^{b_0}+\dfrac{1}{(Y_{t_k}^{b_0})^{\frac{\frac{2a}{\sigma^2}-1}{2}+5}}\Big)+C_1\vert u\vert\Delta_n^{\frac{7}{2}+\frac{1}{q_0\overline{p}}}\varphi_{n\Delta_n}(b_0)\\
&\qquad\times e^{C_2u^2(\varphi_{n\Delta_n}(b_0))^2\Delta_nY_{t_k}^{b_0}}\Big(Y_{t_k}^{b_0}+\dfrac{1}{(Y_{t_k}^{b_0})^{\frac{\frac{2a}{\sigma^2}-1}{2q_0}+5}}\Big)  \big(1+\sqrt{Y_{t_k}^{b_0}}+\sqrt{\Delta_n}\vert u\vert\varphi_{n\Delta_n}(b_0) Y_{t_k}^{b_0}\big)\bigg\}\\
&\leq C\frac{u^2(\varphi_{n\Delta_n}(b_0))^2}{\Delta_n^2}\sum_{k=0}^{n-1}\bigg\{\Delta_n^{3+\frac{12}{9+\sqrt{33}}}\Big(Y_{t_k}^{b_0}+\dfrac{1}{(Y_{t_k}^{b_0})^{\frac{\frac{2a}{\sigma^2}-1}{2}+5}}\Big)+C_1\vert u\vert\Delta_n^{\frac{7}{2}+\frac{1}{q_0\overline{p}}}\varphi_{n\Delta_n}(b_0)\\
&\qquad\times e^{C_2u^2(\varphi_{n\Delta_n}(b_0))^2\Delta_nY_{t_k}^{b_0}}\Big((Y_{t_k}^{b_0})^{\frac{3}{2}}+\sqrt{\Delta_n}\vert u\vert\varphi_{n\Delta_n}(b_0) (Y_{t_k}^{b_0})^{2}+\dfrac{1}{(Y_{t_k}^{b_0})^{\frac{\frac{2a}{\sigma^2}-1}{2q_0}+5}}\Big)\bigg\},
\end{align*}
for some constant $C>0$, where $\overline{p}=\frac{5q_0+4+\sqrt{25q_0^2+8q_0}}{4(2q_0+1)}$. Then, using Young's inequality for products with $\frac{1}{\alpha}+\frac{1}{\beta}=1$ and $\beta$ close to $1$,
\begin{align*}
	\sum_{k=0}^{n-1}\widehat{\E}^{b_0}[\zeta_{k,n}^2\vert \widehat{\mathcal{F}}_{t_k}]&\leq  Cu^2(\varphi_{n\Delta_n}(b_0))^2\Delta_n^{1+\frac{12}{9+\sqrt{33}}}\sum_{k=0}^{n-1}\Big(Y_{t_k}^{b_0}+\dfrac{1}{(Y_{t_k}^{b_0})^{\frac{\frac{2a}{\sigma^2}-1}{2}+5}}\Big)\\
	&\quad+C\vert u\vert^3(\varphi_{n\Delta_n}(b_0))^3\Delta_n^{\frac{3}{2}+\frac{1}{q_0\overline{p}}}\bigg(\frac{1}{\alpha}\sum_{k=0}^{n-1} e^{C_2\alpha u^2(\varphi_{n\Delta_n}(b_0))^2\Delta_nY_{t_k}^{b_0}}\\
	&\quad+\frac{1}{\beta}\sum_{k=0}^{n-1} \Big((Y_{t_k}^{b_0})^{\frac{3}{2}\beta}+(\sqrt{\Delta_n}\vert u\vert\varphi_{n\Delta_n}(b_0))^{\beta} (Y_{t_k}^{b_0})^{2\beta}+\dfrac{1}{(Y_{t_k}^{b_0})^{\beta(\frac{\frac{2a}{\sigma^2}-1}{2q_0}+5)}}\Big)\bigg).
\end{align*}
\begin{color}{black}
	Thus, for subcritical case $b_0>0$, taking $\varphi_{n\Delta_n}(b_0)=\frac{1}{\sqrt{n\Delta_n}}$, using \eqref{boundedexp} and Lemma \ref{moment} \textnormal{(i)}, we get $\widehat{\E}^{b_0}[\vert\sum_{k=0}^{n-1}\widehat{\E}^{b_0}[\zeta_{k,n}^2\vert \widehat{\mathcal{F}}_{t_k}]\vert]\leq C\Delta_n^{\frac{12}{9+\sqrt{33}}}$, which tends to zero. Here notice that $\beta(\frac{\frac{2a}{\sigma^2}-1}{2q_0}+5)<\frac{2a}{\sigma^2}$.
	
	For critical case $b_0=0$, taking $\varphi_{n\Delta_n}(b_0)=\frac{1}{n\Delta_n}$, using \eqref{boundedexp} and Lemma \ref{auxi} \textnormal{(i)} with $p\in\{-1, \frac{2a/\sigma^2-1}{2}+5, -\frac{3}{2}\beta, -2\beta,  \beta(\frac{2a/\sigma^2-1}{2q_0}+5)\}$, we get 
	\begin{align*}
	&\widehat{\E}^{0}\Big[\big\vert\sum_{k=0}^{n-1}\widehat{\E}^{0}[\zeta_{k,n}^2\vert \widehat{\mathcal{F}}_{t_k}]\big\vert\Big]\leq  Cu^2\frac{1}{(n\Delta_n)^2}\Delta_n^{\frac{12}{9+\sqrt{33}}}\bigg((n\Delta_n)^2+\dfrac{1}{(n\Delta_n)^{\frac{\frac{2a}{\sigma^2}-1}{2}+4}}\bigg)\\
	&+C\frac{\vert u\vert^3}{(n\Delta_n)^3}\Delta_n^{\frac{1}{2}+\frac{1}{q_0\overline{p}}}\bigg(n\Delta_n+(n\Delta_n)^{\frac{3}{2}\beta+1}+(\sqrt{\Delta_n}\vert u\vert\frac{1}{n\Delta_n})^{\beta} (n\Delta_n)^{2\beta+1}+\dfrac{1}{(n\Delta_n)^{\beta(\frac{\frac{2a}{\sigma^2}-1}{2q_0}+5)-1}}\bigg)\\
	&\leq C\Delta_n^{\frac{12}{9+\sqrt{33}}},
	\end{align*}
	which tends to zero by choosing $\beta$ close to $1$. 
	
	For supercritical case $b_0<0$, taking   $\varphi_{n\Delta_n}(b_0)=e^{b_0\frac{n\Delta_n}{2}}$, using \eqref{boundedexp} and Lemma \ref{auxi} \textnormal{(ii)} with $p\in\{-1, \frac{2a/\sigma^2-1}{2}+5, -\frac{3}{2}\beta, -2\beta,  \beta(\frac{2a/\sigma^2-1}{2q_0}+5)\}$, we get 
	\begin{align*}
	&\widehat{\E}^{b_0}\Big[\big\vert\sum_{k=0}^{n-1}\widehat{\E}^{b_0}[\zeta_{k,n}^2\vert \widehat{\mathcal{F}}_{t_k}]\big\vert\Big]\leq  Cu^2e^{b_0n\Delta_n}\Delta_n^{\frac{12}{9+\sqrt{33}}}\bigg(e^{-b_0n\Delta_n}+\dfrac{1}{(n\Delta_n)^{\frac{\frac{2a}{\sigma^2}-1}{2}+4}}\bigg)\\
	&+C\vert u\vert^3e^{b_0\frac{3n\Delta_n}{2}}\Delta_n^{\frac{1}{2}+\frac{1}{q_0\overline{p}}}\bigg(n\Delta_n+e^{-\frac{3}{2}\beta b_0n\Delta_n}+(\sqrt{\Delta_n}\vert u\vert e^{b_0\frac{n\Delta_n}{2}})^{\beta} e^{-2\beta b_0n\Delta_n}\\
	&\qquad+\dfrac{1}{(n\Delta_n)^{\beta(\frac{\frac{2a}{\sigma^2}-1}{2q_0}+5)-1}}\bigg)\\
	&\leq C\left(\Delta_n^{\frac{12}{9+\sqrt{33}}}+e^{b_0\frac{3n\Delta_n}{2}}\Delta_n^{\frac{1}{2}+\frac{1}{q_0\overline{p}}}e^{-\frac{3}{2}\beta b_0n\Delta_n}\right)\\
	&= C\left(\Delta_n^{\frac{12}{9+\sqrt{33}}}+(\Delta_n^{2+\frac{7q_0\overline{p}+2-6q_0\overline{p}\beta}{3q_0\overline{p}(\beta-1)}}e^{-b_0n\Delta_n})^{\frac{3}{2}(\beta-1)}\right),
\end{align*}
which tends to zero by choosing $\beta\in (1,\frac{7q_0\overline{p}+2}{6q_0\overline{p}})$. This completes the proof.
\end{color}	 
\end{proof}
\begin{lemma}\label{lemma2} Assume conditions {\bf(A1)} and {\bf(A2)}. Then, as $n\to\infty$,
\begin{align*}
\sum_{k=0}^{n-1}\frac{\varphi_{n\Delta_n}(b_0)}{\Delta_n}\int_0^1\Big(H_8^{b_0}-\widetilde{\E}_{t_k,Y_{t_{k}}^{b_0}}^{b(\ell)}\big[H_5^{b(\ell)}\big\vert X_{t_{k+1}}^{b(\ell)}=Y_{t_{k+1}}^{b_0}\big]\Big)d\ell\overset{\widehat{\P}^{b_0}}{\longrightarrow}0.
\end{align*}
\end{lemma}
\begin{proof}
We apply Lemma \ref{zero} with
$$
\zeta_{k,n}:=\frac{\varphi_{n\Delta_n}(b_0)}{\Delta_n}\int_0^1\Big(H_8^{b_0}-\widetilde{\E}_{t_k,Y_{t_{k}}^{b_0}}^{b(\ell)}\big[H_5^{b(\ell)}\big\vert X_{t_{k+1}}^{b(\ell)}=Y_{t_{k+1}}^{b_0}\big]\Big)d\ell.
$$
Using \eqref{for1} of Lemma \ref{change} with $V=H_5^{b(\ell)}$,  $\widehat{\E}_{t_k,Y_{t_k}^{b_0}}^{b_0}[H^{b_0}_{8}]=0$ and $\widetilde{\E}_{t_k,Y_{t_{k}}^{b_0}}^{b(\ell)}[H_5^{b(\ell)}]=0$, \eqref{for3} of Lemma \ref{deviation1} with $q=2$, we obtain for $n$ large enough,
\begin{align*} 
\Big\vert\sum_{k=0}^{n-1}\widehat{\E}^{b_0}[\zeta_{k,n}\vert \widehat{\mathcal{F}}_{t_k}]\Big\vert &=\left\vert\sum_{k=0}^{n-1}\frac{\varphi_{n\Delta_n}(b_0)}{\Delta_n}\int_0^1\widehat{\E}_{t_k,Y_{t_k}^{b_0}}^{b_0}\left[\widetilde{\E}_{t_k,Y_{t_{k}}^{b_0}}^{b(\ell)}\left[H_5^{b(\ell)}\vert X_{t_{k+1}}^{b(\ell)}=Y_{t_{k+1}}^{b_0}\right]\right]d\ell\right\vert\\
&\leq C\dfrac{\vert u\vert(\varphi_{n\Delta_n}(b_0))^2 }{\sqrt{\Delta_n}}\sum_{k=0}^{n-1}\int_0^1\big(\widetilde{\E}_{t_k,Y_{t_k}^{b_0}}^{b(\ell)}[\vert H_5^{b(\ell)}\vert^2]\big)^{\frac{1}{2}}e^{C_2u^2(\varphi_{n\Delta_n}(b_0))^2\Delta_nY_{t_k}^{b_0}} \\
&\qquad\times\big(1+\sqrt{Y_{t_k}^{b_0}}+\sqrt{\Delta_n}\vert u\vert\varphi_{n\Delta_n}(b_0) Y_{t_k}^{b_0}\big)d\ell.
\end{align*}	
Now, using Burkholder-Davis-Gundy's (BDG's) inequality, It\^o's formula, H\"older's inequality, \eqref{e1}, \eqref{e2}, we get for $p\geq 1$, 
\begin{align}\label{sqrtincre}
\widetilde{\E}_{t_k,Y_{t_k}^{b_0}}^{b}\big[\vert H_5^{b}\vert^{2p}\big]&\leq C\Delta_n^{2p}\Delta_n^{p-1}\int_{t_k}^{t_{k+1}}\widetilde{\E}_{t_k,Y_{t_k}^{b_0}}^{b}\big[\vert\sqrt{X_s^{b}}-\sqrt{X_{t_k}^{b}}\vert^{2p}\big]ds\notag\\
&\leq C\Delta_n^{3p-1}\int_{t_k}^{t_{k+1}}\widetilde{\E}_{t_k,Y_{t_k}^{b_0}}^{b}\Big[\Big\vert\int_{t_k}^{s}\big((\frac{a}{2}-\frac{\sigma^2}{8})\frac{1}{\sqrt{X_u^{b}}}-\frac{b}{2}\sqrt{X_u^{b}}\big)du+\frac{\sigma}{2}\int_{t_k}^{s}dB_u\notag\\
&\qquad+\int_{t_k}^{s}\int_0^{\infty}(\sqrt{X_{u-}^{b}+z}-\sqrt{X_{u-}^{b}})M(du,dz)\Big\vert^{2p}\Big]ds\notag\\
&\leq C\Delta_n^{3p-1}\int_{t_k}^{t_{k+1}}\Big((s-t_k)^{2p-1}\int_{t_k}^{s}\big(\widetilde{\E}_{t_k,Y_{t_k}^{b_0}}^{b}\big[\frac{1}{(X_u^{b})^p}\big]+\vert b\vert\widetilde{\E}_{t_k,Y_{t_k}^{b_0}}^{b}[(X_u^{b})^p]\big)du\notag\\
&\qquad+(s-t_k)^p+\int_0^{\infty}z^{2p}m(dz)\int_{t_k}^{s}\widetilde{\E}_{t_k,Y_{t_k}^{b_0}}^{b}\big[\frac{1}{(X_u^{b})^p}\big]du\Big)ds\notag\\
&\leq  C\Delta_n^{3p+1}\Big(\frac{1}{(Y_{t_k}^{b_0})^{p}}+\vert b\vert(Y_{t_k}^{b_0})^p+1\Big),
\end{align}
provided that $p<\frac{2a}{\sigma^2}-1$. Then, applying \eqref{sqrtincre} with $p=1$, condition $\frac{a}{\sigma^2}>1$, we get
\begin{align*} 
&\Big\vert\sum_{k=0}^{n-1}\widehat{\E}^{b_0}[\zeta_{k,n}\vert \widehat{\mathcal{F}}_{t_k}]\Big\vert\\
&\leq C\dfrac{\vert u\vert(\varphi_{n\Delta_n}(b_0))^2 }{\sqrt{\Delta_n}}\sum_{k=0}^{n-1}\Delta_n^{2}\Big(\frac{1}{\sqrt{Y_{t_k}^{b_0}}}+(\vert b_0\vert+\vert u\vert \varphi_{n\Delta_n}(b_0))\sqrt{Y_{t_k}^{b_0}}+1\Big) \\
&\qquad\times e^{C_2u^2(\varphi_{n\Delta_n}(b_0))^2\Delta_nY_{t_k}^{b_0}}\big(1+\sqrt{Y_{t_k}^{b_0}}+\sqrt{\Delta_n}\vert u\vert\varphi_{n\Delta_n}(b_0) Y_{t_k}^{b_0}\big)\\
&\leq C\dfrac{\vert u\vert(\varphi_{n\Delta_n}(b_0))^2 }{\sqrt{\Delta_n}}\sum_{k=0}^{n-1}\Delta_n^{2}e^{C_2u^2(\varphi_{n\Delta_n}(b_0))^2\Delta_nY_{t_k}^{b_0}}\Big(\frac{1}{\sqrt{Y_{t_k}^{b_0}}}+(\vert b_0\vert+\vert u\vert \varphi_{n\Delta_n}(b_0))Y_{t_k}^{b_0}\\
&\qquad+(\vert b_0\vert+\vert u\vert \varphi_{n\Delta_n}(b_0))\sqrt{\Delta_n}\vert u\vert\varphi_{n\Delta_n}(b_0) (Y_{t_k}^{b_0})^{\frac{3}{2}}\Big). 
\end{align*}
Then, using Young's inequality for products with $\frac{1}{\alpha}+\frac{1}{\beta}=1$ and $\beta$ close to $1$,
\begin{align*} 
	&\Big\vert\sum_{k=0}^{n-1}\widehat{\E}^{b_0}[\zeta_{k,n}\vert \widehat{\mathcal{F}}_{t_k}]\Big\vert\leq C\vert u\vert(\varphi_{n\Delta_n}(b_0))^2 \Delta_n^{\frac{3}{2}}\bigg(\frac{1}{\alpha}\sum_{k=0}^{n-1}e^{C_2\alpha u^2(\varphi_{n\Delta_n}(b_0))^2\Delta_nY_{t_k}^{b_0}} +\frac{1}{\beta}\sum_{k=0}^{n-1}\Big(\frac{1}{(Y_{t_k}^{b_0})^{\frac{\beta}{2}}}\\
	&\qquad+(\vert b_0\vert+\vert u\vert \varphi_{n\Delta_n}(b_0))^{\beta}(Y_{t_k}^{b_0})^{\beta}+((\vert b_0\vert+\vert u\vert \varphi_{n\Delta_n}(b_0))\sqrt{\Delta_n}\vert u\vert\varphi_{n\Delta_n}(b_0))^{\beta} (Y_{t_k}^{b_0})^{\frac{3}{2}\beta}\Big)\bigg). 
\end{align*}
Thus, for subcritical case $b_0>0$ (respectively critical case $b_0=0$, respectively supercritical case $b_0<0$), using \eqref{boundedexp}, Lemma \ref{moment} \textnormal{(i)} (respectively \textcolor{black}{Lemma \ref{auxi} \textnormal{(i)}}, respectively \textcolor{black}{Lemma \ref{auxi} \textnormal{(ii)}}), we get $\widehat{\E}^{b_0}[\vert\sum_{k=0}^{n-1}\widehat{\E}^{b_0}[\zeta_{k,n}\vert \widehat{\mathcal{F}}_{t_k}]\vert]\leq C\sqrt{\Delta_n}$ (respectively $\widehat{\E}^{0}[\vert\sum_{k=0}^{n-1}\widehat{\E}^{0}[\zeta_{k,n}\vert \widehat{\mathcal{F}}_{t_k}]\vert]\leq C\sqrt{\Delta_n}$, respectively $\widehat{\E}^{b_0}[\vert\sum_{k=0}^{n-1}\widehat{\E}^{b_0}[\zeta_{k,n}\vert \widehat{\mathcal{F}}_{t_k}]\vert]\leq C(\Delta_n^{2+\frac{5-4\beta}{2(\beta-1)}}e^{-b_0n\Delta_n})^{\beta-1}$), which tends to zero by choosing $\beta\in (1,\frac{5}{4})$ in the supercritical case. Here notice that $\frac{\beta}{2}<\frac{2a}{\sigma^2}$. 

Similarly, we use Jensen's inequality, \eqref{for1} of Lemma \ref{change} with $V=(H_5^{b(\ell)})^2$, \eqref{for3} of Lemma \ref{deviation1} with $q=2$, \eqref{sqrtincre} with $p\in\{1,2\}$, and condition $\frac{a}{\sigma^2}>\frac{3}{2}$, we obtain for $n$ large enough, 
\begin{align*}
&\sum_{k=0}^{n-1}\widehat{\E}^{b_0}[\zeta_{k,n}^2\vert \widehat{\mathcal{F}}_{t_k}]\\
&\leq 2\frac{(\varphi_{n\Delta_n}(b_0))^2}{\Delta_n^2}\sum_{k=0}^{n-1}\int_0^1\bigg\{\widehat{\E}_{t_k,Y_{t_k}^{b_0}}^{b_0}\Big[(H^{b_0}_{8})^2\Big]+\widehat{\E}_{t_k,Y_{t_k}^{b_0}}^{b_0}\Big[\widetilde{\E}_{t_k,Y_{t_{k}}^{b_0}}^{b(\ell)}\big[(H_5^{b(\ell)})^2\vert X_{t_{k+1}}^{b(\ell)}=Y_{t_{k+1}}^{b_0}\big]\Big]\bigg\}d\ell\\	
&\leq 2\frac{(\varphi_{n\Delta_n}(b_0))^2}{\Delta_n^2}\sum_{k=0}^{n-1}\int_0^1\bigg\{\widehat{\E}_{t_k,Y_{t_k}^{b_0}}^{b_0}[(H^{b_0}_{8})^2]+\widetilde{\E}_{t_k,Y_{t_k}^{b_0}}^{b(\ell)}[(H_5^{b(\ell)})^2]+C_1\sqrt{\Delta_n}\big(\widetilde{\E}_{t_k,Y_{t_k}^{b_0}}^{b(\ell)}\big[\vert H_5^{b(\ell)}\vert^{4}\big]\big)^{\frac{1}{2}} \\
&\qquad\times e^{C_2u^2(\varphi_{n\Delta_n}(b_0))^2\Delta_nY_{t_k}^{b_0}} \vert u\vert\varphi_{n\Delta_n}(b_0)\big(1+\sqrt{Y_{t_k}^{b_0}}+\sqrt{\Delta_n}\vert u\vert\varphi_{n\Delta_n}(b_0) Y_{t_k}^{b_0}\big)\bigg\}d\ell\\
&\leq C\frac{(\varphi_{n\Delta_n}(b_0))^2}{\Delta_n^2}\sum_{k=0}^{n-1}\bigg\{\Delta_n^{4}\Big(\frac{1}{Y_{t_k}^{b_0}}+(\vert b_0\vert+\vert u\vert \varphi_{n\Delta_n}(b_0)) Y_{t_k}^{b_0}+1\Big)+C_1\sqrt{\Delta_n}\Delta_n^{\frac{7}{2}}\Big(\frac{1}{Y_{t_k}^{b_0}}\\
&\qquad+(\vert b_0\vert+\vert u\vert \varphi_{n\Delta_n}(b_0)) Y_{t_k}^{b_0}+1\Big) e^{C_2u^2(\varphi_{n\Delta_n}(b_0))^2\Delta_nY_{t_k}^{b_0}} \vert u\vert\varphi_{n\Delta_n}(b_0)\big(1+\sqrt{Y_{t_k}^{b_0}}+\sqrt{\Delta_n}\vert u\vert\\
&\qquad \times \varphi_{n\Delta_n}(b_0) Y_{t_k}^{b_0}\big)\bigg\}\\
&\leq C\frac{(\varphi_{n\Delta_n}(b_0))^2}{\Delta_n^2}\sum_{k=0}^{n-1}\bigg\{\Delta_n^{4}\Big(\frac{1}{Y_{t_k}^{b_0}}+(\vert b_0\vert+\vert u\vert \varphi_{n\Delta_n}(b_0)) Y_{t_k}^{b_0}+1\Big)\\
&\qquad+C_1\vert u\vert\Delta_n^{4}\varphi_{n\Delta_n}(b_0)e^{C_2u^2(\varphi_{n\Delta_n}(b_0))^2\Delta_nY_{t_k}^{b_0}}\Big(\frac{1}{Y_{t_k}^{b_0}}+(\vert b_0\vert+\vert u\vert \varphi_{n\Delta_n}(b_0))(Y_{t_k}^{b_0})^{\frac{3}{2}}\\
&\qquad+(\vert b_0\vert+\vert u\vert \varphi_{n\Delta_n}(b_0))\sqrt{\Delta_n}\vert u\vert\varphi_{n\Delta_n}(b_0) (Y_{t_k}^{b_0})^{2}\Big) \bigg\},
\end{align*}
where the estimate for $\widehat{\E}_{t_k,Y_{t_k}^{b_0}}^{b_0}[(H^{b_0}_{8})^2]$ is proceeded similarly as for \eqref{sqrtincre}. Then, using Young's inequality for products with $\frac{1}{\alpha}+\frac{1}{\beta}=1$ and $\beta$ close to $1$,
\begin{align*}
	&\sum_{k=0}^{n-1}\widehat{\E}^{b_0}[\zeta_{k,n}^2\vert \widehat{\mathcal{F}}_{t_k}]\leq C(\varphi_{n\Delta_n}(b_0))^2\Delta_n^2\sum_{k=0}^{n-1}\Big(\frac{1}{Y_{t_k}^{b_0}}+(\vert b_0\vert+\vert u\vert \varphi_{n\Delta_n}(b_0)) Y_{t_k}^{b_0}+1\Big)\\
	&\qquad+C\vert u\vert(\varphi_{n\Delta_n}(b_0))^3\Delta_n^2\bigg(\frac{1}{\alpha}\sum_{k=0}^{n-1}e^{C_2\alpha u^2(\varphi_{n\Delta_n}(b_0))^2\Delta_nY_{t_k}^{b_0}}+\frac{1}{\beta}\sum_{k=0}^{n-1}\Big(\frac{1}{(Y_{t_k}^{b_0})^{\beta}}\\
	&\qquad+(\vert b_0\vert+\vert u\vert \varphi_{n\Delta_n}(b_0))^{\beta}(Y_{t_k}^{b_0})^{\frac{3}{2}\beta}+((\vert b_0\vert+\vert u\vert \varphi_{n\Delta_n}(b_0))\sqrt{\Delta_n}\vert u\vert\varphi_{n\Delta_n}(b_0))^{\beta} (Y_{t_k}^{b_0})^{2\beta}\Big)\bigg). 
\end{align*}
Thus, for subcritical case $b_0>0$ (respectively critical case $b_0=0$, respectively supercritical case $b_0<0$), using \eqref{boundedexp}, Lemma \ref{moment} \textnormal{(i)} (respectively \textcolor{black}{Lemma \ref{auxi} \textnormal{(i)}}, respectively \textcolor{black}{Lemma \ref{auxi} \textnormal{(ii)}}) and standard calculations, we get $\widehat{\E}^{b_0}[\vert\sum_{k=0}^{n-1}\widehat{\E}^{b_0}[\zeta_{k,n}^2\vert \widehat{\mathcal{F}}_{t_k}]\vert]\leq C\Delta_n$ (respectively $\widehat{\E}^{0}[\vert\sum_{k=0}^{n-1}\widehat{\E}^{0}[\zeta_{k,n}^2\vert \widehat{\mathcal{F}}_{t_k}]\vert]\leq C\Delta_n$, respectively $\widehat{\E}^{b_0}[\vert\sum_{k=0}^{n-1}\widehat{\E}^{b_0}[\zeta_{k,n}^2\vert \widehat{\mathcal{F}}_{t_k}]\vert]\leq C(\Delta_n^{2+\frac{8-6\beta}{3(\beta-1)}}e^{-b_0n\Delta_n})^{\frac{3}{2}(\beta-1)}$), which tends to zero by choosing $\beta\in (1,\frac{4}{3})$ in the supercritical case. Here notice that $\beta<\frac{2a}{\sigma^2}$. This completes the proof.	
\end{proof}

\begin{lemma}\label{lemma3} Assume conditions {\bf(A1)} and {\bf(A2)}. Then, as $n\to\infty$,
\begin{align*}
\sum_{k=0}^{n-1}\frac{\varphi_{n\Delta_n}(b_0)}{\Delta_n}\int_0^1\Big(H_7^{b_0}-\widetilde{\E}_{t_k,Y_{t_{k}}^{b_0}}^{b(\ell)}\big[H_4^{b(\ell)}\big\vert X_{t_{k+1}}^{b(\ell)}=Y_{t_{k+1}}^{b_0}\big]\Big)d\ell\overset{\widehat{\P}^{b_0}}{\longrightarrow}0.
\end{align*}
\end{lemma}
\begin{proof}
	We write 
	\begin{align*}
		\frac{\varphi_{n\Delta_n}(b_0)}{\Delta_n}\int_0^1\Big(H_7^{b_0}-\widetilde{\E}_{t_k,Y_{t_{k}}^{b_0}}^{b(\ell)}\big[H_4^{b(\ell)}\big\vert X_{t_{k+1}}^{b(\ell)}=Y_{t_{k+1}}^{b_0}\big]\Big)d\ell=M_{k,n,1}+M_{k,n,2},
	\end{align*}	
	where	
	\begin{align*}
		M_{k,n,1}&=-\dfrac{u}{2\sigma^2 }(\varphi_{n\Delta_n}(b_0))^2\int_{t_k}^{t_{k+1}}(Y_s^{b_0}-Y_{t_k}^{b_0})ds,\\
		M_{k,n,2}&=\dfrac{1}{\sigma^2}\varphi_{n\Delta_n}(b_0)\int_0^1b(\ell)\Big(\int_{t_k}^{t_{k+1}}(Y_s^{b_0}-Y_{t_k}^{b_0})ds\\
		&\qquad-\widetilde{\E}_{t_k,Y_{t_{k}}^{b_0}}^{b(\ell)}\Big[\int_{t_k}^{t_{k+1}}(X_s^{b(\ell)}-X_{t_k}^{b(\ell)})ds\vert X_{t_{k+1}}^{b(\ell)}=Y_{t_{k+1}}^{b_0}\Big]\Big)d\ell.
	\end{align*}
First, as for $H_{8,k,n}^{b_0}$ in \eqref{m1}, using equation \eqref{eqintegral} and Lemma \ref{Pro_moments}, we get $\widehat{\E}^{b_0}[\vert\sum_{k=0}^{n-1} M_{k,n,1}\vert]\leq C\sqrt{\Delta_n}$ for $b_0>0$, $b_0=0$ and $b_0<0$, which tends to zero.
Thus, $\sum_{k=0}^{n-1}M_{k,n,1}\overset{\widehat{\P}^{b_0}}{\longrightarrow}0$. 

Next, we wish to show that $\sum_{k=0}^{n-1}M_{k,n,2}\overset{\widehat{\P}^{b_0}}{\longrightarrow}0$ by applying Lemma \ref{zero}. For this, using \eqref{for1} of Lemma \ref{change} with $V=V^{b(\ell)}=\int_{t_k}^{t_{k+1}}(X_s^{b(\ell)}-X_{t_k}^{b(\ell)})ds$, \eqref{for2} of Lemma \ref{change} with $\widehat{V}=\widehat{V}^{b_0}=\int_{t_k}^{t_{k+1}}(Y_s^{b_0}-Y_{t_k}^{b_0})ds$, \eqref{for3} and \eqref{for4} of Lemma \ref{deviation1} with $q=2$, we get for $n$ large enough,
\begin{align*}
	&\Big\vert\sum_{k=0}^{n-1}\widehat{\E}^{b_0}[M_{k,n,2}\vert \widehat{\mathcal{F}}_{t_k}]\Big\vert \leq C\vert u\vert(\varphi_{n\Delta_n}(b_0))^2\sqrt{\Delta_n}\sum_{k=0}^{n-1}\int_0^1\vert b(\ell)\vert  e^{C_2u^2(\varphi_{n\Delta_n}(b_0))^2\Delta_nY_{t_k}^{b_0}}\\
	&\qquad\times  \big(1+\sqrt{Y_{t_k}^{b_0}}+\sqrt{\Delta_n}\vert u\vert\varphi_{n\Delta_n}(b_0) Y_{t_k}^{b_0}\big)\Big(\Big(\widehat{\E}_{t_k,Y_{t_k}^{b_0}}^{b(\ell)}\big[\vert \widehat{V}^{b(\ell)}\vert^{2}\big]\Big)^{\frac{1}{2}}+\Big(\widetilde{\E}_{t_k,Y_{t_k}^{b_0}}^{b(\ell)}\big[\vert V^{b(\ell)}\vert^{2}\big]\Big)^{\frac{1}{2}}\Big)d\ell,
\end{align*}
where we use $\widehat{\E}_{t_k,Y_{t_k}^{b_0}}^{b(\ell)}[\widehat{V}^{b(\ell)}]-\widetilde{\E}_{t_k,Y_{t_k}^{b_0}}^{b(\ell)}[V^{b(\ell)}]=0$ since $X^{b(\ell)}$ is the independent copy of $Y^{b(\ell)}$.

Now, using equation \eqref{c2eq1rajoute}, BDG's inequality, H\"older's inequality and \eqref{e1}, we get for $p\geq 1$,
\begin{align}\label{incrementY}
	&\widetilde{\E}_{t_k,Y_{t_k}^{b_0}}^{b}\Big[\Big\vert \int_{t_k}^{t_{k+1}}(X_s^{b}-X_{t_k}^{b})ds\Big\vert^{2p}\Big] \notag\\
	&\leq \Delta_n^{2p-1}\int_{t_k}^{t_{k+1}}\widetilde{\E}_{t_k,Y_{t_k}^{b_0}}^{b}\Big[\big\vert X_s^{b}-X_{t_k}^{b}\big\vert^{2p}\Big]ds\notag\\
	&\leq C\Delta_n^{2p-1}\int_{t_k}^{t_{k+1}}\Big((s-t_k)^{2p-1}\int_{t_k}^{s}\big(1+\vert b\vert\widetilde{\E}_{t_k,Y_{t_k}^{b_0}}^{b}[\vert X_u^{b} \vert^{2p}]\big)du\notag\\
	&\qquad+(s-t_k)^{p-1}\int_{t_k}^{s}\widetilde{\E}_{t_k,Y_{t_k}^{b_0}}^{b}[\vert X_u^{b} \vert^{p}]du+(s-t_k)\int_0^{\infty}z^{2p}m(dz)\Big)ds\notag\\
	&\leq C\Delta_n^{2p+1}\big(1+\vert b\vert(Y_{t_k}^{b_0})^{2p}+(Y_{t_k}^{b_0})^{p}\big).
\end{align}
Then, applying \eqref{incrementY} with $p=1$, we get
\begin{align*}
	&\Big\vert\sum_{k=0}^{n-1}\widehat{\E}^{b_0}[M_{k,n,2}\vert \widehat{\mathcal{F}}_{t_k}]\Big\vert\\
	& \leq C\vert u\vert(\varphi_{n\Delta_n}(b_0))^2\Delta_n^{2}\sum_{k=0}^{n-1} e^{C_2u^2(\varphi_{n\Delta_n}(b_0))^2\Delta_nY_{t_k}^{b_0}}\big(1+\sqrt{Y_{t_k}^{b_0}}\\
	&\qquad  +\sqrt{\Delta_n}\vert u\vert\varphi_{n\Delta_n}(b_0) Y_{t_k}^{b_0}\big)\big(1+(\vert b_0\vert+\vert u\vert \varphi_{n\Delta_n}(b_0)) Y_{t_k}^{b_0}+\sqrt{Y_{t_k}^{b_0}}\big)\\
	&\leq C\vert u\vert(\varphi_{n\Delta_n}(b_0))^2\Delta_n^{2}\sum_{k=0}^{n-1} e^{C_2u^2(\varphi_{n\Delta_n}(b_0))^2\Delta_nY_{t_k}^{b_0}}\Big(1+(\vert b_0\vert+\vert u\vert \varphi_{n\Delta_n}(b_0))(Y_{t_k}^{b_0})^{\frac{3}{2}}\\
	&\qquad  +(\vert b_0\vert+\vert u\vert \varphi_{n\Delta_n}(b_0))\sqrt{\Delta_n}\vert u\vert\varphi_{n\Delta_n}(b_0) (Y_{t_k}^{b_0})^{2}\Big).
\end{align*}
Then, using Young's inequality for products with $\frac{1}{\alpha}+\frac{1}{\beta}=1$ and $\beta$ close to $1$,
\begin{align*}
	&\Big\vert\sum_{k=0}^{n-1}\widehat{\E}^{b_0}[M_{k,n,2}\vert \widehat{\mathcal{F}}_{t_k}]\Big\vert \leq  C\vert u\vert(\varphi_{n\Delta_n}(b_0))^2\Delta_n^{2}\bigg(\frac{1}{\alpha}\sum_{k=0}^{n-1} e^{C_2\alpha u^2(\varphi_{n\Delta_n}(b_0))^2\Delta_nY_{t_k}^{b_0}}+\frac{1}{\beta}\sum_{k=0}^{n-1} \Big(1\\
	&+(\vert b_0\vert+\vert u\vert \varphi_{n\Delta_n}(b_0))^{\beta}(Y_{t_k}^{b_0})^{\frac{3}{2}\beta}+((\vert b_0\vert+\vert u\vert \varphi_{n\Delta_n}(b_0))\sqrt{\Delta_n}\vert u\vert\varphi_{n\Delta_n}(b_0))^{\beta} (Y_{t_k}^{b_0})^{2\beta}\Big)\bigg).
\end{align*}
Thus, for subcritical case $b_0>0$ (respectively critical case $b_0=0$, respectively supercritical case $b_0<0$), using \eqref{boundedexp}, Lemma \ref{moment} \textnormal{(i)} (respectively \textcolor{black}{Lemma \ref{auxi} \textnormal{(i)}}, respectively \textcolor{black}{Lemma \ref{auxi} \textnormal{(ii)}}), we get $\widehat{\E}^{b_0}[\vert\sum_{k=0}^{n-1}\widehat{\E}^{b_0}[M_{k,n,2}\vert \widehat{\mathcal{F}}_{t_k}]\vert]\leq C\Delta_n$ (respectively $\widehat{\E}^{0}[\vert\sum_{k=0}^{n-1}\widehat{\E}^{0}[M_{k,n,2}\vert \widehat{\mathcal{F}}_{t_k}]\vert]\leq C\Delta_n$, respectively $\widehat{\E}^{b_0}[\vert\sum_{k=0}^{n-1}\widehat{\E}^{b_0}[M_{k,n,2}\vert \widehat{\mathcal{F}}_{t_k}]\vert]\leq C(\Delta_n^{2-\frac{6(\beta-1)}{3\beta-2}}e^{-b_0n\Delta_n})^{\frac{3}{2}\beta-1}$), which tends to zero by choosing $\beta$ close to $1$ in the supercritical case. This is because $\Delta_n^{2-\varepsilon}e^{-b_0n\Delta_n}\to 0$ for the supercritical case where $\varepsilon>0$ is arbitrarily small.

Next, using Jensen's inequality, \eqref{for1} of Lemma \ref{change} with $(V^{b(\ell)})^2$, \eqref{for3} of Lemma \ref{deviation1} with $q=2$, \eqref{incrementY} with $p\in\{1,2\}$, we get for $n$ large enough,
\begin{align*}
	&\sum_{k=0}^{n-1}\widehat{\E}^{b_0}[M_{k,n,2}^2\vert \widehat{\mathcal{F}}_{t_k}]\\
	&\leq C(\varphi_{n\Delta_n}(b_0))^2\sum_{k=0}^{n-1}\int_0^1b(\ell)^2\bigg\{\widehat{\E}_{t_k,Y_{t_k}^{b_0}}^{b_0}[(\widehat{V}^{b_0})^2]+\widetilde{\E}_{t_k,Y_{t_k}^{b_0}}^{b(\ell)}[(V^{b(\ell)})^2]\\
	&\qquad+C_1\sqrt{\Delta_n}\big(\widetilde{\E}_{t_k,Y_{t_k}^{b_0}}^{b(\ell)}[\vert V^{b(\ell)}\vert^4]\big)^{\frac{1}{2}}e^{C_2u^2(\varphi_{n\Delta_n}(b_0))^2\Delta_nY_{t_k}^{b_0}}\vert u\vert\varphi_{n\Delta_n}(b_0)\big(1+\sqrt{Y_{t_k}^{b_0}} +\sqrt{\Delta_n}\vert u\vert\\
	 &\qquad \times \varphi_{n\Delta_n}(b_0)Y_{t_k}^{b_0}\big) \bigg\}d\ell\\
	 &\leq C(\varphi_{n\Delta_n}(b_0))^2\sum_{k=0}^{n-1}\bigg\{\Delta_n^{3}\big(1+(\vert b_0\vert+\vert u\vert \varphi_{n\Delta_n}(b_0))(Y_{t_k}^{b_0})^{2}+Y_{t_k}^{b_0}\big)\\
	&\qquad+C_1\vert u\vert\Delta_n^3 \varphi_{n\Delta_n}(b_0)\big(1+(\vert b_0\vert+\vert u\vert \varphi_{n\Delta_n}(b_0))(Y_{t_k}^{b_0})^{2}+Y_{t_k}^{b_0}\big)e^{C_2u^2(\varphi_{n\Delta_n}(b_0))^2\Delta_nY_{t_k}^{b_0}}\\
	&\qquad\times \big(1+\sqrt{Y_{t_k}^{b_0}}+\sqrt{\Delta_n}\vert u\vert\varphi_{n\Delta_n}(b_0) Y_{t_k}^{b_0}\big) \bigg\} \\
	&\leq C(\varphi_{n\Delta_n}(b_0))^2\sum_{k=0}^{n-1}\bigg\{\Delta_n^{3}\big(1+Y_{t_k}^{b_0}+(\vert b_0\vert+\vert u\vert \varphi_{n\Delta_n}(b_0))(Y_{t_k}^{b_0})^{2}\big) \\
	&\qquad+C_1\vert u\vert\Delta_n^3 \varphi_{n\Delta_n}(b_0)e^{C_2u^2(\varphi_{n\Delta_n}(b_0))^2\Delta_nY_{t_k}^{b_0}}\\
	&\qquad\times  \Big(1+(\vert b_0\vert+\vert u\vert \varphi_{n\Delta_n}(b_0))(Y_{t_k}^{b_0})^{\frac{5}{2}}+(\vert b_0\vert+\vert u\vert \varphi_{n\Delta_n}(b_0))\sqrt{\Delta_n}\vert u\vert\varphi_{n\Delta_n}(b_0) (Y_{t_k}^{b_0})^{3}\Big) \bigg\}.
\end{align*}
Then, using Young's inequality for products with $\frac{1}{\alpha}+\frac{1}{\beta}=1$ and $\beta$ close to $1$, 
\begin{align*}
	&\sum_{k=0}^{n-1}\widehat{\E}^{b_0}[M_{k,n,2}^2\vert \widehat{\mathcal{F}}_{t_k}]\leq  C(\varphi_{n\Delta_n}(b_0))^2\Delta_n^{3}\sum_{k=0}^{n-1}\big(1+(\vert b_0\vert+\vert u\vert \varphi_{n\Delta_n}(b_0))(Y_{t_k}^{b_0})^{2}+Y_{t_k}^{b_0}\big)\\
	&+C\vert u\vert(\varphi_{n\Delta_n}(b_0))^3\Delta_n^3\bigg(\frac{1}{\alpha}\sum_{k=0}^{n-1}e^{C_2\alpha u^2(\varphi_{n\Delta_n}(b_0))^2\Delta_nY_{t_k}^{b_0}}+\frac{1}{\beta}\sum_{k=0}^{n-1}\Big(1+(\vert b_0\vert+\vert u\vert \varphi_{n\Delta_n}(b_0))^{\beta}\\
	&\qquad\times (Y_{t_k}^{b_0})^{\frac{5}{2}\beta}+((\vert b_0\vert+\vert u\vert \varphi_{n\Delta_n}(b_0))\sqrt{\Delta_n}\vert u\vert\varphi_{n\Delta_n}(b_0))^{\beta} (Y_{t_k}^{b_0})^{3\beta}\Big)\bigg). 
\end{align*}
Thus, for subcritical case $b_0>0$ (respectively critical case $b_0=0$, respectively supercritical case $b_0<0$), using \eqref{boundedexp}, Lemma \ref{moment} \textnormal{(i)} (respectively \textcolor{black}{Lemma \ref{auxi} \textnormal{(i)}}, respectively \textcolor{black}{Lemma \ref{auxi} \textnormal{(ii)}}), we get $\widehat{\E}^{b_0}[\vert\sum_{k=0}^{n-1}\widehat{\E}^{b_0}[M_{k,n,2}^2\vert \widehat{\mathcal{F}}_{t_k}]\vert]\leq C\Delta_n^2$ (respectively $\widehat{\E}^{0}[\vert\sum_{k=0}^{n-1}\widehat{\E}^{0}[M_{k,n,2}^2\vert \widehat{\mathcal{F}}_{t_k}]\vert]\leq C\Delta_n^2$, respectively $\widehat{\E}^{b_0}[\vert\sum_{k=0}^{n-1}\widehat{\E}^{b_0}[M_{k,n,2}^2\vert \widehat{\mathcal{F}}_{t_k}]\vert]\leq C(\Delta_n^{2-\frac{10(\beta-1)}{5\beta-3}}e^{-b_0n\Delta_n})^{\frac{5\beta-3}{2}}$), which tends to zero by choosing $\beta$ close to $1$ in the supercritical case. This completes the proof.	
\end{proof}
Finally, it remains to deal with the jump terms $H_6^{b(\ell)}$ and $H_9^{b_0}$. To treat the subcritical case, some notations are introduced. Let $(\upsilon_n)_{n\geq 1}$ be a positive sequence satisfying $\upsilon_n\to 0$ and $\frac{\Delta_n}{\upsilon_n}\to 0$ as $n\to\infty$. The process $J^{\upsilon_n}=(J_t^{\upsilon_n})_{t\in\R_+}$ defined by $J_t^{\upsilon_n}=\sum_{0\leq s\leq t}\Delta J_s{\bf 1}_{\{\Delta J_s>\upsilon_n\}}$ is a compound Poisson process with intensity of big jumps $\lambda_{\upsilon_n}:=\int_{z>\upsilon_n}m(dz)$ and distribution of big jumps $\frac{{\bf 1}_{z>\upsilon_n}m(dz)}{\lambda_{\upsilon_n}}$. Then, we can split the jump amplitudes of the subordinator $J_t$ into small jumps and big jumps as follows
\begin{align*}
\int_0^t\int_0^{\infty}zN(ds,dz)=\int_0^t\int_{z\leq\upsilon_n}z\widetilde{N}(ds,dz)+t\int_{z\leq\upsilon_n}zm(dz)+\int_0^t\int_{z>\upsilon_n}zN(ds,dz).
\end{align*}
Hence, from \eqref{eqintegral}, for any $t\in\R_+$, we write
\begin{equation}\label{splitY}
\begin{split}
Y_t^b&=y_0+\int_0^t(a-bY_s^b)ds +\sigma\int_0^t\sqrt{Y_s^b}dW_s+\int_0^t\int_{z\leq\upsilon_n}z\widetilde{N}(ds,dz)\\
&\qquad+t\int_{z\leq\upsilon_n}zm(dz)+\int_0^t\int_{z>\upsilon_n}zN(ds,dz).
\end{split}
\end{equation}
We then denote by $N^{\upsilon_n}=(N_t^{\upsilon_n})_{t\in\R_+}$ the Poisson process with intensity $\lambda_{\upsilon_n}$, which counts the big jumps of the compound Poisson process $J^{\upsilon_n}$.

Similarly, the process $\widetilde{J}^{\upsilon_n}=(\widetilde{J}_t^{\upsilon_n})_{t\in\R_+}$ defined by $\widetilde{J}_t^{\upsilon_n}=\sum_{0\leq s\leq t}\Delta \widetilde{J}_s{\bf 1}_{\{\Delta \widetilde{J}_s>\upsilon_n\}}$ is a compound Poisson process with intensity of big jumps $\lambda_{\upsilon_n}$ and distribution of big jumps $\frac{{\bf 1}_{z>\upsilon_n}m(dz)}{\lambda_{\upsilon_n}}$. Then, we can split the jump amplitudes of the subordinator $\widetilde{J}_t$ into small jumps and big jumps as follows
\begin{align*}
\int_0^t\int_0^{\infty}zM(ds,dz)=\int_0^t\int_{z\leq\upsilon_n}z\widetilde{M}(ds,dz)+t\int_{z\leq\upsilon_n}zm(dz)+\int_0^t\int_{z>\upsilon_n}zM(ds,dz).
\end{align*}
Hence, from \eqref{c2eq1rajoute}, for any $t\in\R_+$, we can write
\begin{equation}\label{splitX}
\begin{split}
X_t^b&=y_0+\int_0^t(a-bX_s^b)ds +\sigma\int_0^t\sqrt{X_s^b}dB_s+\int_0^t\int_{z\leq\upsilon_n}z\widetilde{M}(ds,dz)\\
&\qquad+t\int_{z\leq\upsilon_n}zm(dz)+\int_0^t\int_{z>\upsilon_n}zM(ds,dz).
\end{split}
\end{equation}
Let $M^{\upsilon_n}=(M_t^{\upsilon_n})_{t\in\R_+}$ denote the Poisson process with intensity $\lambda_{\upsilon_n}$ counting the big jumps of the compound Poisson process $\widetilde{J}^{\upsilon_n}$. Now, we need to show the following estimate.
\begin{lemma}\label{jumpestimate2} Let $b=b(\ell)$. Assume conditions {\bf(A1)} and {\bf(A2)}. Then, there exist constants $C, C_1, C_2>0$ such that for all $k \in\{0,...,n-1\}$, $q>1$, $q_0>1$, $q_1\in (1,2]$, $p_2, q_2>1$ satisfying $\frac{1}{p_2}+\frac{1}{q_2}=1$, and for $n$ large enough,
	\begin{align*}
	&\widehat{\E}_{t_k,Y_{t_k}^{b_0}}^{b_0}\left[\Big(\int_{t_k}^{t_{k+1}}\int_{0}^{\infty}zN(ds,dz)-\widetilde{\E}_{t_k,Y_{t_{k}}^{b_0}}^{b}\Big[\int_{t_k}^{t_{k+1}}\int_{0}^{\infty}zM(ds,dz)\big\vert X_{t_{k+1}}^{b}=Y_{t_{k+1}}^{b_0}\Big]\Big)^2\right]\\
	&\leq C(1+(Y_{t_k}^{b_0})^2)\Delta_n\bigg(\left(\lambda_{\upsilon_n}\Delta_n\right)^{\frac{1}{q}}+\left(\lambda_{\upsilon_n}\Delta_n\right)^{\frac{1}{q_2}}+\Big((\sqrt{\Delta_n}\varphi_{n\Delta_n}(b_0))^{\frac{1}{q}}\left(\lambda_{\upsilon_n}\Delta_n\right)^{\frac{1}{q_0q}}+\left(\lambda_{\upsilon_n}\Delta_n\right)^{\frac{1}{q_2}}\\
	& \times  (\sqrt{\Delta_n}\varphi_{n\Delta_n}(b_0))^{\frac{1}{p_2}}\Big) e^{C_1u^2(\varphi_{n\Delta_n}(b_0))^2\Delta_nY_{t_k}^{b_0}}\big(1+\sqrt{Y_{t_k}^{b_0}}+\sqrt{\Delta_n}\varphi_{n\Delta_n}(b_0) Y_{t_k}^{b_0}\big)\bigg)\notag\\
	&  +C\Delta_n\bigg(\int_{z\leq\upsilon_n}z^2m(dz)+\Delta_n\Big(\int_{z\leq\upsilon_n}zm(dz)\Big)^2\bigg)+C\varphi_{n\Delta_n}(b_0)\Delta_n^{\frac{1}{2}+\frac{1}{q_1}}\Big(\int_{z\leq\upsilon_n}z^{2q_1}m(dz)\Big)^{\frac{1}{q_1}} \notag\\ 
	&\times e^{C_2u^2(\varphi_{n\Delta_n}(b_0))^2\Delta_nY_{t_k}^{b_0}} \big(1+\sqrt{Y_{t_k}^{b_0}}+\sqrt{\Delta_n}\varphi_{n\Delta_n}(b_0) Y_{t_k}^{b_0}\big).
	\end{align*}
\end{lemma}

\begin{lemma}\label{lemma4} Assume conditions {\bf(A1)} and {\bf(A2)}. Then, as $n\to\infty$,
	\begin{align*}
	\sum_{k=0}^{n-1}\frac{\varphi_{n\Delta_n}(b_0)}{\Delta_n}\int_0^1\Big(H_9^{b_0}-\widetilde{\E}_{t_k,Y_{t_{k}}^{b_0}}^{b(\ell)}\big[H_6^{b(\ell)}\big\vert X_{t_{k+1}}^{b(\ell)}=Y_{t_{k+1}}^{b_0}\big]\Big)d\ell\overset{\widehat{\P}^{b_0}}{\longrightarrow}0.
	\end{align*}
\end{lemma}
\begin{proof} 
	We apply Lemma \ref{zero} with
	$$
	\zeta_{k,n}:=\frac{\varphi_{n\Delta_n}(b_0)}{\Delta_n}\int_0^1\Big(H_9^{b_0}-\widetilde{\E}_{t_k,Y_{t_{k}}^{b_0}}^{b(\ell)}\big[H_6^{b(\ell)}\big\vert X_{t_{k+1}}^{b(\ell)}=Y_{t_{k+1}}^{b_0}\big]\Big)d\ell.
	$$
	First, using \eqref{for1} of Lemma \ref{change} with $V=\int_{t_k}^{t_{k+1}}\int_{0}^{\infty}z\widetilde{M}(ds,dz)$, $\widetilde{\E}_{t_k,Y_{t_{k}}^{b_0}}^{b(\ell)}[V]=0$, \eqref{for3} of Lemma \ref{deviation1} with $q=q_0\in (1,2)$, BDG's inequality and {\bf(A1)}, {\bf(A2)}, we obtain for $n$ large enough,
	\begin{align*} 
		&\Big\vert\sum_{k=0}^{n-1}\widehat{\E}^{b_0}[\zeta_{k,n}\vert \widehat{\mathcal{F}}_{t_k}]\Big\vert\\
		&=\Big\vert\sum_{k=0}^{n-1}\frac{\varphi_{n\Delta_n}(b_0)}{\sigma^2}\int_0^1\widehat{\E}_{t_k,Y_{t_k}^{b_0}}^{b_0}\Big[\widetilde{\E}_{t_k,Y_{t_{k}}^{b_0}}^{b(\ell)}\Big[\int_{t_k}^{t_{k+1}}\int_{0}^{\infty}z\widetilde{M}(ds,dz)\vert X_{t_{k+1}}^{b(\ell)}=Y_{t_{k+1}}^{b_0}\Big]\Big]d\ell\Big\vert\\
		&\leq C\vert u\vert(\varphi_{n\Delta_n}(b_0))^2\sqrt{\Delta_n}\sum_{k=0}^{n-1}\int_0^1\Big(\widetilde{\E}_{t_k,Y_{t_k}^{b_0}}^{b(\ell)}\Big[\Big\vert \int_{t_k}^{t_{k+1}}\int_{0}^{\infty}z\widetilde{M}(ds,dz)\Big\vert^{q_0}\Big]\Big)^{\frac{1}{q_0}}e^{C_2u^2(\varphi_{n\Delta_n}(b_0))^2\Delta_nY_{t_k}^{b_0}} \\
		&\qquad\times\big(1+\sqrt{Y_{t_k}^{b_0}}+\sqrt{\Delta_n}\vert u\vert\varphi_{n\Delta_n}(b_0) Y_{t_k}^{b_0}\big)d\ell\\
		&\leq C\vert u\vert(\varphi_{n\Delta_n}(b_0))^2\Delta_n^{\frac{1}{2}+\frac{1}{q_0}}\sum_{k=0}^{n-1}e^{C_2u^2(\varphi_{n\Delta_n}(b_0))^2\Delta_nY_{t_k}^{b_0}} \big(1+\sqrt{Y_{t_k}^{b_0}}+\sqrt{\Delta_n}\vert u\vert\varphi_{n\Delta_n}(b_0) Y_{t_k}^{b_0}\big).
	\end{align*}
Then, using Young's inequality for products with $\frac{1}{\alpha}+\frac{1}{\beta}=1$ and $\beta$ close to $1$, 
\begin{align*} 
	&\Big\vert\sum_{k=0}^{n-1}\widehat{\E}^{b_0}[\zeta_{k,n}\vert \widehat{\mathcal{F}}_{t_k}]\Big\vert\leq C\vert u\vert(\varphi_{n\Delta_n}(b_0))^2\Delta_n^{\frac{1}{2}+\frac{1}{q_0}}\bigg(\frac{1}{\alpha}\sum_{k=0}^{n-1}e^{C_2\alpha u^2(\varphi_{n\Delta_n}(b_0))^2\Delta_nY_{t_k}^{b_0}} \\
	&\qquad+\frac{1}{\beta}\sum_{k=0}^{n-1} \big(1+(Y_{t_k}^{b_0})^{\frac{\beta}{2}}+(\sqrt{\Delta_n}\vert u\vert\varphi_{n\Delta_n}(b_0))^{\beta} (Y_{t_k}^{b_0})^{\beta}\big)\bigg).
\end{align*}
Thus, for subcritical case $b_0>0$ (respectively critical case $b_0=0$, respectively supercritical case $b_0<0$), using \eqref{boundedexp}, Lemma \ref{moment} \textnormal{(i)} (respectively \textcolor{black}{Lemma \ref{auxi} \textnormal{(i)}}, respectively \textcolor{black}{Lemma \ref{auxi} \textnormal{(ii)}}), we get $\widehat{\E}^{b_0}[\vert\sum_{k=0}^{n-1}\widehat{\E}^{b_0}[\zeta_{k,n}\vert \widehat{\mathcal{F}}_{t_k}]\vert]\leq C\Delta_n^{\frac{1}{q_0}-\frac{1}{2}}$ for $b_0>0$, $b_0=0$ and $b_0<0$, which tends to zero since $q_0\in (1,2)$. 

Next, we have
\begin{align}
	\sum_{k=0}^{n-1}\widehat{\E}^{b_0}[\zeta_{k,n}^2\vert \widehat{\mathcal{F}}_{t_k}]&\leq \frac{(\varphi_{n\Delta_n}(b_0))^2}{\sigma^4}\sum_{k=0}^{n-1}\int_0^1\widehat{\E}_{t_k,Y_{t_k}^{b_0}}^{b_0}\Big[\Big(\int_{t_k}^{t_{k+1}}\int_{0}^{\infty}zN(ds,dz) \notag\\
	&\qquad-\widetilde{\E}_{t_k,Y_{t_{k}}^{b_0}}^{b(\ell)}\Big[\int_{t_k}^{t_{k+1}}\int_{0}^{\infty}zM(ds,dz)\big\vert X_{t_{k+1}}^{b(\ell)}=Y_{t_{k+1}}^{b_0}\Big]\Big)^2\Big]d\ell. \label{estime}
\end{align}
For critical and supercritical cases, using \eqref{estime}, \eqref{for1} of Lemma \ref{change}, \eqref{for3} of Lemma \ref{deviation1} with $q=q_0\in (1,2)$, BDG's inequality, {\bf(A1)}, {\bf(A2)}, Young's inequality for products with $\frac{1}{\alpha}+\frac{1}{\beta}=1$ and $\beta$ close to $1$, we get for $n$ large enough,
\begin{align*}
	&\sum_{k=0}^{n-1}\widehat{\E}^{b_0}[\zeta_{k,n}^2\vert \widehat{\mathcal{F}}_{t_k}]\\
	&\leq  2\frac{(\varphi_{n\Delta_n}(b_0))^2}{\sigma^4}\sum_{k=0}^{n-1}\int_0^1\bigg\{\widehat{\E}_{t_k,Y_{t_k}^{b_0}}^{b_0}\Big[\Big(\int_{t_k}^{t_{k+1}}\int_{0}^{\infty}zN(ds,dz)\Big)^2\Big]\\
	&\qquad+\widetilde{\E}_{t_k,Y_{t_k}^{b_0}}^{b(\ell)}\Big[\Big(\int_{t_k}^{t_{k+1}}\int_{0}^{\infty}zM(ds,dz)\Big)^2\Big]+C_1\sqrt{\Delta_n}\Big(\widetilde{\E}_{t_k,Y_{t_k}^{b_0}}^{b(\ell)}\Big[\Big\vert \int_{t_k}^{t_{k+1}}\int_{0}^{\infty}zM(ds,dz)\Big\vert^{2q_0}\Big]\Big)^{\frac{1}{q_0}}\\
	&\qquad\times e^{C_2u^2(\varphi_{n\Delta_n}(b_0))^2\Delta_nY_{t_k}^{b_0}}\vert u\vert\varphi_{n\Delta_n}(b_0)\big(1+\sqrt{Y_{t_k}^{b_0}}+\sqrt{\Delta_n}\vert u\vert\varphi_{n\Delta_n}(b_0) Y_{t_k}^{b_0}\big) \bigg\}d\ell\\
	&\leq C(\varphi_{n\Delta_n}(b_0))^2\sum_{k=0}^{n-1}\bigg\{\Delta_n+\Delta_n^{\frac{1}{2}+\frac{1}{q_0}} e^{C_2u^2(\varphi_{n\Delta_n}(b_0))^2\Delta_nY_{t_k}^{b_0}}\vert u\vert\varphi_{n\Delta_n}(b_0)\\
	&\qquad\times\big(1+\sqrt{Y_{t_k}^{b_0}}+\sqrt{\Delta_n}\vert u\vert\varphi_{n\Delta_n}(b_0) Y_{t_k}^{b_0}\big) \bigg\}\\
	&\leq C(\varphi_{n\Delta_n}(b_0))^2n\Delta_n+ C\vert u\vert(\varphi_{n\Delta_n}(b_0))^3\Delta_n^{\frac{1}{2}+\frac{1}{q_0}}\bigg(\frac{1}{\alpha}\sum_{k=0}^{n-1} e^{C_2\alpha u^2(\varphi_{n\Delta_n}(b_0))^2\Delta_nY_{t_k}^{b_0}}\\
	&\qquad +  \frac{1}{\beta}\sum_{k=0}^{n-1}\big(1+(Y_{t_k}^{b_0})^{\frac{\beta}{2}}+(\sqrt{\Delta_n}\vert u\vert\varphi_{n\Delta_n}(b_0))^{\beta} (Y_{t_k}^{b_0})^{\beta}\big)\bigg). 
\end{align*}
Thus, for critical case $b_0=0$ (respectively supercritical case $b_0<0$), using \eqref{boundedexp}, \textcolor{black}{Lemma \ref{auxi} \textnormal{(i)}} (respectively \textcolor{black}{Lemma \ref{auxi} \textnormal{(ii)}}), we get  $\widehat{\E}^{0}[\vert\sum_{k=0}^{n-1}\widehat{\E}^{0}[\zeta_{k,n}^2\vert \widehat{\mathcal{F}}_{t_k}]\vert]\leq \frac{C}{n\Delta_n}$ (respectively $\widehat{\E}^{b_0}[\vert\sum_{k=0}^{n-1}\widehat{\E}^{b_0}[\zeta_{k,n}^2\vert \widehat{\mathcal{F}}_{t_k}]\vert]\leq C\Delta_n^{\frac{1}{q_0}-\frac{1}{2}}e^{\frac{1}{2}(3-\beta)b_0n\Delta_n}$), which tends to zero since $q_0\in (1,2)$ and $\beta$ is close to $1$. 

For subcritical case with $\varphi_{n\Delta_n}(b_0)=\frac{1}{\sqrt{n\Delta_n}}$, using \eqref{estime} and Lemma \ref{jumpestimate2}, we get for $q>1$, $q_0>1$, $q_1\in (1,2]$, $p_2, q_2>1$ satisfying $\frac{1}{p_2}+\frac{1}{q_2}=1$, and for $n$ large enough,
\begin{align*}
&\sum_{k=0}^{n-1}\widehat{\E}^{b_0}[\zeta_{k,n}^2\vert \widehat{\mathcal{F}}_{t_k}]\leq \dfrac{C}{n}\sum_{k=0}^{n-1}(1+(Y_{t_k}^{b_0})^2)\bigg(\left(\lambda_{\upsilon_n}\Delta_n\right)^{\frac{1}{q}}+\left(\lambda_{\upsilon_n}\Delta_n\right)^{\frac{1}{q_2}}+\Big(\frac{1}{n^{\frac{1}{2q}}}\left(\lambda_{\upsilon_n}\Delta_n\right)^{\frac{1}{q_0q}}\\
&\qquad+\left(\lambda_{\upsilon_n}\Delta_n\right)^{\frac{1}{q_2}} \frac{1}{n^{\frac{1}{2p_2}}}\Big) e^{C_1u^2(\varphi_{n\Delta_n}(b_0))^2\Delta_nY_{t_k}^{b_0}}\big(1+\sqrt{Y_{t_k}^{b_0}}+\frac{1}{\sqrt{n}} Y_{t_k}^{b_0}\big)\bigg)\\
&\qquad  +C\bigg(\int_{z\leq\upsilon_n}z^2m(dz)+\Delta_n\Big(\int_{z\leq\upsilon_n}zm(dz)\Big)^2\bigg)+\frac{C}{\sqrt{n\Delta_n}}\Delta_n^{\frac{1}{q_1}-\frac{1}{2}}\Big(\int_{z\leq\upsilon_n}z^{2q_1}m(dz)\Big)^{\frac{1}{q_1}}\\
&\qquad\times \frac{1}{n}\sum_{k=0}^{n-1} e^{C_2u^2(\varphi_{n\Delta_n}(b_0))^2\Delta_nY_{t_k}^{b_0}} \big(1+\sqrt{Y_{t_k}^{b_0}}+\frac{1}{\sqrt{n}} Y_{t_k}^{b_0}\big).
\end{align*}
Then, using \eqref{boundedexp} and Lemma \ref{moment} \textnormal{(i)}, we obtain that
\begin{align*}
&\widehat{\E}^{b_0}\Big[\Big\vert\sum_{k=0}^{n-1}\widehat{\E}^{b_0}\big[\zeta_{k,n}^2\vert \widehat{\mathcal{F}}_{t_k}\big]\Big\vert\Big]\leq C\bigg(\left(\lambda_{\upsilon_n}\Delta_n\right)^{\frac{1}{q}}+\left(\lambda_{\upsilon_n}\Delta_n\right)^{\frac{1}{q_2}}+\frac{1}{n^{\frac{1}{2q}}}\left(\lambda_{\upsilon_n}\Delta_n\right)^{\frac{1}{q_0q}}+\left(\lambda_{\upsilon_n}\Delta_n\right)^{\frac{1}{q_2}} \frac{1}{n^{\frac{1}{2p_2}}}\\
&\qquad+\int_{z\leq\upsilon_n}z^2m(dz)+\Delta_n\Big(\int_{z\leq\upsilon_n}zm(dz)\Big)^2+\frac{1}{\sqrt{n\Delta_n}}\Delta_n^{\frac{1}{q_1}-\frac{1}{2}}\Big(\int_{z\leq\upsilon_n}z^{2q_1}m(dz)\Big)^{\frac{1}{q_1}}\bigg),
\end{align*}
which tends to zero. Here, we have used the fact that
\begin{enumerate}
	\item $	\lambda_{\upsilon_n}\Delta_n=\int_{z>\upsilon_n}m(dz)\Delta_n\leq \dfrac{\Delta_n}{\upsilon_n}\int_{z>\upsilon_n}z m(dz)\to 0$ since $\frac{\Delta_n}{\upsilon_n}\to 0$ and $\int_{z>\upsilon_n}z m(dz)\leq \int_{0}^{\infty}zm(dz)<\infty$ by {\bf(A1)}. Therefore, $\lambda_{\upsilon_n}\Delta_n\to 0$ as $n\to\infty$.
	\item $\int_{z\leq\upsilon_n}z^{p}m(dz)\to 0$ as $n\to\infty$ for all $p\geq 1$ by using Lebesgue's dominated convergence theorem and the fact that $\upsilon_n\to 0$ and $\int_{0}^{\infty}z^pm(dz)<\infty$.
\end{enumerate}
This completes the proof.
\end{proof}

\section{Appendix A: Proof of technical results}

\subsection{Proof of Lemma \ref{moment}}
\begin{proof} The proof of (i), (ii) and (iii) follows from \cite[Lemma 1]{BKT17}, the Comparision Theorem (see Lemma \ref{comparisontheorem}) and \cite[Theorem 1.1]{JKR17}. \begin{color}{black} Now, it remains to prove (iv). Indeed, it suffices to show that for any $p>0$, there exists a positive constant $C_p$ such that for any $t\geq 0$,
\begin{align}\label{momentpY}
\widehat{\E}\big[(Y_t^b)^p\big]\leq\begin{cases}
 C_p (1+t)^{p} & \text{ if \;\;} b  = 0,\\
C_p e^{-pbt} & \text{ if \;\;} b < 0. 
\end{cases}
\end{align}	
We first show \eqref{momentpY} for any natural number $p\geq 1$. Observe that from Lemma \ref{Pro_moments}, we have
\begin{align}
\begin{cases}\label{firstmoment}
\widehat{\E}[Y_t^b]
=y_0 + (a + \int_0^\infty z m(dz)) t & \text{ if \;\;} b  = 0,\\
\widehat{\E}[Y_t^b]\leq\left(y_0-\frac{1}{b}\big(a + \int_0^\infty z m(dz)\big)\right)e^{-bt}  & \text{ if \;\;} b < 0. 
\end{cases}
\end{align}
Thus, \eqref{momentpY} holds for $p=1$. 

Now,  we define $\tau_N:=\inf\{t\geq0: Y_t^b\geq N\}$ for each $N>0$. Then, using \eqref{eqintegral}, we write 
\begin{equation*}
\begin{split}
Y_{t\wedge \tau_N}^b&=y_0+\int_0^{t\wedge \tau_N}(a-bY_s^b)ds +\sigma\int_0^{t\wedge \tau_N}\sqrt{Y_s^b}dW_s+\int_0^{t\wedge \tau_N}\int_0^{\infty}z(N(ds,dz)-m(dz)ds)\\
&\qquad+\int_0^{t\wedge \tau_N}\int_0^{\infty}zm(dz)ds.
\end{split}
\end{equation*}
Therefore, taking the expectation on both sides and using \eqref{firstmoment}, we obtain for $b\leq 0$,
\begin{align*}
\widehat{\E}^{b}\left[Y_{t\wedge \tau_N}^b\right]\leq y_0+\int_0^{t}(a-b\widehat{\E}^{b}[Y_s^b])ds+t\int_0^{\infty}zm(dz)\leq c(t),
\end{align*}
where
\begin{align*}
c(t)=\begin{cases}
y_0 + (a + \int_0^\infty z m(dz)) t  & \text{ if \;\;} b  = 0,\\
y_0 + (a + \int_0^\infty z m(dz)) t+\left(y_0-\frac{1}{b}\big(a + \int_0^\infty z m(dz)\big)\right)(e^{-bt}-1)& \text{ if \;\;} b < 0,
\end{cases}
\end{align*}
which does not depend on $N$. 

Next, observe that
\begin{align*}
\{\tau_N<t\}=\{\tau_N<t,Y_{t\wedge \tau_N}^b\geq N\} \subset \{Y_{t\wedge \tau_N}^b\geq N\},
\end{align*}
which, together with Markov's inequality, implies that
\begin{align*}
\widehat{\P}^{b}(\tau_N<t)\leq \widehat{\P}^{b}(Y_{t\wedge \tau_N}^b\geq N)&\leq \dfrac{\widehat{\E}^{b}\left[Y_{t\wedge \tau_N}^b\right]}{N}\leq \dfrac{c(t)}{N}.
\end{align*}
This fact, together with the monotonicity of $\tau_N$ with respect to $N$,  implies that $\tau_N\uparrow \infty$ a.s. as $N\uparrow \infty$. 

Now, we suppose that \eqref{momentpY} is valid for all natural number $q\in \{1,\ldots,p-1\}$. That is, 
\begin{align}\label{induc}
\widehat{\E}\big[(Y_t^b)^q\big]\leq\begin{cases}
C_q (1+t)^{q} & \text{ if \;\;} b  = 0,\\
C_q e^{-qbt} & \text{ if \;\;} b < 0. 
\end{cases}
\end{align}	
We shall show that \eqref{momentpY} holds for $p$. For this, using It\^o's formula, we have for all $p\geq 1$,
\begin{align*}
e^{pbt}(Y_t^b)^p&=y_0^p+\left(pa+\frac{1}{2}p(p-1)\sigma^2\right)\int_0^te^{pbs}(Y_s^b)^{p-1}ds +p\sigma\int_0^t e^{pbs}(Y_s^b)^{p-1}\sqrt{Y_s^b}dW_s \\
&\qquad+\int_0^t\int_0^{\infty}e^{pbs}\left((Y_{s-}^b+z)^p-(Y_{s-}^b)^p\right)N(ds,dz). 
\end{align*}
Thus, replacing $t$ by $t\wedge \tau_N$, we have
\begin{align*}
&e^{pb(t\wedge \tau_N)}(Y_{t\wedge \tau_N}^b)^p=y_0^p+\left(pa+\frac{1}{2}p(p-1)\sigma^2\right)\int_0^{t\wedge \tau_N}e^{pbs}(Y_s^b)^{p-1}ds\\
&\quad +p\sigma\int_0^{t\wedge \tau_N} e^{pbs}(Y_s^b)^{p-1}\sqrt{Y_s^b}dW_s + \int_0^{t\wedge \tau_N}\int_0^{\infty}e^{pbs}\left((Y_{s}^b+z)^p-(Y_{s}^b)^p\right)m(dz)ds\\
&\quad+\int_0^{t\wedge \tau_N}\int_0^{\infty}e^{pbs}\left((Y_{s-}^b+z)^p-(Y_{s-}^b)^p\right)(N(ds,dz)-m(dz)ds). 
\end{align*}
Therefore, taking expectation on both sides, and using the binomial theorem, we obtain 
\begin{align*}
\widehat{\E}^{b}\left[e^{pb(t\wedge \tau_N)}(Y_{t\wedge \tau_N}^b)^p\right]&\leq y_0^p+\left(pa+\frac{1}{2}p(p-1)\sigma^2\right)\int_0^te^{pbs}\widehat{\E}^{b}[(Y_s^b)^{p-1}]ds\\
&\qquad+\int_0^te^{pbs}\sum_{i=0}^{p-1}{p \choose i}\int_0^{\infty}z^{p-i}m(dz) \widehat{\E}^{b}[(Y_s^b)^{i}]ds.
\end{align*}	
Then, using the fact that $\tau_N\uparrow \infty$ a.s. as $N\uparrow  \infty$, letting $N\uparrow  \infty$ on the left hand side and using the dominated convergence theorem, we get
\begin{align*}
\widehat{\E}^{b}\left[e^{pbt}(Y_{t}^b)^p\right]&\leq y_0^p+\left(pa+\frac{1}{2}p(p-1)\sigma^2\right)\int_0^te^{pbs}\widehat{\E}^{b}[(Y_s^b)^{p-1}]ds\\
&\qquad+\int_0^te^{pbs}\sum_{i=0}^{p-1}{p \choose i}\int_0^{\infty}z^{p-i}m(dz) \widehat{\E}^{b}[(Y_s^b)^{i}]ds.
\end{align*}	
Therefore, when $b=0$, using {\bf(A1)}, {\bf(A2)} and the inductive assumption \eqref{induc}, we get
\begin{align*}
\widehat{\E}^{0}\left[(Y_t^0)^p\right]&\leq y_0^p+C_p \sum_{i=0}^{p-1}\int_0^t\widehat{\E}^{0}\left[(Y_s^0)^i\right] ds\\
&\leq y_0^p+C_p \sum_{i=0}^{p-1}\int_0^t(1+s)^i ds\\
&\leq C_p(1+t)^p,
\end{align*}
for some positive constant $C_p$ whose values may change from one line to the next.
	
When $b<0$, using again {\bf(A1)}, {\bf(A2)} and the inductive assumption \eqref{induc}, we get
\begin{align*}
\widehat{\E}^{b}\left[e^{pbt}(Y_t^b)^p\right]&\leq y_0^p+C_p\sum_{i=0}^{p-1}\int_0^te^{pbs} \widehat{\E}^{b}[(Y_s^b)^{i}]ds\\
&\leq y_0^p+C_p\sum_{i=0}^{p-1}\int_0^te^{pbs} e^{-ibs}ds\\
&= y_0^p+C_p \sum_{i=0}^{p-1}\dfrac{1}{(p-i)b}\left(e^{(p-i)bt}-1\right),
\end{align*}	
which implies that
\begin{align*}
\widehat{\E}^{b}\left[(Y_t^b)^p\right]\leq C_pe^{-pbt}.
\end{align*}	
Therefore, \eqref{momentpY} is valid for $p$. By the induction principle, \eqref{momentpY} holds for any natural number $p\geq 1$. Finally, using the monotonicity of the $L^p(\widehat{\P})$-norms as a function of $p$, we get the result for any $p>0$.
\end{color}	
\end{proof}	

\begin{color}{black} 
\subsection{Proof of Lemma \ref{auxi}}
	\begin{proof}
		The first inequality is straightforward for both cases $b_0=0$ and $b_0<0$. Now, for the second inequality, we first write 
		\begin{align}
		\sum_{k=0}^{n-1}\int_{t_k}^{t_{k+1}}\sup_{u\in[t_k,s]}\widehat{\E}^{b_0}\big[\frac{1}{(Y_u^{b_0})^{p}}\big]ds&=\sum_{k=0}^{[\frac{1}{\Delta_n}]}\int_{t_k}^{t_{k+1}}\sup_{u\in[t_k,s]}\widehat{\E}^{b_0}\big[\frac{1}{(Y_u^{b_0})^{p}}\big]ds \notag\\
		&\qquad+\sum_{k=[\frac{1}{\Delta_n}]+1}^{n-1}\int_{t_k}^{t_{k+1}}\sup_{u\in[t_k,s]}\widehat{\E}^{b_0}\big[\frac{1}{(Y_u^{b_0})^{p}}\big]ds.\label{decomsum}
		\end{align}
		\noindent\textnormal{(i)}  For $b_0=0$, we use \eqref{decomsum}, combined with \textnormal{(ii)} of Lemma \ref{moment} for the case $0<p<\frac{2a}{\sigma^2}$. For the case $p<0$, we also use \eqref{decomsum}, combined with the first assertion of \textnormal{(iv)} in Lemma \ref{moment}.
		
		\noindent\textnormal{(ii)} For $b_0<0$, we use \eqref{decomsum}, combined with \textnormal{(iii)} of Lemma \ref{moment} for the case $0<p<\frac{2a}{\sigma^2}$. For the case $p<0$, we use in the same way \eqref{decomsum}, combined with the second assertion of \textnormal{(iv)} in Lemma \ref{moment}. The result follows by standard calculations with distinguishing the cases $p\neq 1$ and $p=1$.
	\end{proof}	
\end{color}	

\subsection{Proof of Proposition \ref{smoothdensity}}
\begin{proof}		
	First we recall that since the jump-type CIR process  $Y^b=(Y_t^b)_{t\in\R_+}$ is an affine process, the corresponding characteristic function of $Y_t^b$ is of exponential-affine form (see page 287 and 288 of \cite{JRT16}, Section 3 of \cite{JKR17}, Section 4.1 of \cite{FMS13}, \eqref{Laplace2}). That is, for all $(t,u)\in \R_+\times \mathcal{U}$ with $\mathcal{U}:=\{u\in \mathbb{C}: \textup{Re}\; u\leq 0\}$,
	\begin{align*}
		\widehat{\E}\left[e^{uY_t^{b}}\right]=e^{\phi_b(t,u)+y_0\psi_b(t,u)},
	\end{align*}
	where $\textup{Re}\; u$ denotes the real part of $u$ and the functions $\phi_b(t,u)$ and $\psi_b(t,u)$ are solutions to the generalized Riccati equations
	\begin{align*}
		\begin{cases}
			\partial_t\phi_b(t,u)=F(\psi_b(t,u)),\; \phi_b(0,u)=0,\\
			\partial_t\psi_b(t,u)=R(\psi_b(t,u)),\; \psi_b(0,u)=u\in \mathcal{U},
		\end{cases}
	\end{align*}
	with the functions $F$ and $R$ given by
	\begin{align*}
		F(u)=au+\int_0^{\infty}(e^{uz}-1)m(dz),\qquad R(u)=\dfrac{\sigma^2u^2}{2}-bu.
	\end{align*}
	Solving the system above, we get the following explicit form 	
\begin{align*}
	\psi_b(t,u)\equiv\psi_{u,0}(t)=\begin{cases}
		\frac{u}{1-\frac{\sigma^2u}{2}t} &\textnormal{if}\quad b=0,\\
		\frac{ue^{-bt}}{1-\frac{\sigma^2u}{2b}(1-e^{-bt})} &\textnormal{if}\quad b\neq 0.
	\end{cases}
\end{align*}
	In what follows, the notation constant $C$  will designate a generic constant which can change values from one bound to another.
	Now, using Lemma C.6 of \cite{FMS13}, there exist constants $C>0$ and $\overline{R}>0$ such that
	\begin{align}\label{ineq}
		\left\vert e^{\phi_b(t,iu)+y_0\psi_b(t,iu)}\right\vert\leq \dfrac{C}{(1+\vert u\vert)^{\frac{2a}{\sigma^2}}},
	\end{align}
	for all $u\in\R$ with $\vert u\vert\geq \overline{R}$. Hence,  as $\frac{2a}{\sigma^2}>1$ we get
	\begin{align*}
		\int_{\R}\left\vert \widehat{\E}\left[e^{iuY_t^{b}}\right]\right\vert du&=\int_{\vert u\vert\leq \overline{R}}\left\vert \widehat{\E}\left[e^{iuY_t^{b}}\right]\right\vert du+\int_{\vert u\vert\geq \overline{R}}\left\vert e^{\phi_b(t,iu)+y_0\psi_b(t,iu)}\right\vert du<+\infty.
	\end{align*}	
	Therefore, by the inversion Fourier theorem, we obtain the existence of the density 
	\begin{align*}
		p^{b}(t,y_0,y)=\dfrac{1}{2\pi}\int_{\R}e^{-iy u}\widehat{\E}\left[e^{iuY_t^{b}}\right]du=\dfrac{1}{2\pi}\int_{\R}e^{-iy u}e^{\phi_b(t,iu)+y_0\psi_b(t,iu)}du.
	\end{align*}	
	Next, in order to prove the smoothness of the density w.r.t. $b$, we are going to show that 
	\begin{align}\label{integrable}
		\int_{\R}\left\vert e^{-iy u}\dfrac{\partial}{\partial b}\widehat{\E}\left[e^{iuY_t^{b}}\right]\right\vert du=\int_{\R}\left\vert e^{-iy u}\dfrac{\partial}{\partial b}e^{\phi_b(t,iu)+y_0\psi_b(t,iu)}\right\vert du<+\infty, \;\forall b\in \mathbb R.
	\end{align}
	First, observe that	
	\begin{align} \label{deri}
		\left\vert e^{-iy u}\dfrac{\partial}{\partial b}e^{\phi_b(t,iu)+y_0\psi_b(t,iu)}\right\vert=\left\vert\dfrac{\partial}{\partial b}\phi_b(t,iu)+y_0\dfrac{\partial}{\partial b}\psi_b(t,iu)\right\vert\left\vert e^{\phi_b(t,iu)+y_0\psi_b(t,iu)}\right\vert,
	\end{align}	
	and
	\begin{align*}
		\dfrac{\partial}{\partial b}\psi_b(t,iu)=\dfrac{-iute^{-bt}\left(1-\frac{\sigma^2iu}{2b}(1-e^{-bt})\right)-iue^{-bt}\left(\frac{\sigma^2iu}{2b^2}(1-e^{-bt})-\frac{\sigma^2iu}{2b}te^{-bt}\right)}{\left(1-\frac{\sigma^2iu}{2b}(1-e^{-bt})\right)^2},
	\end{align*}
	which is continuous w.r.t. $b$ for all $b\neq0$ and  we can easily see that it is continuous for all $b\in\mathbb R$.	
	Therefore, by standard calculations there exists $R'>0$ such that for all $\vert u\vert\geq R'$,
	\begin{equation}\begin{split}\label{dpsi}
			\left\vert \dfrac{\partial}{\partial b}\psi_b(t,iu)\right\vert&\leq \dfrac{C\vert u\vert^2}{1+\frac{\sigma^4u^2}{4b^2}(1-e^{-bt})^2}\leq C.
		\end{split}
	\end{equation}
	Furthermore,
	\begin{align*}
		\phi_b(t,iu)=a\int_0^t\psi_b(s,iu)ds+\int_0^t\int_0^{\infty}(e^{\psi_b(s,iu)z}-1)m(dz)ds.
	\end{align*} 
	Now, for all $\vert u\vert\geq R'$, we have $\vert \frac{\partial}{\partial b}\psi_b(t,iu)\vert\leq C$ and
	\begin{align*}
		\left\vert z\frac{\partial}{\partial b}\psi_b(s,iu)e^{\psi_b(s,iu)z}\right\vert \leq C ze^{z\textup{Re}\; \psi_b(s,iu)}\leq Cz,
	\end{align*}
	since 
	$$
	\textup{Re}\; \psi_b(s,iu)=-\dfrac{\frac{\sigma^2u^2}{2b}e^{-bt}(1-e^{-bt})}{1+\frac{\sigma^4u^2}{4b^2}(1-e^{-bt})^2}\leq 0.
	$$
	Then, for all $\vert u\vert\geq R'$, we have $\int_0^t\vert \frac{\partial}{\partial b}\psi_b(s,iu)\vert ds\leq Ct$ and using condition {\bf(A1)},
	\begin{align*}
		\int_0^t\int_0^{\infty}\left\vert z\frac{\partial}{\partial b}\psi_b(s,iu)e^{\psi_b(s,iu)z}\right\vert m(dz)ds\leq C\int_0^t\int_0^{\infty} zm(dz)ds=Ct\int_0^{\infty} zm(dz)<+\infty.
	\end{align*}
	Thus, we have shown that for all $\vert u\vert\geq R'$,
	\begin{align*}
		\frac{\partial}{\partial b}\phi_b(t,iu)=a\int_0^t\frac{\partial}{\partial b}\psi_b(s,iu)ds+\int_0^t\int_0^{\infty}z\frac{\partial}{\partial b}\psi_b(s,iu)e^{\psi_b(s,iu)z}m(dz)ds,
	\end{align*}
	which is also continuous w.r.t. $b$ for all $b\in\mathbb R$ and then
	\begin{align}\label{dphi}
		\left\vert \frac{\partial}{\partial b}\phi_b(t,iu)\right\vert\leq Ct.
	\end{align}
	Hence, from \eqref{deri}, \eqref{ineq}, \eqref{dpsi} and \eqref{dphi}, for all $\vert u\vert\geq \overline{R}\vee R'$, 
	\begin{align*}
		\left\vert e^{-iy u}\dfrac{\partial}{\partial b}e^{\phi_b(t,iu)+y_0\psi_b(t,iu)}\right\vert\leq \dfrac{C}{(1+\vert u\vert)^{\frac{2a}{\sigma^2}}},
	\end{align*}
	which implies that
	\begin{align}\label{big}
		\int_{\vert u\vert\geq \overline{R}\vee R'}\left\vert e^{-iy u}\dfrac{\partial}{\partial b}e^{\phi_b(t,iu)+y_0\psi_b(t,iu)}\right\vert du\leq \int_{\vert u\vert\geq \overline{R}\vee R'}\dfrac{C}{(1+\vert u\vert)^{\frac{2a}{\sigma^2}}}du <+\infty.
	\end{align}
	On the other hand, for all $\vert u\vert\leq \overline{R}\vee R'$, using \eqref{dxb2} for $Y^b$, H\"older's inequality with $\frac{1}{p}+\frac{1}{q}+\frac{1}{r}=1$ where $r$ is close to $1$, \eqref{e1} and \eqref{e4} for $Y^b$, we get
	\begin{align*}
		&\left\vert e^{-iy u}\dfrac{\partial}{\partial b}e^{\phi_b(t,iu)+y_0\psi_b(t,iu)}\right\vert=\left\vert \dfrac{\partial}{\partial b}\widehat{\E}\left[e^{iuY_t^{b}}\right]\right\vert=\left\vert \widehat{\E}\left[iu\dfrac{\partial}{\partial b}Y_t^{b}e^{iuY_t^{b}}\right]\right\vert\\
		&\leq  \vert u\vert \widehat{\E}\left[\left\vert\partial_bY_t^{b}\right\vert\right]\leq (\overline{R}\vee R') \int_0^t\widehat{\E}\left[\left\vert Y_s^{b}\partial_xY_t^{b}(\partial_xY_s^{b})^{-1}\right\vert\right]ds\\
		&\leq (\overline{R}\vee R') \int_0^t\left(\widehat{\E}\big[( Y_s^{b})^p\big]\right)^{\frac{1}{p}}\left(\widehat{\E}\big[(\partial_xY_t^{b})^q\big]\right)^{\frac{1}{q}}\left(\widehat{\E}\big[ (\partial_xY_s^{b})^{-r}\big]\right)^{\frac{1}{r}}ds\\
		&\leq Ct(\overline{R}\vee R')(1+y_0)\left(1+\dfrac{1}{y_0^{\frac{a}{\sigma^2}}}\right)\\
		&\leq Ct.
	\end{align*}
Here, $-r\geq -\frac{(\frac{2a}{\sigma^2}-1)^2}{2(\frac{2a}{\sigma^2}-\frac{1}{2})}$ and $r$ is close to $1$, which implies that $\frac{a}{\sigma^2}>1+\frac{\sqrt{2}}{2}$.
	Consequently, under $\frac{a}{\sigma^2}>1+\frac{\sqrt{2}}{2}$, we get
	\begin{align}\label{small}
		\int_{\vert u\vert\leq \overline{R}\vee R'}\left\vert e^{-iy u}\dfrac{\partial}{\partial b}e^{\phi_b(t,iu)+y_0\psi_b(t,iu)}\right\vert du\leq Ct\int_{\vert u\vert\leq \overline{R}\vee R'}du <+\infty,\; \forall b\in \mathbb R.
	\end{align}
	From \eqref{big} and \eqref{small}, we conclude \eqref{integrable}, which implies that the $p^{b}(t,y_0,y)$ is of class $C^1$ w.r.t. $b$ for all $b\in\mathbb R$ under $\frac{a}{\sigma^2}>1+\frac{\sqrt{2}}{2}$ and its derivative is given by
	\begin{align*}
		\partial_bp^{b}(t,y_0,y)&=\dfrac{1}{2\pi}\int_{\R}e^{-iy u}\dfrac{\partial}{\partial b}e^{\phi_b(t,iu)+y_0\psi_b(t,iu)}du\\
		&=\dfrac{1}{2\pi}\int_{\R}e^{-iy u}\left(\dfrac{\partial}{\partial b}\phi_b(t,iu)+y_0\dfrac{\partial}{\partial b}\psi_b(t,iu)\right) e^{\phi_b(t,iu)+y_0\psi_b(t,iu)}du.
	\end{align*}
	Thus, the result follows.
\end{proof}	

\subsection{Proof of Lemma \ref{estimates}}
\begin{proof}
	{\it Proof of \eqref{e1}.} This estimate can be obtained from equation \eqref{flowk}, inequality $\vert a+b+c\vert^p\leq 3^{p-1}(\vert a\vert^p+\vert b\vert^p+\vert c\vert^p)$, BDG's inequality and Gronwall's inequality together with conditions {\bf(A1)}-{\bf(A2)}.	
	
	{\it Proof of \eqref{e4}.} We consider on the probability space $(\widetilde{\Omega}, \widetilde{\mathcal{F}}, \widetilde{\P})$ the diffusion-type CIR process $\overline{X}^{a,b}=(\overline{X}_t^{a,b})_{t\in\R_+}$ defined by
	\begin{equation*}
	\overline{X}_t^{a,b}=y_0+\int_0^t (a-b\overline{X}_s^{a,b})ds +\sigma\int_0^t\sqrt{\overline{X}_s^{a,b}}dB_s.
	\end{equation*}
	The flow process $\overline{X}^{a,b}(s,x)=(\overline{X}_t^{a,b}(s,x), t\geq s)$ on the time interval $[s,\infty)$ and with initial condition $\overline{X}_{s}^{a,b}(s,x)=x\in\R_{++}$ is defined by
	\begin{equation*}
	\overline{X}_t^{a,b}(s,x)=x+\int_s^t \big(a-b\overline{X}_u^{a,b}(s,x)\big)du +\sigma \int_s^t\sqrt{\overline{X}_u^{a,b}(s,x)}dB_u,
	\end{equation*}	
	for any $t\geq s$. Applying \cite[Lemma 3.1]{BD07}, under $\frac{a}{\sigma^2}>1$, for any $\mu\leq(\frac{2a}{\sigma^2}-1)^2\frac{\sigma^2}{8}$ and $t\in[t_k,t_{k+1}]$, we have 
	\begin{align*}
	\widetilde{\E}_{t_k,x}^{a,b}\left[\exp\left\{\mu\int_{t_k}^t\dfrac{du}{\overline{X}_u^{a,b}(t_k,x)}\right\}\right]\leq C\left(1+\dfrac{1}{x^{\frac{1}{2}(\frac{2a}{\sigma^2}-1)}}\right).
	\end{align*}
	Indeed, the proof for $b\geq 0$ is given in \cite[Lemma A.2]{BD07}. The proof for $b<0$ is based on the comparison theorem.
	
	The comparison theorem (see \cite[Proposition A.1]{BBKP17} or Lemma \ref{comparisontheorem}) gives
	$$
	\widetilde{\P}\left(X_t^{a,b}\geq \overline{X}_t^{a,b},\ \forall t\in\R_+\right)=1,
	$$
	which implies that 
	\begin{align}\label{compar}
	\widetilde{\P}\left(X_t^{a,b}(t_k,x)\geq \overline{X}_t^{a,b}(t_k,x),\ \forall t\geq t_k\right)=1.
	\end{align}
	Consequently, under $\frac{a}{\sigma^2}>1$, for any $\mu\leq(\frac{2a}{\sigma^2}-1)^2\frac{\sigma^2}{8}$ and $t\in[t_k,t_{k+1}]$, we get
	\begin{align}
	\widetilde{\E}_{t_k,x}^{a,b}\left[\exp\left\{\mu\int_{t_k}^t\dfrac{du}{X_u^{a,b}(t_k,x)}\right\}\right]&\leq\widetilde{\E}_{t_k,x}^{a,b}\left[\exp\left\{\mu\int_{t_k}^t\dfrac{du}{\overline{X}_u^{a,b}(t_k,x)}\right\}\right]\notag\\
	&\leq C\left(1+\dfrac{1}{x^{\frac{1}{2}(\frac{2a}{\sigma^2}-1)}}\right),\label{expmoment}
	\end{align}
	for some constant $C>0$.	
	
	Now, for all $a, a_1\in\R_+$, $b, b_1\in\R$, $x\in\R_{++}$ and $k\in\{0,...,n-1\}$, the probability measures $\widetilde{\P}_{t_k,x}^{a_1,b_1}$ and $\widetilde{\P}_{t_k,x}^{a,b}$ are absolutely continuous w.r.t. each other and its Radon-Nikodym derivative is given by
	\begin{align*}
	&\dfrac{d\widetilde{\P}_{t_k,x}^{a_1,b_1}}{d\widetilde{\P}_{t_k,x}^{a,b}}\big((X_s^{a,b}(t_k,x))_{s\in I_k}\big)=\exp\Bigg\{\frac{1}{\sigma}\int_{t_k}^{t_{k+1}}\frac{a_1-a-(b_1-b)X_s^{a,b}(t_k,x)}{\sqrt{X_s^{a,b}(t_k,x)}}dB_s\\
	&\qquad-\frac{1}{2\sigma^2}\int_{t_k}^{t_{k+1}}\frac{(a_1-a-(b_1-b)X_s^{a,b}(t_k,x))^2}{X_s^{a,b}(t_k,x)}ds\Bigg\},
	\end{align*}
	where $I_k:=[t_k, t_{k+1}]$. Using \eqref{dxe} and the change of measures, we have that for any $p\in\R$,
	\begin{align}
		&\widetilde{\E}_{t_k,x}^{a,b}\big[(\partial_xX_t^{a,b}(t_k,x))^p\big]\notag\\
		&=e^{-bp(t-t_k)}\widetilde{\E}_{t_k,x}^{a,b}\bigg[\exp{\Big\{\frac{-p\sigma^2}{8}\int_{t_k}^t\frac{du}{X_u^{a,b}(t_k,x)}+\frac{p\sigma}{2}\int_{t_k}^t\dfrac{dB_u}{\sqrt{X_u^{a,b}(t_k,x)}}\Big\}}\bigg]\notag\\
		&=e^{-bp(t-t_k)}\widetilde{\E}_{t_k,x}^{a,b}\Bigg[\exp{\bigg\{\dfrac{p(p-1)\sigma^2}{8}\int_{t_k}^t\dfrac{du}{X_u^{a,b}(t_k,x)}\bigg\}}\notag	\\
		&\qquad\times\exp{\bigg\{\dfrac{p\sigma}{2}\int_{t_k}^t\dfrac{dB_u}{\sqrt{X_u^{a,b}(t_k,x)}}-\dfrac{p^2\sigma^2}{8}\int_{t_k}^t\dfrac{du}{X_u^{a,b}(t_k,x)}\bigg\}}\Bigg]\notag\\
		&=e^{-bp(t-t_k)}\widetilde{\E}_{t_k,x}^{a,b}\left[\exp{\bigg\{\dfrac{p(p-1)\sigma^2}{8}\int_{t_k}^t\dfrac{du}{X_u^{a,b}(t_k,x)}\bigg\}}\dfrac{d\widetilde{\P}_{t_k,x}^{a+p\frac{\sigma^2}{2},b}}{d\widetilde{\P}_{t_k,x}^{a,b}}\big((X_s^{a,b}(t_k,x))_{s\in [t_k,t]}\big)\right]\notag\\
		&=e^{-bp(t-t_k)}\widetilde{\E}_{t_k,x}^{a+p\frac{\sigma^2}{2},b}\left[\exp{\bigg\{\dfrac{p(p-1)\sigma^2}{8}\int_{t_k}^t\dfrac{du}{X_u^{a+p\frac{\sigma^2}{2},b}(t_k,x)}\bigg\}}\right].\label{measureq}
	\end{align}
	Then, applying the estimate \ref{expmoment} to the probability measure $\widetilde{\P}_{t_k,x}^{a+p\frac{\sigma^2}{2},b}$ with $\mu=\frac{p(p-1)\sigma^2}{8}$, we get that for any $p\geq -\frac{(\frac{2a}{\sigma^2}-1)^2}{2(\frac{2a}{\sigma^2}-\frac{1}{2})}$ and $t\in[t_k,t_{k+1}]$, 
	\begin{align*}
		\widetilde{\E}_{t_k,x}^{a+p\frac{\sigma^2}{2},b}\left[\exp{\bigg\{\dfrac{p(p-1)\sigma^2}{8}\int_{t_k}^t\dfrac{du}{X_u^{a+p\frac{\sigma^2}{2},b}(t_k,x)}\bigg\}}\right]\leq C_p\left(1+\dfrac{1}{x^{\frac{\frac{2a}{\sigma^2}-1+p}{2}}}\right).
	\end{align*}
	This, combined with \eqref{measureq}, gives the desired estimate \eqref{e4}.	
	
	{\it Proof of \eqref{e2}.} From \cite[Lemma 3.1]{BD07}, we have that for any $p\in[1,\frac{2a}{\sigma^2}-1)$, 
	\begin{align*}
	\widetilde{\E}_{t_k,x}^{b}\left[\dfrac{1}{ (\overline{X}_t^{b}(t_k,x))^p}\right]\leq\dfrac{C_p}{x^p},
	\end{align*}
	for a constant $C_p>0$. Then, using  \eqref{compar}, we get
	\begin{align*}
	\widetilde{\E}_{t_k,x}^{b}\left[\dfrac{1}{(X_t^{b}(t_k,x))^p}\right]\leq\widetilde{\E}_{t_k,x}^{b}\left[\dfrac{1}{( \overline{X}_t^{b}(t_k,x))^p}\right]\leq\dfrac{C_p}{x^p}.
	\end{align*}
	Thus, the result follows.	
\end{proof}

\subsection{Useful lemma}
\textcolor{black}{In order to prove (ii) of Lemma \ref{Malliderivable},  the following lemma will be needed.}
\begin{lemma}\label{Malliflow} Let $p\geq 2$. Assume condition $\frac{a}{\sigma^2}>\frac{1}{2}(3p+1+\sqrt{9p^2+2p})$ for \eqref{dmif1} and \eqref{dmif2}, and condition $\frac{a}{\sigma^2}>\frac{1}{2}(2p+1+\sqrt{2p(2p+1)})$ for \eqref{dmif3}. Then for any $b\in\R$, $k \in \{0,...,n-1\}$, $t_k\leq s\leq t\leq t_{k+1}$, and $x\in\R_{++}$, there exists a constant $C>0$ which does not depend on $x$ such that 
		\begin{align}
		&\widetilde{\E}_{t_k,x}^{b}\Big[\Big\vert \frac{D_s(\partial_xX_t^{b}(t_k,x))}{(\partial_xX_t^{b}(t_k,x))^2}\Big\vert^p\Big]\leq C\Big(1+\dfrac{1}{x^{\frac{\frac{2a}{\sigma^2}-1}{p_5}-\frac{p}{2}}}\Big)\Big(\dfrac{1}{x^{\frac{p}{2}}}+\dfrac{1}{x^{3p}}\Big),\label{dmif1}\\
		&\widetilde{\E}_{t_k,x}^{b}\Big[\Big\vert \dfrac{D_sX_{t}^{b}(t_k,x)}{\partial_xX_{t}^{b}(t_k,x)}\Big\vert^p\Big]\leq C(1+x^{\frac{p}{2}})\Big(1+\dfrac{1}{x^{\frac{\frac{2a}{\sigma^2}-1}{2p_7}-p}}\Big)\Big(1+\dfrac{1}{x^{\frac{\frac{2a}{\sigma^2}-1}{2q_7q_8}+\frac{p}{2}}}\Big),\label{dmif3}\\
		&\widetilde{\E}_{t_k,x}^{b}\Big[\Big\vert X_{t}^{b}(t_k,x)D_s\Big(\dfrac{1}{\partial_xX_{t}^{b}(t_k,x)}\Big)\Big\vert^p\Big]\leq \frac{C}{x^{2p}}(1+x^{3p})\Big(1+\dfrac{1}{x^{\frac{\frac{2a}{\sigma^2}-1}{p_4p_5}-p}}\Big)\Big(1+\dfrac{1}{x^{\frac{\frac{2a}{\sigma^2}-1}{2q_4q_6}+\frac{p}{2}}}\Big),\label{dmif2}
		\end{align}
		where $p_5=\frac{6p+2+2\sqrt{9p^2+2p}}{4p+1}$, $\frac{1}{p_7}+\frac{1}{q_7}=1$ with $p_7>1$ and $p_7$  close to $1$, $q_8>1$, $\frac{1}{p_4}+\frac{1}{q_4}=1$ with  $p_4>1$ and $p_4$ close to $1$, and $q_6>1$.
	\end{lemma}
\begin{proof}
	First, we treat \eqref{dmif1}. For $t_k\leq s\leq t \leq t_{k+1}$, from \eqref{dmpx} and \eqref{expression1}, under $\frac{a}{\sigma^2}>\frac{5+3\sqrt{2}}{2}$ we have	
	\begin{align}\label{deriinver}
	&\frac{-1}{(\partial_xX_t^{b}(t_k,x))^2}D_s\big(\partial_xX_t^{b}(t_k,x)\big) \notag\\
	&=\frac{-1}{\partial_xX_t^{b}(t_k,x)}\bigg(\dfrac{\sigma}{2\sqrt{X_s^{b}(t_k,x)}} +\dfrac{\sigma^2}{8}\int_{s}^t\dfrac{D_sX_u^{b}(t_k,x)}{(X_u^{b}(t_k,x))^2}du -\dfrac{\sigma}{4}\int_{s}^t\dfrac{D_sX_u^{b}(t_k,x)}{(X_u^{b}(t_k,x))^{\frac{3}{2}}}dB_u\bigg)\notag\\ &=\frac{-1}{\partial_xX_t^{b}(t_k,x)}\bigg(\dfrac{\sigma}{2\sqrt{X_s^{b}(t_k,x)}} +\dfrac{\sigma^3}{8}\int_{s}^t\dfrac{\sqrt{X_s^{b}(t_k,x)}\partial_xX_u^{b}(t_k,x)}{(X_u^{b}(t_k,x))^2\partial_xX_s^{b}(t_k,x)}du\notag\\
	&\qquad-\dfrac{\sigma^2}{4}\int_{s}^t\dfrac{\sqrt{X_s^{b}(t_k,x)}\partial_xX_u^{b}(t_k,x)}{(X_u^{b}(t_k,x))^{\frac{3}{2}}\partial_xX_s^{b}(t_k,x)}dB_u\bigg).
	\end{align} 	
	Then, using BDG's and H\"older's inequalities with $\frac{1}{p_4}+\frac{1}{q_4}=1$, $\frac{1}{p_5}+\frac{1}{p_5}+\frac{1}{q_5}=1$, $\frac{1}{p_6}+\frac{1}{q_6}=1$, \eqref{e1}, \eqref{e2} and \eqref{e4}, we get that
	\begin{align*}
	&\widetilde{\E}_{t_k,x}^{b}\Big[\Big\vert \frac{D_s(\partial_xX_t^{b}(t_k,x))}{(\partial_xX_t^{b}(t_k,x))^2}\Big\vert^{p}\Big]\leq C\widetilde{\E}_{t_k,x}^{b}\Big[\Big\vert \dfrac{1}{\partial_xX_t^{b}(t_k,x)\sqrt{X_s^{b}(t_k,x)}}\Big\vert^{p}\Big]\\
	&\qquad+C\widetilde{\E}_{t_k,x}^{b}\Big[\Big\vert \int_{s}^t\dfrac{\sqrt{X_s^{b}(t_k,x)}\partial_xX_u^{b}(t_k,x)}{\partial_xX_t^{b}(t_k,x)(X_u^{b}(t_k,x))^2\partial_xX_s^{b}(t_k,x)}du\Big\vert^{p}\Big]\\
	&\qquad+C\widetilde{\E}_{t_k,x}^{b}\Big[\Big\vert \int_{s}^t\Big(\dfrac{\sqrt{X_s^{b}(t_k,x)}\partial_xX_u^{b}(t_k,x)}{\partial_xX_t^{b}(t_k,x)(X_u^{b}(t_k,x))^{\frac{3}{2}}\partial_xX_s^{b}(t_k,x)}\Big)^2du\Big\vert^{\frac{p}{2}}\Big]\\
	&\leq C\widetilde{\E}_{t_k,x}^{b}\Big[\Big\vert \dfrac{1}{\partial_xX_t^{b}(t_k,x)\sqrt{X_s^{b}(t_k,x)}}\Big\vert^{p}\Big]\\
	&\qquad+C\Delta_n^{p-1} \int_{s}^t\widetilde{\E}_{t_k,x}^{b}\Big[\Big\vert\dfrac{\sqrt{X_s^{b}(t_k,x)}\partial_xX_u^{b}(t_k,x)}{\partial_xX_t^{b}(t_k,x)(X_u^{b}(t_k,x))^2\partial_xX_s^{b}(t_k,x)}\Big\vert^{p}\Big]du\\
	&\qquad+C\Delta_n^{\frac{p}{2}-1} \int_{s}^t\widetilde{\E}_{t_k,x}^{b}\Big[\Big\vert\dfrac{\sqrt{X_s^{b}(t_k,x)}\partial_xX_u^{b}(t_k,x)}{\partial_xX_t^{b}(t_k,x)(X_u^{b}(t_k,x))^{\frac{3}{2}}\partial_xX_s^{b}(t_k,x)}\Big\vert^{p}\Big]du\\
	&\leq C\Big(\widetilde{\E}_{t_k,x}^{b}\Big[\dfrac{1}{(\partial_xX_t^{b}(t_k,x))^{p\frac{p_5}{2}}}\Big]\Big)^{\frac{2}{p_5}}\Big(\widetilde{\E}_{t_k,x}^{b}\Big[\dfrac{1}{(X_s^{b}(t_k,x))^{p\frac{q_5}{2}}}\Big]\Big)^{\frac{1}{q_5}}\\
	&\qquad+C\Delta_n^{p-1} \int_{s}^t\Big(\widetilde{\E}_{t_k,x}^{b}\Big[\dfrac{1}{(\partial_xX_t^{b}(t_k,x)(X_u^{b}(t_k,x))^2\partial_xX_s^{b}(t_k,x))^{pp_4}}\Big]\Big)^{\frac{1}{p_4}}\\
	&\qquad\times \Big(\widetilde{\E}_{t_k,x}^{b}\Big[\Big(\sqrt{X_s^{b}(t_k,x)}\partial_xX_u^{b}(t_k,x)\Big)^{pq_4}\Big]\Big)^{\frac{1}{q_4}}du\\
	&\qquad+C\Delta_n^{\frac{p}{2}-1} \int_{s}^t\Big(\widetilde{\E}_{t_k,x}^{b}\Big[\dfrac{1}{(\partial_xX_t^{b}(t_k,x)(X_u^{b}(t_k,x))^{\frac{3}{2}}\partial_xX_s^{b}(t_k,x))^{pp_4}}\Big]\Big)^{\frac{1}{p_4}}\\
	&\qquad\times \Big(\widetilde{\E}_{t_k,x}^{b}\Big[\Big(\sqrt{X_s^{b}(t_k,x)}\partial_xX_u^{b}(t_k,x)\Big)^{pq_4}\Big]\Big)^{\frac{1}{q_4}}du\\
	&\leq \dfrac{C}{x^{\frac{p}{2}}}\Big(1+\dfrac{1}{x^{\frac{\frac{2a}{\sigma^2}-1}{p_5}-\frac{p}{2}}}\Big)+C\Delta_n^{p-1} \int_{s}^t\Big(\widetilde{\E}_{t_k,x}^{b}\Big[\dfrac{1}{(\partial_xX_t^{b}(t_k,x))^{pp_4p_5}}\Big]\Big)^{\frac{1}{p_4p_5}}\\
	&\qquad\times\Big(\widetilde{\E}_{t_k,x}^{b}\Big[\dfrac{1}{(\partial_xX_s^{b}(t_k,x))^{pp_4p_5}}\Big]\Big)^{\frac{1}{p_4p_5}}\Big(\widetilde{\E}_{t_k,x}^{b}\Big[\dfrac{1}{(X_u^{b}(t_k,x))^{2pp_4q_5}}\Big]\Big)^{\frac{1}{p_4q_5}}\\
	&\qquad\times\Big(\widetilde{\E}_{t_k,x}^{b}\Big[ (\sqrt{X_s^{b}(t_k,x)})^{pq_4p_6}\Big]\Big)^{\frac{1}{q_4p_6}}\Big(\widetilde{\E}_{t_k,x}^{b}\Big[(\partial_xX_u^{b}(t_k,x))^{pq_4q_6}\Big]\Big)^{\frac{1}{q_4q_6}}du\\
	&\qquad+C\Delta_n^{\frac{p}{2}-1} \int_{s}^t\Big(\widetilde{\E}_{t_k,x}^{b}\Big[\dfrac{1}{(\partial_xX_t^{b}(t_k,x))^{pp_4p_5}}\Big]\Big)^{\frac{1}{p_4p_5}}\\
	&\qquad\times\Big(\widetilde{\E}_{t_k,x}^{b}\Big[\dfrac{1}{(\partial_xX_s^{b}(t_k,x))^{pp_4p_5}}\Big]\Big)^{\frac{1}{p_4p_5}}\Big(\widetilde{\E}_{t_k,x}^{b}\Big[\dfrac{1}{(X_u^{b}(t_k,x))^{\frac{3}{2}pp_4q_5}}\Big]\Big)^{\frac{1}{p_4q_5}}\\
	&\qquad\times\Big(\widetilde{\E}_{t_k,x}^{b}\Big[ (\sqrt{X_s^{b}(t_k,x)})^{pq_4p_6}\Big]\Big)^{\frac{1}{q_4p_6}}\Big(\widetilde{\E}_{t_k,x}^{b}\Big[(\partial_xX_u^{b}(t_k,x))^{pq_4q_6}\Big]\Big)^{\frac{1}{q_4q_6}}du\\
	&\leq \dfrac{C}{x^{\frac{p}{2}}}\Big(1+\dfrac{1}{x^{\frac{\frac{2a}{\sigma^2}-1}{p_5}-\frac{p}{2}}}\Big)+C\frac{\Delta_n^{p}}{x^{2p}}(1+x^{\frac{p}{2}})\Big(1+\dfrac{1}{x^{\frac{\frac{2a}{\sigma^2}-1}{2p_4p_5}-\frac{p}{2}}}\Big)^2\Big(1+\dfrac{1}{x^{\frac{\frac{2a}{\sigma^2}-1}{2q_4q_6}+\frac{p}{2}}}\Big)\\
	&\qquad +C\frac{\Delta_n^{p}}{x^{\frac{3}{2}p}}(1+x^{\frac{p}{2}})\Big(1+\dfrac{1}{x^{\frac{\frac{2a}{\sigma^2}-1}{2p_4p_5}-\frac{p}{2}}}\Big)^2\Big(1+\dfrac{1}{x^{\frac{\frac{2a}{\sigma^2}-1}{2q_4q_6}+\frac{p}{2}}}\Big)\\
	&\leq C\Big(1+\dfrac{1}{x^{\frac{\frac{2a}{\sigma^2}-1}{p_5}-\frac{p}{2}}}\Big)\Big(\dfrac{1}{x^{\frac{p}{2}}}+\dfrac{1}{x^{3p}}\Big),
	\end{align*}
	where $p_4>1$ and $p_4$ is close to $1$. Here, conditions are required as follows
	\begin{align*}
	2pp_4q_5<\dfrac{2a}{\sigma^2}-1,\; -pp_4p_5\geq -\frac{(\frac{2a}{\sigma^2}-1)^2}{2(\frac{2a}{\sigma^2}-\frac{1}{2})}.	
	\end{align*}
	This implies that
	\begin{align*}
	2pq_5<\dfrac{2a}{\sigma^2}-1,\; pp_5<\frac{(\frac{2a}{\sigma^2}-1)^2}{2(\frac{2a}{\sigma^2}-\frac{1}{2})}.
	\end{align*}
	Thus,
	\begin{align*}
	\frac{2a}{\sigma^2}>2pq_5+1=2p\frac{p_5}{p_5-2}+1,\; \frac{2a}{\sigma^2}> pp_5+\sqrt{pp_5\left(pp_5+1\right)}+1.
	\end{align*}
	Hence, the optimal choice for $p_5$ is solution to
	$$
	2p\frac{p_5}{p_5-2}=pp_5+\sqrt{pp_5\left(pp_5+1\right)}.
	$$	
	The unique positive solution is given by $p_5=\frac{6p+2+2\sqrt{9p^2+2p}}{4p+1}$. Thus, \eqref{dmif1} is valid under condition $\frac{a}{\sigma^2}>\frac{1}{2}(3p+1+\sqrt{9p^2+2p})$.
	
	Next, we treat \eqref{dmif3}. For $t_k\leq s\leq t \leq t_{k+1}$, using \eqref{expression1},  H\"older's inequality with $\frac{1}{p_7}+\frac{1}{q_7}=1$, $\frac{1}{p_8}+\frac{1}{q_8}=1$, Cauchy-Schwarz’s inequality, \eqref{e1} and \eqref{e4}, under condition $2a\geq \sigma^2$ we get that
	\begin{align*}
	&\widetilde{\E}_{t_k,x}^{b}\Big[\Big\vert\dfrac{D_sX_{t}^{b}(t_k,x)}{\partial_xX_{t}^{b}(t_k,x)}\Big\vert^p\Big]=\sigma^p\widetilde{\E}_{t_k,x}^{b}\Big[\Big\vert\frac{\sqrt{X_s^{b}(t_k,x)}\partial_xX_t^{b}(t_k,x)}{\partial_xX_{t}^{b}(t_k,x)\partial_xX_s^{b}(t_k,x)}\Big\vert^p\Big]\\
	&\leq \sigma^p \Big(\widetilde{\E}_{t_k,x}^{b}\Big[\frac{1}{(\partial_xX_{t}^{b}(t_k,x)\partial_xX_s^{b}(t_k,x))^{pp_7}}\Big]\Big)^{\frac{1}{p_7}}\Big(\widetilde{\E}_{t_k,x}^{b}\Big[(\sqrt{X_s^{b}(t_k,x)}\partial_xX_t^{b}(t_k,x))^{pq_7}\Big]\Big)^{\frac{1}{q_7}}\\
	&\leq \sigma^p \Big(\widetilde{\E}_{t_k,x}^{b}\Big[\frac{1}{(\partial_xX_{t}^{b}(t_k,x))^{2pp_7}}\Big]\Big)^{\frac{1}{2p_7}}\Big(\widetilde{\E}_{t_k,x}^{b}\Big[\frac{1}{(\partial_xX_s^{b}(t_k,x))^{2pp_7}}\Big]\Big)^{\frac{1}{2p_7}}\\
	&\qquad\times\Big(\widetilde{\E}_{t_k,x}^{b}\Big[(\sqrt{X_s^{b}(t_k,x)})^{pq_7p_8}\Big]\Big)^{\frac{1}{q_7p_8}}\Big(\widetilde{\E}_{t_k,x}^{b}\Big[(\partial_xX_t^{b}(t_k,x))^{pq_7q_8}\Big]\Big)^{\frac{1}{q_7q_8}}\\
	&\leq C(1+x^{\frac{p}{2}})\Big(1+\dfrac{1}{x^{\frac{\frac{2a}{\sigma^2}-1}{2p_7}-p}}\Big)\Big(1+\dfrac{1}{x^{\frac{\frac{2a}{\sigma^2}-1}{2q_7q_8}+\frac{p}{2}}}\Big),
	\end{align*}
	where $p_7>1$ and $p_7$ is close to $1$. Indeed, the condition required here is $-2pp_7\geq -\frac{(\frac{2a}{\sigma^2}-1)^2}{2(\frac{2a}{\sigma^2}-\frac{1}{2})}$ and $\frac{2a}{\sigma^2}>2p+1+\sqrt{2p(2p+1)}$, thus we only need to choose $p_7\in(1,\frac{(\frac{2a}{\sigma^2}-1)^2}{4p(\frac{2a}{\sigma^2}-\frac{1}{2})})$. Hence, \eqref{dmif3} is valid under condition $\frac{a}{\sigma^2}>\frac{1}{2}(2p+1+\sqrt{2p(2p+1)})$.	
	
	Finally, we treat \eqref{dmif2}. For $t_k\leq s\leq t \leq t_{k+1}$, using \eqref{Mallinveflow1} and \eqref{deriinver}, under $\frac{a}{\sigma^2}>\frac{7}{2}+\sqrt{10}$, we have
	\begin{align}\label{deriinver2}
	&X_{t}^{b}(t_k,x)D_s\Big(\dfrac{1}{\partial_xX_{t}^{b}(t_k,x)}\Big)=-\frac{X_{t}^{b}(t_k,x)}{\partial_xX_t^{b}(t_k,x)}\bigg(\dfrac{\sigma}{2\sqrt{X_s^{b}(t_k,x)}}\notag\\
	&\qquad+\dfrac{\sigma^3}{8}\int_{s}^t\dfrac{\sqrt{X_s^{b}(t_k,x)}\partial_xX_u^{b}(t_k,x)}{(X_u^{b}(t_k,x))^2\partial_xX_s^{b}(t_k,x)}du-\dfrac{\sigma^2}{4}\int_{s}^t\dfrac{\sqrt{X_s^{b}(t_k,x)}\partial_xX_u^{b}(t_k,x)}{(X_u^{b}(t_k,x))^{\frac{3}{2}}\partial_xX_s^{b}(t_k,x)}dB_u\bigg).
	\end{align}
	Then, using BDG's and H\"older's inequalities with $\frac{1}{p_4}+\frac{1}{q_4}=1$, $\frac{1}{p_5}+\frac{1}{p_5}+\frac{1}{q_5}=1$, $\frac{1}{p_6}+\frac{1}{q_6}+\frac{1}{r_6}=1$, \eqref{e1}, \eqref{e2} and \eqref{e4}, we get that
	\begin{align*}
	&\widetilde{\E}_{t_k,x}^{b}\Big[\Big\vert X_{t}^{b}(t_k,x)D_s\Big(\dfrac{1}{\partial_xX_{t}^{b}(t_k,x)}\Big)\Big\vert^p\Big]\leq C\widetilde{\E}_{t_k,x}^{b}\Big[\Big\vert \dfrac{X_{t}^{b}(t_k,x)}{\partial_xX_t^{b}(t_k,x)\sqrt{X_s^{b}(t_k,x)}}\Big\vert^{p}\Big]\\
	&\qquad+C\widetilde{\E}_{t_k,x}^{b}\Big[\Big\vert \int_{s}^t\dfrac{X_{t}^{b}(t_k,x)\sqrt{X_s^{b}(t_k,x)}\partial_xX_u^{b}(t_k,x)}{\partial_xX_t^{b}(t_k,x)(X_u^{b}(t_k,x))^2\partial_xX_s^{b}(t_k,x)}du\Big\vert^{p}\Big]\\
	&\qquad+C\widetilde{\E}_{t_k,x}^{b}\Big[\Big\vert \int_{s}^t\Big(\dfrac{X_{t}^{b}(t_k,x)\sqrt{X_s^{b}(t_k,x)}\partial_xX_u^{b}(t_k,x)}{\partial_xX_t^{b}(t_k,x)(X_u^{b}(t_k,x))^{\frac{3}{2}}\partial_xX_s^{b}(t_k,x)}\Big)^2du\Big\vert^{\frac{p}{2}}\Big]\\
	&\leq C\widetilde{\E}_{t_k,x}^{b}\Big[\Big\vert \dfrac{X_{t}^{b}(t_k,x)}{\partial_xX_t^{b}(t_k,x)\sqrt{X_s^{b}(t_k,x)}}\Big\vert^{p}\Big]\\
	&\qquad+C\Delta_n^{p-1} \int_{s}^t\widetilde{\E}_{t_k,x}^{b}\Big[\Big\vert\dfrac{X_{t}^{b}(t_k,x)\sqrt{X_s^{b}(t_k,x)}\partial_xX_u^{b}(t_k,x)}{\partial_xX_t^{b}(t_k,x)(X_u^{b}(t_k,x))^2\partial_xX_s^{b}(t_k,x)}\Big\vert^{p}\Big]du\\
	&\qquad+C\Delta_n^{\frac{p}{2}-1} \int_{s}^t\widetilde{\E}_{t_k,x}^{b}\Big[\Big\vert\dfrac{X_{t}^{b}(t_k,x)\sqrt{X_s^{b}(t_k,x)}\partial_xX_u^{b}(t_k,x)}{\partial_xX_t^{b}(t_k,x)(X_u^{b}(t_k,x))^{\frac{3}{2}}\partial_xX_s^{b}(t_k,x)}\Big\vert^{p}\Big]du\\
	&\leq C\Big(\widetilde{\E}_{t_k,x}^{b}\Big[\dfrac{1}{(\partial_xX_t^{b}(t_k,x)\sqrt{X_s^{b}(t_k,x)})^{pp_4}}\Big]\Big)^{\frac{1}{p_4}}\Big(\widetilde{\E}_{t_k,x}^{b}\Big[(X_{t}^{b}(t_k,x))^{pq_4}\Big]\Big)^{\frac{1}{q_4}}\\
	&\qquad+C\Delta_n^{p-1} \int_{s}^t\Big(\widetilde{\E}_{t_k,x}^{b}\Big[\dfrac{1}{(\partial_xX_t^{b}(t_k,x)(X_u^{b}(t_k,x))^2\partial_xX_s^{b}(t_k,x))^{pp_4}}\Big]\Big)^{\frac{1}{p_4}}\\
	&\qquad\times \Big(\widetilde{\E}_{t_k,x}^{b}\Big[\Big(X_{t}^{b}(t_k,x)\sqrt{X_s^{b}(t_k,x)}\partial_xX_u^{b}(t_k,x)\Big)^{pq_4}\Big]\Big)^{\frac{1}{q_4}}du\\
	&\qquad+C\Delta_n^{\frac{p}{2}-1} \int_{s}^t\Big(\widetilde{\E}_{t_k,x}^{b}\Big[\dfrac{1}{(\partial_xX_t^{b}(t_k,x)(X_u^{b}(t_k,x))^{\frac{3}{2}}\partial_xX_s^{b}(t_k,x))^{pp_4}}\Big]\Big)^{\frac{1}{p_4}}\\
	&\qquad\times \Big(\widetilde{\E}_{t_k,x}^{b}\Big[\Big(X_{t}^{b}(t_k,x)\sqrt{X_s^{b}(t_k,x)}\partial_xX_u^{b}(t_k,x)\Big)^{pq_4}\Big]\Big)^{\frac{1}{q_4}}du\\
	&\leq C\Big(\widetilde{\E}_{t_k,x}^{b}\Big[\dfrac{1}{(\partial_xX_t^{b}(t_k,x))^{pp_4\frac{p_5}{2}}}\Big]\Big)^{\frac{2}{p_4p_5}}\Big(\widetilde{\E}_{t_k,x}^{b}\Big[\dfrac{1}{(X_s^{b}(t_k,x))^{\frac{1}{2}pp_4q_5}}\Big]\Big)^{\frac{1}{p_4q_5}}\Big(\widetilde{\E}_{t_k,x}^{b}\Big[(X_{t}^{b}(t_k,x))^{pq_4}\Big]\Big)^{\frac{1}{q_4}}\\
	&\qquad+C\Delta_n^{p-1} \int_{s}^t\Big(\widetilde{\E}_{t_k,x}^{b}\Big[\dfrac{1}{(\partial_xX_t^{b}(t_k,x))^{pp_4p_5}}\Big]\Big)^{\frac{1}{p_4p_5}}\\
	&\qquad\times\Big(\widetilde{\E}_{t_k,x}^{b}\Big[\dfrac{1}{(\partial_xX_s^{b}(t_k,x))^{pp_4p_5}}\Big]\Big)^{\frac{1}{p_4p_5}}\Big(\widetilde{\E}_{t_k,x}^{b}\Big[\dfrac{1}{(X_u^{b}(t_k,x))^{2pp_4q_5}}\Big]\Big)^{\frac{1}{p_4q_5}}\\
	&\qquad\times\Big(\widetilde{\E}_{t_k,x}^{b}\Big[ (\sqrt{X_s^{b}(t_k,x)})^{pq_4p_6}\Big]\Big)^{\frac{1}{q_4p_6}}\Big(\widetilde{\E}_{t_k,x}^{b}\Big[(\partial_xX_u^{b}(t_k,x))^{pq_4q_6}\Big]\Big)^{\frac{1}{q_4q_6}}\\
	&\qquad\times \Big(\widetilde{\E}_{t_k,x}^{b}\Big[(X_t^{b}(t_k,x))^{pq_4r_6}\Big]\Big)^{\frac{1}{q_4r_6}}du+C\Delta_n^{\frac{p}{2}-1} \int_{s}^t\Big(\widetilde{\E}_{t_k,x}^{b}\Big[\dfrac{1}{(\partial_xX_t^{b}(t_k,x))^{pp_4p_5}}\Big]\Big)^{\frac{1}{p_4p_5}}\\
	&\qquad\times\Big(\widetilde{\E}_{t_k,x}^{b}\Big[\dfrac{1}{(\partial_xX_s^{b}(t_k,x))^{pp_4p_5}}\Big]\Big)^{\frac{1}{p_4p_5}}\Big(\widetilde{\E}_{t_k,x}^{b}\Big[\dfrac{1}{(X_u^{b}(t_k,x))^{\frac{3}{2}pp_4q_5}}\Big]\Big)^{\frac{1}{p_4q_5}}\\
	&\qquad\times\Big(\widetilde{\E}_{t_k,x}^{b}\Big[ (\sqrt{X_s^{b}(t_k,x)})^{pq_4p_6}\Big]\Big)^{\frac{1}{q_4p_6}}\Big(\widetilde{\E}_{t_k,x}^{b}\Big[(\partial_xX_u^{b}(t_k,x))^{pq_4q_6}\Big]\Big)^{\frac{1}{q_4q_6}}\\
	&\qquad\times \Big(\widetilde{\E}_{t_k,x}^{b}\Big[(X_t^{b}(t_k,x))^{pq_4r_6}\Big]\Big)^{\frac{1}{q_4r_6}}du\\
	&\leq C(1+x^p)\dfrac{1}{x^{\frac{p}{2}}}\Big(1+\dfrac{1}{x^{\frac{\frac{2a}{\sigma^2}-1}{p_4p_5}-\frac{p}{2}}}\Big)+C\frac{\Delta_n^{p}}{x^{2p}}(1+x^{\frac{3p}{2}})\Big(1+\dfrac{1}{x^{\frac{\frac{2a}{\sigma^2}-1}{2p_4p_5}-\frac{p}{2}}}\Big)^2\Big(1+\dfrac{1}{x^{\frac{\frac{2a}{\sigma^2}-1}{2q_4q_6}+\frac{p}{2}}}\Big)\\
	&\qquad +C\frac{\Delta_n^{p}}{x^{\frac{3}{2}p}}(1+x^{\frac{3p}{2}})\Big(1+\dfrac{1}{x^{\frac{\frac{2a}{\sigma^2}-1}{2p_4p_5}-\frac{p}{2}}}\Big)^2\Big(1+\dfrac{1}{x^{\frac{\frac{2a}{\sigma^2}-1}{2q_4q_6}+\frac{p}{2}}}\Big)\\
	&\leq C\Big(\frac{1}{x^{\frac{p}{2}}}+\frac{1}{x^{2p}}\Big)(1+x^{\frac{3p}{2}})\Big(1+\dfrac{1}{x^{\frac{\frac{2a}{\sigma^2}-1}{2p_4p_5}-\frac{p}{2}}}\Big)^2\Big(1+\dfrac{1}{x^{\frac{\frac{2a}{\sigma^2}-1}{2q_4q_6}+\frac{p}{2}}}\Big)\\
	&\leq C\frac{1}{x^{2p}}(1+x^{3p})\Big(1+\dfrac{1}{x^{\frac{\frac{2a}{\sigma^2}-1}{p_4p_5}-p}}\Big)\Big(1+\dfrac{1}{x^{\frac{\frac{2a}{\sigma^2}-1}{2q_4q_6}+\frac{p}{2}}}\Big),
	\end{align*}
	where $p_4>1$ and $p_4$ is close to $1$. Here, conditions are required as follows
	\begin{align*}
	2pp_4q_5<\dfrac{2a}{\sigma^2}-1,\; -pp_4p_5\geq -\frac{(\frac{2a}{\sigma^2}-1)^2}{2(\frac{2a}{\sigma^2}-\frac{1}{2})}.	
	\end{align*}
	This implies that
	\begin{align*}
	2pq_5<\dfrac{2a}{\sigma^2}-1,\; pp_5<\frac{(\frac{2a}{\sigma^2}-1)^2}{2(\frac{2a}{\sigma^2}-\frac{1}{2})}.
	\end{align*}
	Thus,
	\begin{align*}
	\frac{2a}{\sigma^2}>2pq_5+1=2p\frac{p_5}{p_5-2}+1,\; \frac{2a}{\sigma^2}> pp_5+\sqrt{pp_5\left(pp_5+1\right)}+1.
	\end{align*}
	Hence, the optimal choice for $p_5$ is solution to
	$$
	2p\frac{p_5}{p_5-2}=pp_5+\sqrt{pp_5\left(pp_5+1\right)}.
	$$	
	The unique positive solution is given by $p_5=\frac{6p+2+2\sqrt{9p^2+2p}}{4p+1}$. Thus, \eqref{dmif2} is valid under condition $\frac{a}{\sigma^2}>\frac{1}{2}(3p+1+\sqrt{9p^2+2p})$.
\end{proof}

\subsection{Proof of Lemma \ref{Malliderivable}}
\begin{proof}
	\textnormal{(i)} For $t_k\leq s\leq t \leq t_{k+1}$, using \eqref{expression1}, we have	
	\begin{align*}
	&\partial_xX_t^{b}(t_k,x)\bigg(\dfrac{\sigma}{2\sqrt{X_s^{b}(t_k,x)}}+\dfrac{\sigma^2}{8}\int_{s}^t\dfrac{D_sX_u^{b}(t_k,x)}{(X_u^{b}(t_k,x))^2}du-\dfrac{\sigma}{4}\int_{s}^t\dfrac{D_sX_u^{b}(t_k,x)}{(X_u^{b}(t_k,x))^{\frac{3}{2}}}dB_u\bigg)\notag\\
	&=\partial_xX_t^{b}(t_k,x)\bigg(\dfrac{\sigma}{2\sqrt{X_s^{b}(t_k,x)}}+\dfrac{\sigma^3}{8}\int_{s}^t\dfrac{\sqrt{X_s^{b}(t_k,x)}\partial_xX_u^{b}(t_k,x)}{(X_u^{b}(t_k,x))^2\partial_xX_s^{b}(t_k,x)}du\\
	&\qquad-\dfrac{\sigma^2}{4}\int_{s}^t\dfrac{\sqrt{X_s^{b}(t_k,x)}\partial_xX_u^{b}(t_k,x)}{(X_u^{b}(t_k,x))^{\frac{3}{2}}\partial_xX_s^{b}(t_k,x)}dB_u\bigg).
	\end{align*}
	We will show that under condition $\frac{a}{\sigma^2}>\frac{5+3\sqrt{2}}{2}$, this expression is contained in $L^2(\widetilde{\Omega}\times [t_k,t_{k+1}])$. In fact, using BDG's and H\"older's inequalities with $\frac{1}{p}+\frac{1}{q}=1$, $\frac{1}{p_1}+\frac{1}{q_1}=1$, $\frac{1}{p_2}+\frac{1}{q_2}+\frac{1}{r_2}=1$, \eqref{e1}, \eqref{e2} and \eqref{e4}, we get 
	\begin{align*}
	&\widetilde{\E}_{t_k,x}^{b}\bigg[\bigg\vert \partial_xX_t^{b}(t_k,x)\bigg(\dfrac{\sigma}{2\sqrt{X_s^{b}(t_k,x)}}+\dfrac{\sigma^2}{8}\int_{s}^t\dfrac{D_sX_u^{b}(t_k,x)}{(X_u^{b}(t_k,x))^2}du-\dfrac{\sigma}{4}\int_{s}^t\dfrac{D_sX_u^{b}(t_k,x)}{(X_u^{b}(t_k,x))^{\frac{3}{2}}}dB_u\bigg)\bigg\vert^2\bigg]\\
	&\leq C\widetilde{\E}_{t_k,x}^{b}\Big[\Big\vert \frac{\partial_xX_t^{b}(t_k,x)}{\sqrt{X_s^{b}(t_k,x)}}\Big\vert^2\Big]+C\widetilde{\E}_{t_k,x}^{b}\Big[\Big\vert \partial_xX_t^{b}(t_k,x)\int_{s}^t\dfrac{\sqrt{X_s^{b}(t_k,x)}\partial_xX_u^{b}(t_k,x)}{(X_u^{b}(t_k,x))^2\partial_xX_s^{b}(t_k,x)}du\Big\vert^2\Big]\\
	&\qquad+C \int_{s}^t\widetilde{\E}_{t_k,x}^{b}\Big[\Big\vert\dfrac{\partial_xX_t^{b}(t_k,x)\sqrt{X_s^{b}(t_k,x)}\partial_xX_u^{b}(t_k,x)}{(X_u^{b}(t_k,x))^{\frac{3}{2}}\partial_xX_s^{b}(t_k,x)}\Big\vert^2\Big]du\\
	&\leq C\widetilde{\E}_{t_k,x}^{b}\Big[\Big\vert \frac{\partial_xX_t^{b}(t_k,x)}{\sqrt{X_s^{b}(t_k,x)}}\Big\vert^2\Big]+C\Delta_n \int_{s}^t\widetilde{\E}_{t_k,x}^{b}\Big[\Big\vert\dfrac{\partial_xX_t^{b}(t_k,x)\sqrt{X_s^{b}(t_k,x)}\partial_xX_u^{b}(t_k,x)}{(X_u^{b}(t_k,x))^2\partial_xX_s^{b}(t_k,x)}\Big\vert^2\Big]du\\
	&\qquad+C \int_{s}^t\widetilde{\E}_{t_k,x}^{b}\Big[\Big\vert\dfrac{\partial_xX_t^{b}(t_k,x)\sqrt{X_s^{b}(t_k,x)}\partial_xX_u^{b}(t_k,x)}{(X_u^{b}(t_k,x))^{\frac{3}{2}}\partial_xX_s^{b}(t_k,x)}\Big\vert^2\Big]du\\
	&\leq C\Big(\widetilde{\E}_{t_k,x}^{b}\Big[\dfrac{1}{\vert X_s^{b}(t_k,x)\vert^p}\Big]\Big)^{\frac{1}{p}}\Big(\widetilde{\E}_{t_k,x}^{b}\Big[\big\vert \partial_xX_t^{b}(t_k,x)\big\vert^{2q}\Big]\Big)^{\frac{1}{q}}\\
	&\qquad+C\Delta_n\int_{s}^t\Big(\widetilde{\E}_{t_k,x}^{b}\Big[\frac{1}{\vert (X_u^{b}(t_k,x))^2\partial_xX_s^{b}(t_k,x) \vert^{2p}}\Big]\Big)^{\frac{1}{p}}\\
	&\qquad\times\Big(\widetilde{\E}_{t_k,x}^{b}\Big[\big\vert \partial_xX_t^{b}(t_k,x)\sqrt{X_s^{b}(t_k,x)}\partial_xX_u^{b}(t_k,x)\big\vert^{2q}\Big]\Big)^{\frac{1}{q}}du\\
	&\qquad+C\int_{s}^t\Big(\widetilde{\E}_{t_k,x}^{b}\Big[\frac{1}{\vert (X_u^{b}(t_k,x))^{\frac{3}{2}}\partial_xX_s^{b}(t_k,x) \vert^{2p}}\Big]\Big)^{\frac{1}{p}}\\
	&\qquad\times\Big(\widetilde{\E}_{t_k,x}^{b}\Big[\big\vert \partial_xX_t^{b}(t_k,x)\sqrt{X_s^{b}(t_k,x)}\partial_xX_u^{b}(t_k,x)\big\vert^{2q}\Big]\Big)^{\frac{1}{q}}du\\
	&\leq \dfrac{C}{x}\Big(1+\dfrac{1}{x^{\frac{\frac{2a}{\sigma^2}-1}{2q}+1}}\Big)+C\Delta_n\int_{s}^t\Big(\widetilde{\E}_{t_k,x}^{b}\Big[\frac{1}{(X_u^{b}(t_k,x))^{4pp_1}}\Big]\Big)^{\frac{1}{pp_1}}\Big(\widetilde{\E}_{t_k,x}^{b}\Big[\frac{1}{(\partial_xX_s^{b}(t_k,x))^{2pq_1}}\Big]\Big)^{\frac{1}{pq_1}}\\
	&\qquad\times\Big(\widetilde{\E}_{t_k,x}^{b}\Big[(\partial_xX_t^{b}(t_k,x))^{2qp_2}\Big]\Big)^{\frac{1}{qp_2}}\Big(\widetilde{\E}_{t_k,x}^{b}\Big[ (X_s^{b}(t_k,x))^{qq_2}\Big]\Big)^{\frac{1}{qq_2}}\\
	&\qquad\times\Big(\widetilde{\E}_{t_k,x}^{b}\Big[(\partial_xX_u^{b}(t_k,x))^{2qr_2}\Big]\Big)^{\frac{1}{qr_2}}du+C\int_{s}^t\Big(\widetilde{\E}_{t_k,x}^{b}\Big[\frac{1}{(X_u^{b}(t_k,x))^{3pp_1}}\Big]\Big)^{\frac{1}{pp_1}}\\
	&\qquad\times\Big(\widetilde{\E}_{t_k,x}^{b}\Big[\frac{1}{(\partial_xX_s^{b}(t_k,x))^{2pq_1}}\Big]\Big)^{\frac{1}{pq_1}}\Big(\widetilde{\E}_{t_k,x}^{b}\Big[(\partial_xX_t^{b}(t_k,x))^{2qp_2}\Big]\Big)^{\frac{1}{qp_2}}\\
	&\qquad\times\Big(\widetilde{\E}_{t_k,x}^{b}\Big[ (X_s^{b}(t_k,x))^{qq_2}\Big]\Big)^{\frac{1}{qq_2}}\Big(\widetilde{\E}_{t_k,x}^{b}\Big[(\partial_xX_u^{b}(t_k,x))^{2qr_2}\Big]\Big)^{\frac{1}{qr_2}}du\\
	&\leq \dfrac{C}{x}\Big(1+\dfrac{1}{x^{\frac{\frac{2a}{\sigma^2}-1}{2q}+1}}\Big)+C\dfrac{\Delta_n^2}{x^4}(1+x)\Big(1+\dfrac{1}{x^{\frac{\frac{2a}{\sigma^2}-1}{2pq_1}-1}}\Big)\Big(1+\dfrac{1}{x^{\frac{\frac{2a}{\sigma^2}-1}{2qp_2}+1}}\Big)\Big(1+\dfrac{1}{x^{\frac{\frac{2a}{\sigma^2}-1}{2qr_2}+1}}\Big)\\
	&\qquad +C\dfrac{\Delta_n}{x^3}(1+x)\Big(1+\dfrac{1}{x^{\frac{\frac{2a}{\sigma^2}-1}{2pq_1}-1}}\Big)\Big(1+\dfrac{1}{x^{\frac{\frac{2a}{\sigma^2}-1}{2qp_2}+1}}\Big)\Big(1+\dfrac{1}{x^{\frac{\frac{2a}{\sigma^2}-1}{2qr_2}+1}}\Big),
	\end{align*}
	where $p>1$ and $p$ is close to $1$. Here, conditions are required as follows
	\begin{align*}
4pp_1<\dfrac{2a}{\sigma^2}-1,\; -2pq_1\geq -\frac{(\frac{2a}{\sigma^2}-1)^2}{2(\frac{2a}{\sigma^2}-\frac{1}{2})}.	
\end{align*}
	This implies that
	\begin{align*}
	4p_1<\dfrac{2a}{\sigma^2}-1,\; 2q_1<\frac{(\frac{2a}{\sigma^2}-1)^2}{2(\frac{2a}{\sigma^2}-\frac{1}{2})}.
	\end{align*}
	Thus,
	\begin{align*}
	\frac{2a}{\sigma^2}>4p_1+1=4\frac{q_1}{q_1-1}+1,\; \frac{2a}{\sigma^2}> 2q_1+\sqrt{2q_1\left(2q_1+1\right)}+1.
	\end{align*}
	Hence, the optimal choice for $p_1$ is solution to
	$$
	4\frac{q_1}{q_1-1}=2q_1+\sqrt{2q_1\left(2q_1+1\right)}.
	$$	
	The unique positive solution is $q_1=1+\frac{2\sqrt{2}}{3}$. Thus, $\frac{a}{\sigma^2}>\frac{5+3\sqrt{2}}{2}$.
	Thus, under $\frac{a}{\sigma^2}>\frac{5+3\sqrt{2}}{2}$,	
	\begin{align*}
	&\widetilde{\E}_{t_k,x}^{b}\bigg[\int_{t_k}^{t_{k+1}}\Big\vert \partial_xX_t^{b}(t_k,x)\Big(\dfrac{\sigma}{2\sqrt{X_s^{b}(t_k,x)}}+\dfrac{\sigma^2}{8}\int_{s}^t\dfrac{D_sX_u^{b}(t_k,x)}{(X_u^{b}(t_k,x))^2}du\\
	&\qquad-\dfrac{\sigma}{4}\int_{s}^t\dfrac{D_sX_u^{b}(t_k,x)}{(X_u^{b}(t_k,x))^{\frac{3}{2}}}dB_u\Big)\Big\vert^2ds\bigg]<+\infty.
	\end{align*}	
	This combined with \textcolor{black}{\eqref{dxe} and} \eqref{e4} with $p=2$ guarantees that $\partial_xX_t^{b}(t_k,x)\in \mathbb{D}^{1,2}$ under $\frac{a}{\sigma^2}>\frac{5+3\sqrt{2}}{2}$ thanks to Theorem \ref{existenceMallideri}, and furthermore, its Malliavin derivative is given by \eqref{dmpx}.

	\textnormal{(ii)} In the same way 
	we show that the expression $\frac{-1}{(\partial_xX_t^{b}(t_k,x))^2}D(\partial_xX_t^{b}(t_k,x))$ is contained in $L^2(\widetilde{\Omega}\times [t_k,t_{k+1}])$ under condition $\frac{a}{\sigma^2}>\frac{7}{2}+\sqrt{10}$. Indeed, applying \eqref{dmif1} of Lemma \ref{Malliflow} with $p=2$, condition $\frac{a}{\sigma^2}>\frac{7}{2}+\sqrt{10}$ ensures that
	\begin{align*}
	\widetilde{\E}_{t_k,x}^{b}\Big[\Big\vert \frac{D_s(\partial_xX_t^{b}(t_k,x))}{(\partial_xX_t^{b}(t_k,x))^2}\Big\vert^2\Big]&\leq C\Big(1+\dfrac{1}{x^{\frac{\frac{2a}{\sigma^2}-1}{p_5}-1}}\Big)\Big(\dfrac{1}{x}+\dfrac{1}{x^{6}}\Big),
	\end{align*}
	where $p_5=\frac{15+4\sqrt{10}}{9}$. This implies that
	\begin{align*}
	&\widetilde{\E}_{t_k,x}^{b}\Big[\int_{t_k}^{t_{k+1}}\Big\vert \frac{D_s(\partial_xX_t^{b}(t_k,x))}{(\partial_xX_t^{b}(t_k,x))^2}\Big\vert^2ds\Big]\leq C\Delta_n\Big(1+\dfrac{1}{x^{\frac{\frac{2a}{\sigma^2}-1}{p_5}-1}}\Big)\Big(\dfrac{1}{x}+\dfrac{1}{x^{6}}\Big)<+\infty.
	\end{align*}
	Moreover, applying Lemma \ref{estimates} with $p=-2$, \eqref{e4} is satisfied under condition $\frac{a}{\sigma^2}>\frac{3+\sqrt{6}}{2}$. This guarantees that  $(\partial_xX_t^{b}(t_k,x))^{-1}\in \mathbb{D}^{1,2}$ under condition $\frac{a}{\sigma^2}>\frac{7}{2}+\sqrt{10}$ thanks to Lemma \ref{existenceMallideri} and, furthermore, its Malliavin derivative is given by \eqref{Mallinveflow1}. 
	
	Finally, in the same way 
	we show that under condition $\frac{a}{\sigma^2}>\frac{7}{2}+\sqrt{10}$,
	$$
	\dfrac{1}{\partial_xX_{t}^{b}(t_k,x)}DX_{t}^{b}(t_k,x)+X_{t}^{b}(t_k,x)D\Big(\dfrac{1}{\partial_xX_{t}^{b}(t_k,x)}\Big)
	$$
	is contained in $L^2(\widetilde{\Omega}\times [t_k,t_{k+1}])$. Indeed, applying \eqref{dmif3} and \eqref{dmif2} of Lemma \ref{Malliflow} with $p=2$, condition $\frac{a}{\sigma^2}>\frac{7}{2}+\sqrt{10}$ ensures that
	\begin{align*}
	&\widetilde{\E}_{t_k,x}^{b}\Big[\Big\vert \dfrac{1}{\partial_xX_{t}^{b}(t_k,x)}D_sX_{t}^{b}(t_k,x)+X_{t}^{b}(t_k,x)D_s\Big(\dfrac{1}{\partial_xX_{t}^{b}(t_k,x)}\Big)\Big\vert^2\Big]\\
	&\leq 2\widetilde{\E}_{t_k,x}^{b}\Big[\Big\vert \dfrac{1}{\partial_xX_{t}^{b}(t_k,x)}D_sX_{t}^{b}(t_k,x)\Big\vert^2\Big]+2\widetilde{\E}_{t_k,x}^{b}\Big[\Big\vert X_{t}^{b}(t_k,x)D_s\Big(\dfrac{1}{\partial_xX_{t}^{b}(t_k,x)}\Big)\Big\vert^2\Big]\\
	&\leq C(1+x)\Big(1+\dfrac{1}{x^{\frac{\frac{2a}{\sigma^2}-1}{2p_7}-2}}\Big)\Big(1+\dfrac{1}{x^{\frac{\frac{2a}{\sigma^2}-1}{2q_7q_8}+1}}\Big)+\frac{C}{x^{4}}(1+x^{6})\Big(1+\dfrac{1}{x^{\frac{\frac{2a}{\sigma^2}-1}{p_4p_5}-2}}\Big)\Big(1+\dfrac{1}{x^{\frac{\frac{2a}{\sigma^2}-1}{2q_4q_6}+1}}\Big).
	\end{align*}
	This implies that under condition $\frac{a}{\sigma^2}>\frac{7}{2}+\sqrt{10}$,
	\begin{align*}
	&\widetilde{\E}_{t_k,x}^{b}\Big[\int_{t_k}^{t_{k+1}}\Big\vert \dfrac{1}{\partial_xX_{t}^{b}(t_k,x)}D_sX_{t}^{b}(t_k,x)+X_{t}^{b}(t_k,x)D_s\Big(\dfrac{1}{\partial_xX_{t}^{b}(t_k,x)}\Big)\Big\vert^2ds\Big]<+\infty.
	\end{align*}
	On the other hand, using H\"older's inequality with $\frac{1}{p_3}+\frac{1}{q_3}=1$, we get
	\begin{align*} \widetilde{\E}_{t_k,x}^{b}\Big[\Big\vert\dfrac{X_{t}^{b}(t_k,x)}{\partial_xX_{t}^{b}(t_k,x)}\Big\vert^2\Big]&\leq \Big(\widetilde{\E}_{t_k,x}^{b}\Big[\dfrac{1}{(\partial_xX_{t}^{b}(t_k,x))^{2p_3}}\Big]\Big)^{\frac{1}{p_3}}\Big(\widetilde{\E}_{t_k,x}^{b}\Big[(X_{t}^{b}(t_k,x))^{2q_3}\Big]\Big)^{\frac{1}{q_3}}\\
	&\leq C\Big(1+\dfrac{1}{x^{\frac{\frac{2a}{\sigma^2}-1}{2p_3}-1}}\Big)(1+x^2)<+\infty,
	\end{align*}
	where $p_3>1$ and $p_3$ is close to $1$. Indeed, the condition required here is $-2p_3\geq -\frac{(\frac{2a}{\sigma^2}-1)^2}{2(\frac{2a}{\sigma^2}-\frac{1}{2})}$ and $\frac{a}{\sigma^2}>\frac{3+\sqrt{6}}{2}$, thus we only need to choose $p_3\in(1,\frac{(\frac{2a}{\sigma^2}-1)^2}{4(\frac{2a}{\sigma^2}-\frac{1}{2})})$. Hence $\frac{X_{t}^{b}(t_k,x)}{\partial_xX_{t}^{b}(t_k,x)}\in L^2(\widetilde{\Omega})$ under $\frac{a}{\sigma^2}>\frac{3+\sqrt{6}}{2}$. This guarantees thanks to Theorem \ref{existenceMallideri} that  $\frac{X_{t}^{b}(t_k,x)}{\partial_xX_{t}^{b}(t_k,x)}\in \mathbb{D}^{1,2}$ under condition $\frac{a}{\sigma^2}>\frac{7}{2}+\sqrt{10}$ and, furthermore, its Malliavin derivative is given by \eqref{Mallinveflow2}. Thus, the result follows.	
\end{proof}

\subsection{Proof of Proposition \ref{c2prop1}}
\begin{proof}
	We are going to apply Theorem \ref{ruleproNualart}. First, we wish to show  $\partial_{b}X_{t_{k+1}}^{b}(t_k,x)U^{b}(t_k,x)\in \textnormal{Dom}\ \delta$ under condition $\frac{a}{\sigma^2}>\frac{7}{2}+\sqrt{10}$. For this, we write 
	\begin{align}\label{Fu2}
	\partial_{b}X_{t_{k+1}}^{b}(t_k,x)U_{\cdot}^{b}(t_k,x)&=\frac{\partial_{b}X_{t_{k+1}}^{b}(t_k,x)}{\sigma\sqrt{X_{\cdot}^{b}(t_k,x)}}(\partial_xX_{t_{k+1}}^{b}(t_k,x))^{-1}\partial_xX_{\cdot}^{b}(t_k,x)=Fu_{\cdot},
	\end{align}
	where $F=\partial_{b}X_{t_{k+1}}^{b}(t_k,x)(\partial_xX_{t_{k+1}}^{b}(t_k,x))^{-1}$ and $u_{\cdot}=\frac{1}{\sigma\sqrt{X_{\cdot}^{b}(t_k,x)}}\partial_xX_{\cdot}^{b}(t_k,x)$. Here $u=(u_t, t\in [t_k,t_{k+1}])$ is an adapted process then it belongs to $\textnormal{Dom}\ \delta$. By Theorem \ref{ruleproNualart}, it suffices to show that $F\in \mathbb{D}^{1,2}$ and $Fu\in L^2(\widetilde{\Omega};H)\cong L^2([t_k,t_{k+1}]\times \widetilde{\Omega},\R)$, where $H=L^2([t_k,t_{k+1}],\R)$. From \eqref{dxb2}, we have
	\begin{align*}
	F=\dfrac{\partial_{b}X_{t_{k+1}}^{b}(t_k,x)}{\partial_xX_{t_{k+1}}^{b}(t_k,x)}=-\int_{t_k}^{t_{k+1}}\dfrac{X_{r}^{b}(t_k,x)}{\partial_xX_{r}^{b}(t_k,x)}dr.
	\end{align*}	
	By assertion (ii) of Lemma \ref{Malliderivable}, under condition $\frac{a}{\sigma^2}>\frac{7}{2}+\sqrt{10}$, $\frac{X_{t}^{b}(t_k,x)}{\partial_xX_{t}^{b}(t_k,x)}\in \mathbb{D}^{1,2}$ for any $t\in [t_k,t_{k+1}]$. Thus, it is straightforward that
	$F\in \mathbb{D}^{1,2}$ under condition $\frac{a}{\sigma^2}>\frac{7}{2}+\sqrt{10}$. Furthermore, for $t_k\leq s\leq r \leq t_{k+1}$,
	\begin{align*}
	D_sF=-\int_{s}^{t_{k+1}}D_s\Big(\dfrac{X_{r}^{b}(t_k,x)}{\partial_xX_{r}^{b}(t_k,x)}\Big)dr.
	\end{align*}	
	Next, we check $Fu\in L^2([t_k,t_{k+1}]\times \widetilde{\Omega},\R)$. For this, using H\"older's inequality repeatedly with $\frac{1}{p}+\frac{1}{q}=1$, $\frac{1}{p_1}+\frac{1}{q_1}=1$  and $\frac{1}{p_2}+\frac{1}{q_2}=1$, together with \eqref{e1}-\eqref{e4}, we get that
	\begin{align*}
	&\widetilde{\E}_{t_k,x}^{b}\big[(Fu_t)^2\big]=\widetilde{\E}_{t_k,x}^{b}\Big[\Big(\int_{t_k}^{t_{k+1}}\dfrac{X_{r}^{b}(t_k,x)}{\partial_xX_{r}^{b}(t_k,x)}dr\frac{1}{\sigma\sqrt{X_{t}^{b}(t_k,x)}}\partial_xX_{t}^{b}(t_k,x)\Big)^{2}\Big]\\
	&\leq\dfrac{\Delta_n}{\sigma^2}\int_{t_k}^{t_{k+1}}\widetilde{\E}_{t_k,x}^{b}\Big[\Big\vert\dfrac{X_{r}^{b}(t_k,x)\partial_xX_{t}^{b}(t_k,x)}{\partial_xX_{r}^{b}(t_k,x)\sqrt{X_{t}^{b}(t_k,x)}}\Big\vert^2\Big]dr\\
	&\leq\dfrac{\Delta_n}{\sigma^2}\int_{t_k}^{t_{k+1}}\Big(\widetilde{\E}_{t_k,x}^{b}\Big[\dfrac{1}{(\partial_xX_{r}^{b}(t_k,x)\sqrt{X_{t}^{b}(t_k,x)})^{2p}}\Big]\Big)^{\frac{1}{p}}\Big(\widetilde{\E}_{t_k,x}^{b}\Big[(X_{r}^{b}(t_k,x)\partial_xX_{t}^{b}(t_k,x))^{2q}\Big]\Big)^{\frac{1}{q}}dr\\
		&\leq\dfrac{\Delta_n}{\sigma^2}\int_{t_k}^{t_{k+1}}\Big(\widetilde{\E}_{t_k,x}^{b}\Big[\dfrac{1}{(\partial_xX_{r}^{b}(t_k,x))^{2pp_1}}\Big]\Big)^{\frac{1}{pp_1}}\Big(\widetilde{\E}_{t_k,x}^{b}\Big[\dfrac{1}{(X_{t}^{b}(t_k,x))^{pq_1}}\Big]\Big)^{\frac{1}{pq_1}}\\
		&\qquad\times\Big(\widetilde{\E}_{t_k,x}^{b}\Big[(X_{r}^{b}(t_k,x))^{2qp_2}\Big]\Big)^{\frac{1}{qp_2}}\Big(\widetilde{\E}_{t_k,x}^{b}\Big[(\partial_xX_{t}^{b}(t_k,x))^{2qq_2}\Big]\Big)^{\frac{1}{qq_2}}dr\\
		&\leq \dfrac{C\Delta_n^2}{x}(1+x^2)\Big(1+\dfrac{1}{x^{\frac{\frac{2a}{\sigma^2}-1}{2pp_1}-1}}\Big)\Big(1+\dfrac{1}{x^{\frac{\frac{2a}{\sigma^2}-1}{2qq_2}+1}}\Big),
	\end{align*}
	for some constant $C>0$, where $p>1$ and $p$ is close to $1$. This shows that
	\begin{align*}
	\widetilde{\E}_{t_k,x}^{b}\Big[\int_{t_k}^{t_{k+1}}\left(Fu_t\right)^2dt\Big]&=\int_{t_k}^{t_{k+1}}\widetilde{\E}_{t_k,x}^{b}\Big[(Fu_t)^2\Big]dt<+\infty.
	\end{align*}
	All conditions required here are as follows
	\begin{align*}
	-2pp_1\geq -\frac{(\frac{2a}{\sigma^2}-1)^2}{2(\frac{2a}{\sigma^2}-\frac{1}{2})},\;  pq_1<\dfrac{2a}{\sigma^2}-1.
	\end{align*}
	That is,
	\begin{align*}
	-2p_1>-\frac{(\frac{2a}{\sigma^2}-1)^2}{2(\frac{2a}{\sigma^2}-\frac{1}{2})},\;  q_1<\dfrac{2a}{\sigma^2}-1.
	\end{align*}
	This implies that
	\begin{align*}
	\frac{2a}{\sigma^2}> 2p_1+\sqrt{2p_1\left(2p_1+1\right)}+1,\;
	\frac{2a}{\sigma^2}>q_1+1=\frac{p_1}{p_1-1}+1.
	\end{align*}	
	Hence, the optimal choice for $p_1$ is solution to
	$$
	2p_1+\sqrt{2p_1\left(2p_1+1\right)}=\frac{p_1}{p_1-1}.
	$$	
	The unique positive solution is given by $p_1=\frac{9+\sqrt{33}}{12}$. Thus, $Fu\in L^2([t_k,t_{k+1}]\times \widetilde{\Omega},\R)$ under $\frac{a}{\sigma^2}>\frac{7+\sqrt{33}}{4}$. Hence, we have 
	shown that $\partial_{b}X_{t_{k+1}}^{b}(t_k,x)U^{b}(t_k,x)\in \textnormal{Dom}\ \delta$ under condition $\frac{a}{\sigma^2}>\frac{7}{2}+\sqrt{10}$. 	
	
	Next, under condition $\frac{a}{\sigma^2}>\frac{7}{2}+\sqrt{10}$ we proceed as in the proof of Proposition 3.1 of Kohatsu-Higa {\it et al.} (2017) with $\beta=b$ (see pages 441 and 442) to get the following representation of the score function
	\begin{align*}
	\dfrac{\partial_{b}p^{b}}{p^{b}}\left(\Delta_n,x,y\right)=\dfrac{1}{\Delta_n}\widetilde{\E}_{t_k,x}^{b}\left[\delta\left(\partial_{b}X_{t_{k+1}}^{b}(t_k,x)U^{b}(t_k,x)\right)\big\vert X_{t_{k+1}}^{b}=y\right].
\end{align*}	
	Finally, we show \eqref{derib}. In fact, using condition $\frac{a}{\sigma^2}>\frac{7}{2}+\sqrt{10}$ and the fact that the Skorohod integral and the It\^o integral of an adapted process coincide $\delta(u)=\int_{t_k}^{t_{k+1}}\frac{\partial_{x}X_{s}^{b}(t_k,x)}{\sigma\sqrt{X_s^{b}(t_k,x)}}dB_s$, we have		
	\begin{equation*}\begin{split}
	F\delta(u)-\left<DF,u\right>_H&=-\int_{t_k}^{t_{k+1}}\frac{X_r^{b}(t_k,x)}{\partial_xX_r^{b}(t_k,x)}dr\int_{t_k}^{t_{k+1}}\frac{\partial_{x}X_{s}^{b}(t_k,x)}{\sigma\sqrt{X_s^{b}(t_k,x)}}dB_s\\
	&\qquad+\int_{t_k}^{t_{k+1}}\int_{s}^{t_{k+1}}D_s\Big(\dfrac{X_{r}^{b}(t_k,x)}{\partial_xX_{r}^{b}(t_k,x)}\Big)dr\frac{\partial_{x}X_{s}^{b}(t_k,x)}{\sigma\sqrt{X_s^{b}(t_k,x)}}ds.
	\end{split}
	\end{equation*}
	We next add and subtract the term $
	\frac{X_{t_k}^{b}(t_k,x)}{\partial_{x}X_{t_k}^{b}(t_k,x)}$ in the first integral, and the term $\frac{\partial_{x}X_{t_k}^{b}(t_k,x)}{\sigma\sqrt{X_{t_k}^{b}(t_k,x)}}$ in the second integral. This, together with $X_{t_k}^{b}(t_k,x)=x$, shows that
	\begin{equation}\label{e0}
	F\delta(u)-\left<DF,u\right>_H
	=-\frac{\Delta_n}{\sigma}\sqrt{x} \left(B_{t_{k+1}}-B_{t_{k}}\right)+H_1^{b}+H_2^{b}+H_3^{b},
	\end{equation}	
where $H_1^{b}$, $H_2^{b}$, $H_3^{b}$ are given in Proposition \ref{c2prop1}. Then, applying \eqref{es4} of Lemma \ref{estimate} with $q=1$, under condition $\frac{a}{\sigma^2}>\frac{15+\sqrt{185}}{4}$, we get 
	\begin{align*}
	\widetilde{\E}_{t_k,x}^{b}\Big[\big(F\delta(u)-\left<DF,u\right>_H\big)^2\Big]&\leq 2\frac{\Delta_n^2}{\sigma^2 }x\widetilde{\E}_{t_k,x}^{b}\Big[\big(B_{t_{k+1}}-B_{t_{k}}\big)^2\Big]+2\widetilde{\E}_{t_k,x}^{b}\Big[\big(H_1^{b}+H_2^{b}+H_3^{b}\big)^2\Big]\\
	&\leq 2\frac{\Delta_n^3}{\sigma^2}x+C	\Delta_n^{3+\frac{12}{9+\sqrt{33}}}\Big(x+\dfrac{1}{x^{\frac{\frac{2a}{\sigma^2}-1}{2}+5}}\Big)<+\infty.
	\end{align*}
	Thus, we have shown that $F\delta(u)-\left<DF,u\right>_H$ is square integrable under condition $\frac{a}{\sigma^2}>\frac{15+\sqrt{185}}{4}$. Consequently, by Theorem \ref{ruleproNualart}, under condition $\frac{a}{\sigma^2}>\frac{15+\sqrt{185}}{4}$ we have 
	\begin{align*}
	\delta(Fu)=F\delta(u)-\left<DF,u\right>_H.
	\end{align*}
	This, together with \eqref{Fu2} and \eqref{e0}, gives \eqref{derib} under condition $\frac{a}{\sigma^2}>\frac{15+\sqrt{185}}{4}$. Thus, the result follows.		
\end{proof}

\subsection{Proof of Lemma \ref{estimate}}
\label{Alowerbound}
\begin{proof}
	{\it Proof of \eqref{es3}.} This equality follows from the decomposition \eqref{derib} and the fact that
	\begin{align*}
		\widetilde{\E}_{t_k,x}^{b}\left[\delta\left(\partial_{b}X_{t_{k+1}}^{b}(t_k,x)U^{b}(t_k,x)\right)\right]=\widetilde{\E}_{t_k,x}^{b}\left[B_{t_{k+1}}-B_{t_{k}}\right]=0.
	\end{align*}
	\vskip 5pt
	
	{\it Proof of \eqref{es4}.} Observe that \textcolor{black}{for any $q\geq 1$,}
	\begin{equation}\label{r}
	\begin{split}
	\widetilde{\E}_{t_k,x}^{b}\big[\big\vert H_1^{b}+H_2^{b}+H_3^{b}\big\vert^{2q}\big]
	\leq 3^{2q-1}\big(\widetilde{\E}_{t_k,x}^{b}\big[\big\vert H_1^{b}\big\vert^{2q}\big]+\widetilde{\E}_{t_k,x}^{b}\big[\big\vert H_2^{b}\big\vert^{2q}\big]+\widetilde{\E}_{t_k,x}^{b}\big[\big\vert H_3^{b}\big\vert^{2q}\big]\big).
	\end{split}
	\end{equation}
	First, using BDG's inequality, we have 
	\begin{align*}
	\widetilde{\E}_{t_k,x}^{b}\big[\big\vert H_1^{b}\big\vert^{2q}\big]\leq   C\Delta_n^{2q}x^{2q}\Delta_n^{q-1}\int_{t_k}^{t_{k+1}}H_{11}^{b}ds,
	\end{align*}
	where 
	$$
	H_{11}^{b}=\widetilde{\E}_{t_k,x}^{b}\left[\left\vert\frac{\partial_{x}X_{s}^{b}(t_k,x)}{\sqrt{X_s^{b}(t_k,x)}}-\frac{\partial_{x}X_{t_k}^{b}(t_k,x)}{\sqrt{X_{t_k}^{b}(t_k,x)}}\right\vert^{2q}\right].
	$$
	By It\^o's formula, 
	\begin{align*}
	&\frac{\partial_{x}X_{s}^{b}(t_k,x)}{\sqrt{X_s^{b}(t_k,x)}}-\frac{\partial_{x}X_{t_k}^{b}(t_k,x)}{\sqrt{X_{t_k}^{b}(t_k,x)}}=\int_{t_k}^{s}\partial_{x}X_{u}^{b}(t_k,x)\left(\frac{-\frac{a}{2}+\frac{\sigma^2}{8}}{(X_{u}^{b}(t_k,x))^{\frac{3}{2}}}-\frac{b}{2\sqrt{X_{u}^{b}(t_k,x)}}\right)du\\
	&\qquad+\int_{t_k}^{s}\int_0^{\infty}\partial_{x}X_{u-}^{b}(t_k,x)\left(\frac{1}{\sqrt{X_{u-}^{b}(t_k,x)+z}}-\frac{1}{\sqrt{X_{u-}^{b}(t_k,x)}}\right)M(du,dz),
	\end{align*}
	which, together with BDG's and H\"older's inequalities with $\frac{1}{p_0}+\frac{1}{q_0}=1$, \eqref{e2} and \eqref{e4}, and  {\bf(A2)}, implies that
	\begin{align*}
	H_{11}^{b}&\leq C\Delta_n^{2q-1}\int_{t_k}^{s}\left\{\widetilde{\E}_{t_k,x}^{b}\left[\left\vert\frac{\partial_{x}X_{u}^{b}(t_k,x)}{(X_{u}^{b}(t_k,x))^{\frac{3}{2}}}\right\vert^{2q}\right]+\widetilde{\E}_{t_k,x}^{b}\left[\left\vert\frac{\partial_{x}X_{u}^{b}(t_k,x)}{\sqrt{X_{u}^{b}(t_k,x)}}\right\vert^{2q}\right]\right\}du\\
	&\qquad+C\int_{t_k}^{s}\int_0^{\infty}\widetilde{\E}_{t_k,x}^{b}\left[\left\vert\partial_{x}X_{u}^{b}(t_k,x)\left(\frac{1}{\sqrt{X_{u}^{b}(t_k,x)+z}}-\frac{1}{\sqrt{X_{u}^{b}(t_k,x)}}\right)\right\vert^{2q}\right]m(dz)du\\
	&\leq C\Delta_n^{2q-1}\int_{t_k}^{s}\bigg\{\left(\widetilde{\E}_{t_k,x}^{b}\left[(\partial_{x}X_{u}^{b}(t_k,x))^{2qp_0}\right]\right)^{\frac{1}{p_0}}\left(\widetilde{\E}_{t_k,x}^{b}\left[\frac{1}{(X_{u}^{b}(t_k,x))^{3qq_0}}\right]\right)^{\frac{1}{q_0}}\\
	&\qquad+\left(\widetilde{\E}_{t_k,x}^{b}\left[(\partial_{x}X_{u}^{b}(t_k,x))^{2qp_0}\right]\right)^{\frac{1}{p_0}}\left(\widetilde{\E}_{t_k,x}^{b}\left[\frac{1}{(X_{u}^{b}(t_k,x))^{qq_0}}\right]\right)^{\frac{1}{q_0}}\bigg\}du\\
	&\qquad+C\int_{t_k}^{s}\int_0^{\infty}\widetilde{\E}_{t_k,x}^{b}\left[\left\vert\frac{\partial_{x}X_{u}^{b}(t_k,x)}{(X_{u}^{b}(t_k,x))^{\frac{3}{2}}}\right\vert^{2q}\right]z^{2q}m(dz)du\\
	&\leq C\Delta_n^{2q-1}\int_{t_k}^{s}\left\{\left(1+\dfrac{1}{x^{\frac{\frac{2a}{\sigma^2}-1+2qp_0}{2}}}\right)^{\frac{1}{p_0}}\left(\dfrac{1}{x^{3qq_0}}\right)^{\frac{1}{q_0}}+\left(1+\dfrac{1}{x^{\frac{\frac{2a}{\sigma^2}-1+2qp_0}{2}}}\right)^{\frac{1}{p_0}}\left(\dfrac{1}{x^{qq_0}}\right)^{\frac{1}{q_0}}\right\}du\\
	&\qquad+ C\int_{t_k}^{s}\left(1+\dfrac{1}{x^{\frac{\frac{2a}{\sigma^2}-1+2qp_0}{2}}}\right)^{\frac{1}{p_0}}\left(\dfrac{1}{x^{3qq_0}}\right)^{\frac{1}{q_0}}du\\
	&\leq C\Delta_n\left(1+\dfrac{1}{x^{\frac{\frac{2a}{\sigma^2}-1}{2p_0}+q}}\right)\left(\dfrac{1}{x^{3q}}+\dfrac{1}{x^q}\right),
	\end{align*}
	where $q_0$ should be chosen close to $1$ in order that $3qq_0<\frac{2a}{\sigma^2}-1$. Therefore, under condition $\frac{a}{\sigma^2}>\frac{1}{2}(3q+1)$, we have shown that
	\begin{equation}\label{r1}
	\begin{split}
	\widetilde{\E}_{t_k,x}^{b}\big[\big\vert H_1^{b}\big\vert^{2q}\big]&\leq C\Delta_n^{3q+1}\left(x^q+\frac{1}{x^q}\right)\left(1+\dfrac{1}{x^{\frac{\frac{2a}{\sigma^2}-1}{2p_0}+q}}\right) \\
	&\leq C\Delta_n^{3q+1}\left(x^q+\dfrac{1}{x^{\frac{\frac{2a}{\sigma^2}-1}{2}+2q}}\right).
	\end{split}
	\end{equation}	
	Next, using H\"older's inequality with $\frac{1}{\overline{p}}+\frac{1}{\overline{q}}=1$, we get
	\begin{align*}
	\widetilde{\E}_{t_k,x}^{b}\big[\big\vert H_2^{b}\big\vert^{2q}\big]\leq\left(H_{21}^{b}\right)^{\frac{1}{\overline{p}}}\left(H_{22}^{b}\right)^{\frac{1}{\overline{q}}},
	\end{align*}
	where
	\begin{align*}
	H_{21}^{b}&=\widetilde{\E}_{t_k,x}^{b}\bigg[\bigg\vert \int_{t_k}^{t_{k+1}}\bigg(\frac{X_{s}^{b}(t_k,x)}{\partial_{x}X_{s}^{b}(t_k,x)}-\frac{X_{t_k}^{b}(t_k,x)}{\partial_{x}X_{t_k}^{b}(t_k,x)}\bigg)ds\bigg\vert^{2q\overline{p}}\bigg],\\
	H_{22}^{b}&=\widetilde{\E}_{t_k,x}^{b}\left[\left\vert \int_{t_k}^{t_{k+1}}\frac{\partial_{x}X_{s}^{b}(t_k,x)}{\sigma\sqrt{X_s^{b}(t_k,x)}}dB_s\right\vert^{2q\overline{q}}\right].
	\end{align*}
	First, observe that
	\begin{align*}
	H_{21}^{b}\leq  C\Delta_n^{2q\overline{p}-1}\int_{t_k}^{t_{k+1}}H_{212}^{b}ds,
	\end{align*}
	where
	\begin{align*}	
	H_{212}^{b}=\widetilde{\E}_{t_k,x}^{b}\bigg[\bigg\vert\frac{X_{s}^{b}(t_k,x)}{\partial_{x}X_{s}^{b}(t_k,x)}-\frac{X_{t_k}^{b}(t_k,x)}{\partial_{x}X_{t_k}^{b}(t_k,x)}\bigg\vert^{2q\overline{p}}\bigg].
	\end{align*}
	Then, using \eqref{flowk}, \eqref{px} and It\^o's formula, we get 
	\begin{equation*}\begin{split}
	\frac{X_{s}^{b}(t_k,x)}{\partial_{x}X_{s}^{b}(t_k,x)}-\frac{X_{t_k}^{b}(t_k,x)}{\partial_{x}X_{t_k}^{b}(t_k,x)}&=a\int_{t_k}^s\frac{du}{\partial_{x}X_{u}^{b}(t_k,x)}+\frac{\sigma }{2}\int_{t_k}^s\frac{\sqrt{X_{u}^{b}(t_k,x)}}{\partial_{x}X_{u}^{b}(t_k,x)}dB_u\\
	&\qquad+\int_{t_k}^{s}\int_0^{\infty}\frac{z}{\partial_{x}X_{u-}^{b}(t_k,x)}M(du,dz).
	\end{split}
	\end{equation*}
	Therefore,
	\begin{align*}
	H_{212}^{b}\leq C\left(H_{2121}^{b}+H_{2122}^{b}+H_{2123}^{b}\right),
	\end{align*}
	where
	\begin{align*}
	H_{2121}^{b}&=\widetilde{\E}_{t_k,x}^{b}\bigg[\bigg\vert\int_{t_k}^s\frac{du}{\partial_{x}X_{u}^{b}(t_k,x)}\bigg\vert^{2q\overline{p}}\bigg],\; H_{2122}^{b}=\widetilde{\E}_{t_k,x}^{b}\bigg[\bigg\vert\int_{t_k}^s\frac{\sqrt{X_{u}^{b}(t_k,x)}}{\partial_{x}X_{u}^{b}(t_k,x)}dB_u\bigg\vert^{2q\overline{p}}\bigg],\\
	H_{2123}^{b}&=\widetilde{\E}_{t_k,x}^{b}\bigg[\bigg\vert\int_{t_k}^{s}\int_0^{\infty}\frac{z}{\partial_{x}X_{u-}^{b}(t_k,x)}M(du,dz)\bigg\vert^{2q\overline{p}}\bigg].
	\end{align*}
	Using \eqref{e4}, we obtain
	\begin{align*}
	H_{2121}^{b}&\leq \Delta_n^{2q\overline{p}-1}\int_{t_k}^s\widetilde{\E}_{t_k,x}^{b}\bigg[\frac{1}{(\partial_{x}X_{u}^{b}(t_k,x))^{2q\overline{p}}}\bigg]du\\
	&\leq  C\Delta_n^{2q\overline{p}}\left(1+\dfrac{1}{x^{\frac{\frac{2a}{\sigma^2}-1}{2}-q\overline{p}}}\right).
	\end{align*}	
	Using BDG's and H\"older's inequalities with $\frac{1}{p_1}+\frac{1}{q_1}=1$, \eqref{e1} and \eqref{e4}, we have
	\begin{align*}
	H_{2122}^{b}&\leq C\Delta_n^{q\overline{p}-1}\int_{t_k}^s\widetilde{\E}_{t_k,x}^{b}\bigg[\bigg\vert\frac{\sqrt{X_{u}^{b}(t_k,x)}}{\partial_{x}X_{u}^{b}(t_k,x)}\bigg\vert^{2q\overline{p}}\bigg]du\\
	&\leq C\Delta_n^{q\overline{p}-1}\int_{t_k}^s\left(\widetilde{\E}_{t_k,x}^{b}\bigg[(X_{u}^{b}(t_k,x))^{p_1q\overline{p}}\bigg]\right)^{\frac{1}{p_1}}\left(\widetilde{\E}_{t_k,x}^{b}\bigg[\frac{1}{(\partial_{x}X_{u}^{b}(t_k,x))^{2q_1q\overline{p}}}\bigg]\right)^{\frac{1}{q_1}}du\\
	&\leq C\Delta_n^{q\overline{p}}\left(1+x^{q\overline{p}}\right)\left(1+\dfrac{1}{x^{\frac{\frac{2a}{\sigma^2}-1}{2q_1}-q\overline{p}}}\right),
	\end{align*}
	where $q_1$ should be chosen close to $1$. Finally, using BDG's inequality and {\bf(A2)}, we get
	\begin{align*}
	H_{2123}^{b}&\leq C\int_{t_k}^{s}\int_0^{\infty}\widetilde{\E}_{t_k,x}^{b}\bigg[\frac{1}{(\partial_{x}X_{u}^{b}(t_k,x))^{2q\overline{p}}}\bigg]z^{2q\overline{p}}m(dz)du\\
	&\leq  C\Delta_n\left(1+\dfrac{1}{x^{\frac{\frac{2a}{\sigma^2}-1}{2}-q\overline{p}}}\right).
	\end{align*}	
	Thus, we have shown that
	\begin{align*}
	H_{212}^{b}&\leq C\Delta_n^{2q\overline{p}}\left(1+\dfrac{1}{x^{\frac{\frac{2a}{\sigma^2}-1}{2}-q\overline{p}}}\right)+ C\Delta_n^{q\overline{p}}\left(1+x^{q\overline{p}}\right)\left(1+\dfrac{1}{x^{\frac{\frac{2a}{\sigma^2}-1}{2q_1}-q\overline{p}}}\right)+C\Delta_n\left(1+\dfrac{1}{x^{\frac{\frac{2a}{\sigma^2}-1}{2}-q\overline{p}}}\right)\\
	&\leq C\Delta_n\left(1+x^{q\overline{p}}\right)\left(1+\dfrac{1}{x^{\frac{\frac{2a}{\sigma^2}-1}{2}-q\overline{p}}}\right),
	\end{align*}
	which implies that
	\begin{align*}
	H_{21}^{b}&\leq C\Delta_n^{2q\overline{p}}\Delta_n\left(1+x^{q\overline{p}}\right)\left(1+\dfrac{1}{x^{\frac{\frac{2a}{\sigma^2}-1}{2}-q\overline{p}}}\right).
	\end{align*}
	Next, using BDG's and H\"older's inequalities with $\frac{1}{p_2}+\frac{1}{q_2}=1$,  \eqref{e2} and \eqref{e4}, we have
	\begin{align*}
	H_{22}^{b}&\leq C\widetilde{\E}_{t_k,x}^{b}\left[\left\vert \int_{t_k}^{t_{k+1}}\frac{(\partial_{x}X_{s}^{b}(t_k,x))^2}{X_s^{b}(t_k,x)}ds\right\vert^{q\overline{q}}\right]\\
	&\leq C\Delta_n^{q\overline{q}-1}\int_{t_k}^{t_{k+1}}\widetilde{\E}_{t_k,x}^{b}\left[\left\vert \frac{\partial_{x}X_{s}^{b}(t_k,x)}{\sqrt{X_s^{b}(t_k,x)}}\right\vert^{2q\overline{q}}\right]ds\\
	&\leq C\Delta_n^{q\overline{q}-1}\int_{t_k}^{t_{k+1}}\left(\widetilde{\E}_{t_k,x}^{b}\left[\frac{1}{(X_s^{b}(t_k,x))^{p_2q\overline{q}}}\right]\right)^{\frac{1}{p_2}}\left(\widetilde{\E}_{t_k,x}^{b}\left[(\partial_{x}X_{s}^{b}(t_k,x))^{2q_2q\overline{q}}\right]\right)^{\frac{1}{q_2}}ds\\
	&\leq C\Delta_n^{q\overline{q}}\dfrac{1}{x^{q\overline{q}}}\left(1+\dfrac{1}{x^{\frac{\frac{2a}{\sigma^2}-1}{2q_2}+q\overline{q}}}\right),
	\end{align*}
	where $p_2$ should be chosen close to $1$ in order that $p_2q\overline{q}<\frac{2a}{\sigma^2}-1$. 
	
	To be able to apply \eqref{e2} and \eqref{e4} to estimate two terms above $H_{21}^{b}$ and $H_{22}^{b}$, all conditions required here are the following
	\begin{align*}
	-2q\overline{p}> -\frac{(\frac{2a}{\sigma^2}-1)^2}{2(\frac{2a}{\sigma^2}-\frac{1}{2})},\;
	q\overline{q}<\dfrac{2a}{\sigma^2}-1.
	\end{align*}
	This implies that
	\begin{align*}
	\begin{cases}
	\frac{2a}{\sigma^2}> 2q\overline{p}+\sqrt{2q\overline{p}\left(2q\overline{p}+1\right)}+1\\
	\frac{2a}{\sigma^2}>q\frac{\overline{p}}{\overline{p}-1}+1.
	\end{cases}
	\end{align*}	
	Here, the optimal choice for $\overline{p}$ corresponds to choose it in a way which gives minimal restrictions on the ratio $\frac{2a}{\sigma^2}$. That is,
	$$
	2q\overline{p}+\sqrt{2q\overline{p}\left(2q\overline{p}+1\right)}=q\frac{\overline{p}}{\overline{p}-1}.
	$$	
	Thus, the unique solution is  $\overline{p}=\frac{5q+4+\sqrt{25q^2+8q}}{4(2q+1)}$, which implies  $\frac{a}{\sigma^2}>\frac{5q+2+\sqrt{25q^2+8q}}{4}$. Therefore, under condition $\frac{a}{\sigma^2}>\frac{5q+2+\sqrt{25q^2+8q}}{4}$, we have shown that
	\begin{align}\label{r2}
	\widetilde{\E}_{t_k,x}^{b}\big[\big\vert H_2^{b}\big\vert^{2q}\big]&\leq C\bigg(\Delta_n^{2q\overline{p}}\Delta_n\left(1+x^{q\overline{p}}\right)\left(1+\dfrac{1}{x^{\frac{\frac{2a}{\sigma^2}-1}{2}-q\overline{p}}}\right)\bigg)^{\frac{1}{\overline{p}}}\left(\Delta_n^{q\overline{q}}\dfrac{1}{x^{q\overline{q}}}\left(1+\dfrac{1}{x^{\frac{\frac{2a}{\sigma^2}-1}{2q_2}+q\overline{q}}}\right)\right)^{\frac{1}{\overline{q}}}\notag\\
	&\leq C\Delta_n^{3q+\frac{1}{\overline{p}}}\left(1+\dfrac{1}{x^{\frac{\frac{2a}{\sigma^2}-1}{2}+2q}}\right),
	\end{align}
	where $\overline{p}=\frac{5q+4+\sqrt{25q^2+8q}}{4(2q+1)}$.
	
	Finally, we treat the term $H_3^{b}$. Using H\"older's inequality with $\frac{1}{p_3}+\frac{1}{q_3}=1$, we obtain
	\begin{align*}
		&\widetilde{\E}_{t_k,x}^{b}[\vert H_3^{b}\vert^{2q}]\leq \Delta_n^{2q-1}\int_{t_k}^{t_{k+1}}\widetilde{\E}_{t_k,x}^{b}\Big[\Big\vert \int_{s}^{t_{k+1}}D_s\Big(\dfrac{X_{r}^{b}(t_k,x)}{\partial_xX_{r}^{b}(t_k,x)}\Big)dr\frac{\partial_{x}X_{s}^{b}(t_k,x)}{\sigma\sqrt{X_s^{b}(t_k,x)}}\Big\vert^{2q}\Big]ds\\
		&\leq C\Delta_n^{2q-1}\int_{t_k}^{t_{k+1}}\Big(\widetilde{\E}_{t_k,x}^{b}\Big[\Big\vert \int_{s}^{t_{k+1}}D_s\Big(\dfrac{X_{r}^{b}(t_k,x)}{\partial_xX_{r}^{b}(t_k,x)}\Big)dr\Big\vert^{2qp_3}\Big]\Big)^{\frac{1}{p_3}}\Big(\widetilde{\E}_{t_k,x}^{b}\Big[\Big\vert \frac{\partial_{x}X_{s}^{b}(t_k,x)}{\sqrt{X_s^{b}(t_k,x)}}\Big\vert^{2qq_3}\Big]\Big)^{\frac{1}{q_3}}ds\\
		&\leq C\Delta_n^{2q-1}\int_{t_k}^{t_{k+1}}\Big(\Delta_n^{2qp_3-1} \int_{s}^{t_{k+1}}\widetilde{\E}_{t_k,x}^{b}\Big[\Big\vert D_s\Big(\dfrac{X_{r}^{b}(t_k,x)}{\partial_xX_{r}^{b}(t_k,x)}\Big)\Big\vert^{2qp_3}\Big]dr\Big)^{\frac{1}{p_3}}\\
		&\qquad\times\Big(\widetilde{\E}_{t_k,x}^{b}\Big[\Big\vert \frac{\partial_{x}X_{s}^{b}(t_k,x)}{\sqrt{X_s^{b}(t_k,x)}}\Big\vert^{2qq_3}\Big]\Big)^{\frac{1}{q_3}}ds.
	\end{align*}
	Using \eqref{Mallinveflow2} and the same computations as in the proof of \eqref{dmif3} and \eqref{dmif2} with $p=2qp_3$, we get
	\begin{align*}
		&\widetilde{\E}_{t_k,x}^{b}\Big[\Big\vert D_s\Big(\dfrac{X_{r}^{b}(t_k,x)}{\partial_xX_{r}^{b}(t_k,x)}\Big)\Big\vert^{2qp_3}\Big]\leq C\Big(\widetilde{\E}_{t_k,x}^{b}\Big[\Big\vert \dfrac{D_sX_{r}^{b}(t_k,x)}{\partial_xX_{r}^{b}(t_k,x)}\Big\vert^{2qp_3}\Big]\\
		&\qquad+\widetilde{\E}_{t_k,x}^{b}\Big[\Big\vert X_{r}^{b}(t_k,x)D_s\Big(\dfrac{1}{\partial_xX_{r}^{b}(t_k,x)}\Big)\Big\vert^{2qp_3}\Big]\Big)\\
		&\leq C(1+x^{qp_3})\Big(1+\dfrac{1}{x^{\frac{\frac{2a}{\sigma^2}-1}{2p_7}-2qp_3}}\Big)  \Big(1+\dfrac{1}{x^{\frac{\frac{2a}{\sigma^2}-1}{2q_7q_8}+qp_3}}\Big)\\
		&\qquad+\frac{C}{x^{4qp_3}}(1+x^{6qp_3})\Big(1+\dfrac{1}{x^{\frac{\frac{2a}{\sigma^2}-1}{p_4p_5}-2qp_3}}\Big)\Big(1+\dfrac{1}{x^{\frac{\frac{2a}{\sigma^2}-1}{2q_4q_6}+qp_3}}\Big),
		\end{align*}
	where $\frac{1}{p_5}+\frac{1}{p_5}+\frac{1}{q_5}=1$ which is given in the proof of \eqref{dmif2},  $\frac{1}{p_7}+\frac{1}{q_7}=1$ with $p_7>1$ and $p_7$ close to $1$, $\frac{1}{p_4}+\frac{1}{q_4}=1$ with $p_4>1$ and $p_4$  close to $1$, $q_6>1$ and $q_8>1$.	
	
	Next, using H\"older's inequality with $\frac{1}{p_9}+\frac{1}{q_9}=1$, \eqref{e2} and \eqref{e4},
	\begin{align*}
		\widetilde{\E}_{t_k,x}^{b}\Big[\Big\vert \frac{\partial_{x}X_{s}^{b}(t_k,x)}{\sqrt{X_s^{b}(t_k,x)}}\Big\vert^{2qq_3}\Big]&\leq \Big(\widetilde{\E}_{t_k,x}^{b}\Big[(\partial_{x}X_{s}^{b}(t_k,x))^{2qq_3p_9}\Big]\Big)^{\frac{1}{p_9}}\Big(\widetilde{\E}_{t_k,x}^{b}\Big[\frac{1}{(X_s^{b}(t_k,x))^{qq_3q_9}}\Big]\Big)^{\frac{1}{q_9}}\\
		&\leq C\Big(1+\dfrac{1}{x^{\frac{\frac{2a}{\sigma^2}-1}{2p_9}+qq_3}}\Big)\frac{1}{x^{qq_3}}.
	\end{align*}
	Here, $q_9$ should be chosen close to $1$ in order that $qq_3q_9<\frac{2a}{\sigma^2}-1$. In order to apply \eqref{e2} and \eqref{e4} to estimate the term $R_{3}^{a,b}$, all conditions required here are as follows
	\begin{align*}
		4qp_3q_5<\dfrac{2a}{\sigma^2}-1,\; 2qp_3p_5<\frac{(\frac{2a}{\sigma^2}-1)^2}{2(\frac{2a}{\sigma^2}-\frac{1}{2})},\; qq_3<\dfrac{2a}{\sigma^2}-1.
	\end{align*}
	This implies that
	\begin{align*}
		\frac{2a}{\sigma^2}>\frac{4qp_3p_5}{p_5-2}+1,\; \frac{2a}{\sigma^2}> 2qp_3p_5+\sqrt{2qp_3p_5\left(2qp_3p_5+1\right)}+1,\;  \frac{2a}{\sigma^2}>\frac{qp_3}{p_3-1}+1.
	\end{align*}	
	Here, the optimal choice for $p_3$ and $p_5$ corresponds to choose them in a way which gives minimal restrictions on the ratio $\frac{2a}{\sigma^2}$. That is,
	$$
	2qp_3p_5+\sqrt{2qp_3p_5\left(2qp_3p_5+1\right)}=\frac{4qp_3p_5}{p_5-2}=\frac{qp_3}{p_3-1}.
	$$	
	Thus, the unique solution is $p_3=\frac{65q+8+5\sqrt{169q^2+16q}}{4(17q+2+\sqrt{169q^2+16q})}$ and $p_5=\frac{17q+2+\sqrt{169q^2+16q}}{10q+1}$, which implies that
	$\frac{a}{\sigma^2}>\frac{13q+2+\sqrt{169q^2+16q}}{4}$. Therefore, under $\frac{a}{\sigma^2}>\frac{13q+2+\sqrt{169q^2+16q}}{4}$, we obtain
	\begin{align}\label{r3}
		\widetilde{\E}_{t_k,x}^{b}[\vert H_3^{b}\vert^{2q}]&\leq  C\Delta_n^{4q}\Big\{(1+x^{q})\Big(1+\dfrac{1}{x^{\frac{\frac{2a}{\sigma^2}-1}{2p_7p_3}-2q}}\Big)\Big(1+\dfrac{1}{x^{\frac{\frac{2a}{\sigma^2}-1}{2q_7q_8p_3}+q}}\Big)\notag\\
		&\qquad +\frac{1}{x^{4q}}(1+x^{6q})\Big(1+\dfrac{1}{x^{\frac{\frac{2a}{\sigma^2}-1}{p_4p_5p_3}-2q}}\Big)\Big(1+\dfrac{1}{x^{\frac{\frac{2a}{\sigma^2}-1}{2q_4q_6p_3}+q}}\Big)\Big\}\Big(1+\dfrac{1}{x^{\frac{\frac{2a}{\sigma^2}-1}{2p_9q_3}+q}}\Big)\frac{1}{x^q}\notag\\
		&\leq C\Delta_n^{4q}\Big(x^{q}+\dfrac{1}{x^{\frac{\frac{2a}{\sigma^2}-1}{2}+5q}}\Big),
	\end{align}
	where $\frac{1}{p_7}+\frac{1}{q_7}=1$ with $p_7>1$ and $p_7$ close to $1$, $\frac{1}{p_4}+\frac{1}{q_4}=1$ with $p_4>1$ and $p_4$  close to $1$, $q_6>1$, $q_8>1$,  $p_3=\frac{65q+8+5\sqrt{169q^2+16q}}{4(17q+2+\sqrt{169q^2+16q})}$ and $p_5=\frac{17q+2+\sqrt{169q^2+16q}}{10q+1}$.
	
	From \eqref{r}, \eqref{r1}, \eqref{r2} and \eqref{r3}, under condition $\frac{a}{\sigma^2}>\frac{13q+2+\sqrt{169q^2+16q}}{4}$, we obtain 
	\begin{align*}
		\widetilde{\E}_{t_k,x}^{b}\big[\big\vert H_1^{b}+H_2^{b}+H_3^{b}\big\vert^{2q}\big]\leq C\Delta_n^{3q+\frac{1}{\overline{p}}}\Big(x^{q}+\dfrac{1}{x^{\frac{\frac{2a}{\sigma^2}-1}{2}+5q}}\Big),
	\end{align*}
	where $\overline{p}=\frac{5q+4+\sqrt{25q^2+8q}}{4(2q+1)}$. 	Thus, we conclude the desired estimate \eqref{es4}.	
\end{proof}
\begin{remark}\label{minimal}
	When we use Cauchy-Schwarz's inequality instead of H\"older's inequality to estimate $\big\vert H_1^{b}+H_2^{b}+H_3^{b}\big\vert^{2q}$, the required condition will be $\frac{a}{\sigma^2}>\frac{1}{2}(16q+1+4\sqrt{q(16q+1)})$ which is actually bigger than $\frac{13q+2+\sqrt{169q^2+16q}}{4}$.
\end{remark}

\begin{remark}\label{A3diffusionjumps} In the case of CIR process without jumps studied in \cite{BKT17} when the subordinator is degenerate, using the explicit expression for the Malliavin derivative obtained by Al\`os and Ewald in \cite[Corollary 4.2]{AE08} instead of the expression \eqref{expression1} in order to estimate the term $H_3^{b}$, condition {\bf(A3)} will be $\frac{a}{\sigma^2}>\frac{11+\sqrt{89}}{4}$.
\end{remark}

\subsection{Proof of Lemma \ref{change}}
\begin{proof} We proceed as in the proof of \cite[Lemma 9]{BKT17}.
\end{proof}

\subsection{Proof of Lemma \ref{deviation1}}
\begin{proof} 
	Using \eqref{ratio2} for $X$, we have that
	\begin{align}\label{decom}
	&\dfrac{d\widetilde{\P}_{t_k,x}^{b_0}}{d\widetilde{\P}_{t_k,x}^{b}}((X_t^{b})_{t\in I_k})-1=\dfrac{d\widetilde{\P}_{t_k,x}^{b_0}-d\widetilde{\P}_{t_k,x}^{b}}{d\widetilde{\P}_{t_k,x}^{b}}((X_t^{b})_{t\in I_k})\notag\\
	&=\int_{b}^{b_0}\dfrac{\partial}{\partial \beta}\left(\dfrac{d\widetilde{\P}_{t_k,x}^{\beta}}{d\widetilde{\P}_{t_k,x}^{b}}\right)((X_t^{b})_{t\in I_k})d\beta \notag\\
	&=\dfrac{-1}{\sigma^2}\int_{b}^{b_0}\int_{t_k}^{t_{k+1}}\Big(\sigma\sqrt{X_s^{b}}dB_s+(\beta-b)X_s^{b}ds\Big)\dfrac{d\widetilde{\P}_{t_k,x}^{\beta}}{d\widetilde{\P}_{t_k,x}^{b}}((X_t^{b})_{t\in I_k})d\beta.
	\end{align}
	Applying H\"older's inequality with $\frac{1}{p}+\frac{1}{q}+\frac{1}{r}=1$, BDG's and Jensen's inequalities, \eqref{e1} for $X^{b}$, we get 
	\begin{align}
	&\left\vert\widetilde{\E}_{t_k,x}^{b}\left[\widetilde{\E}_{t_k,x}^{b}\big[V\vert X_{t_{k+1}}^{b}\big]\Big(\frac{d\widetilde{\P}_{t_k,x}^{b_0}}{d\widetilde{\P}_{t_k,x}^{b}}((X_t^{b})_{t\in I_k})-1\Big)\right]\right\vert \notag\\
	&=\dfrac{1}{\sigma^2}\left\vert\int_{b}^{b_0}\widetilde{\E}_{t_k,x}^{b}\left[\widetilde{\E}_{t_k,x}^{b}\big[V\vert X_{t_{k+1}}^{b}\big]\int_{t_k}^{t_{k+1}}\Big(\sigma\sqrt{X_s^{b}}dB_s+(\beta-b)X_s^{b}ds\Big)\dfrac{d\widetilde{\P}_{t_k,x}^{\beta}}{d\widetilde{\P}_{t_k,x}^{b}}((X_t^{b})_{t\in I_k})\right]d\beta\right\vert \notag\\
	&\leq \dfrac{1}{\sigma^2}\bigg\vert\int_{b}^{b_0} \Big(\widetilde{\E}_{t_k,x}^{b}\big[\widetilde{\E}_{t_k,x}^{b}\big[\vert V\vert^{q}\vert X_{t_{k+1}}^{b}\big]\big]\Big)^{\frac{1}{q}}\Big(\widetilde{\E}_{t_k,x}^{b}\Big[\Big\vert \int_{t_k}^{t_{k+1}}\Big(\sigma\sqrt{X_s^{b}}dB_s+(\beta-b)X_s^{b})ds\Big)\Big\vert^{p}\Big]\Big)^{\frac{1}{p}} \notag\\
	&\qquad \times \Big(\widetilde{\E}_{t_k,x}^{b}\Big[\Big( \frac{d\widetilde{\P}_{t_k,x}^{\beta}}{d\widetilde{\P}_{t_k,x}^{b}}((X_t^{b})_{t\in I_k})\Big)^{r}\Big]\Big)^{\frac{1}{r}} d\beta\bigg\vert \notag\\
	&\leq C\sqrt{\Delta_n}\bigg\vert\int_{b}^{b_0} \big(\widetilde{\E}_{t_k,x}^{b}\big[\vert V\vert^{q}\big]\big)^{\frac{1}{q}} \left(1+\sqrt{x}+\sqrt{\Delta_n}\vert b-b_0\vert x\right)\notag\\
	&\qquad\times\Big(\widetilde{\E}_{t_k,x}^{b}\Big[\Big( \frac{d\widetilde{\P}_{t_k,x}^{\beta}}{d\widetilde{\P}_{t_k,x}^{b}}((X_t^{b})_{t\in I_k})\Big)^{r}\Big]\Big)^{\frac{1}{r}} d\beta\bigg\vert, \label{a1}
	\end{align}
for some constant $C>0$. Then, using Cauchy-Schwarz and Jensen's inequalities and \eqref{Laplace2} applied to $X^b$, provided that we have an exponential martingale, we get for $n$ large enough,
	\begin{align}
	&\widetilde{\E}_{t_k,x}^{b}\Big[\Big( \frac{d\widetilde{\P}_{t_k,x}^{\beta}}{d\widetilde{\P}_{t_k,x}^{b}}((X_t^{b})_{t\in I_k})\Big)^{r}\Big]\leq \left(\widetilde{\E}_{t_k,x}^{b}\Big[\exp\Big\{r(2r-1)\frac{1}{\sigma^2}(b_0-b)^2\int_{t_k}^{t_{k+1}}X_s^{b}ds\Big\}\Big]\right)^{\frac{1}{2}} \notag\\
	&\leq\left(\frac{1}{\Delta_n}\int_{t_k}^{t_{k+1}}\widetilde{\E}_{t_k,x}^{b}\Big[\exp\Big\{r(2r-1)\frac{1}{\sigma^2}(b_0-b)^2\Delta_nX_s^{b}\Big\}\Big]ds\right)^{\frac{1}{2}} \notag\\
	&\leq C\left(\sup_{s\in [t_k,t_{k+1}]}\widetilde{\E}_{t_k,x}^{b}\Big[\exp\Big\{r(2r-1)\frac{1}{\sigma^2}(b_0-b)^2\Delta_nX_s^{b}\Big\}\Big]\right)^{\frac{1}{2}} \notag\\
	&=C\left(\sup_{s\in [0,\Delta_n]}\widetilde{\E}_{0,x}^{b}\Big[\exp\Big\{r(2r-1)\frac{1}{\sigma^2}(b_0-b)^2\Delta_nX_s^{b}\Big\}\Big]\right)^{\frac{1}{2}} \notag\\
	&\leq Ce^{c(b_0-b)^2\Delta_nx},\label{a2}
	\end{align}
	for some constants $C, c>0$. 
	
	The same arguments can be used to check Novikov's condition and deduce the validity of the exponential martingale property used above. Thus, \eqref{for3} follows from \eqref{a1} and \eqref{a2}.		
	
	Finally, using \eqref{decom} for $Y$, H\"older's inequality with $\frac{1}{p}+\frac{1}{q}+\frac{1}{r}=1$, BDG's and Jensen's inequalities, \eqref{e1} for $Y^{b}$, and \eqref{a2} for $\widehat{\P}$,  we get for $n$ large enough,
	\begin{align*}
		&\left\vert\widehat{\E}_{t_k,x}^{b}\left[\widehat{V}\Big(\frac{d\widehat{\P}_{t_k,x}^{b_0}}{d\widehat{\P}_{t_k,x}^{b}}((Y_t^{b})_{t\in I_k})-1\Big)\right]\right\vert \\
		&=\dfrac{1}{\sigma^2}\left\vert\int_{b}^{b_0}\widehat{\E}_{t_k,x}^{b}\left[\widehat{V}\int_{t_k}^{t_{k+1}}\Big(\sigma\sqrt{Y_s^{b}}dW_s+(\beta-b)Y_s^{b}ds\Big)\dfrac{d\widehat{\P}_{t_k,x}^{\beta}}{d\widehat{\P}_{t_k,x}^{b}}((Y_t^{b})_{t\in I_k})\right]d\beta\right\vert \\
		&\leq \dfrac{1}{\sigma^2}\bigg\vert\int_{b}^{b_0} \Big(\widehat{\E}_{t_k,x}^{b}\big[\vert \widehat{V}\vert^{q}\big]\Big)^{\frac{1}{q}}\Big(\widehat{\E}_{t_k,x}^{b}\Big[\Big\vert \int_{t_k}^{t_{k+1}}\Big(\sigma\sqrt{Y_s^{b}}dW_s+(\beta-b)Y_s^{b})ds\Big)\Big\vert^{p}\Big]\Big)^{\frac{1}{p}} \\
		&\qquad \times \Big(\widehat{\E}_{t_k,x}^{b}\Big[\Big( \frac{d\widehat{\P}_{t_k,x}^{\beta}}{d\widehat{\P}_{t_k,x}^{b}}((Y_t^{b})_{t\in I_k})\Big)^{r}\Big]\Big)^{\frac{1}{r}} d\beta\bigg\vert \\
		&\leq C\sqrt{\Delta_n}\bigg\vert\int_{b}^{b_0} \big(\widehat{\E}_{t_k,x}^{b}\big[\vert \widehat{V}\vert^{q}\big]\big)^{\frac{1}{q}}  \left(1+\sqrt{x}+\sqrt{\Delta_n}\vert b-b_0\vert x\right)\\
		&\qquad\times\Big(\widehat{\E}_{t_k,x}^{b}\Big[\Big( \frac{d\widehat{\P}_{t_k,x}^{\beta}}{d\widehat{\P}_{t_k,x}^{b}}((Y_t^{b})_{t\in I_k})\Big)^{r}\Big]\Big)^{\frac{1}{r}} d\beta\bigg\vert\\
		&\leq C_1\sqrt{\Delta_n}\big(\widetilde{\E}_{t_k,x}^{b}[\vert V\vert^q]\big)^{\frac{1}{q}}e^{C_2(b_0-b)^2\Delta_nx} \vert b-b_0\vert\left(1+\sqrt{x}+\sqrt{\Delta_n}\vert b-b_0\vert x\right).
	\end{align*}
	Thus, \eqref{for4} follows.	This completes the proof.
\end{proof}

\subsection{Proof of Lemma \ref{jumpestimate2}}
\label{large}
\begin{proof} Splitting the jump amplitudes into small jumps and big jumps, we get 
	\begin{align} 
	&\widehat{\E}_{t_k,Y_{t_k}^{b_0}}^{b_0}\left[\Big(\int_{t_k}^{t_{k+1}}\int_{0}^{\infty}zN(ds,dz)-\widetilde{\E}_{t_k,Y_{t_{k}}^{b_0}}^{b}\Big[\int_{t_k}^{t_{k+1}}\int_{0}^{\infty}zM(ds,dz)\big\vert X_{t_{k+1}}^{b}=Y_{t_{k+1}}^{b_0}\Big]\Big)^2\right]\notag\\
	&=\widehat{\E}_{t_k,Y_{t_k}^{b_0}}^{b_0}\Big[\Big(\int_{t_k}^{t_{k+1}}\int_{z\leq\upsilon_n}z\widetilde{N}(ds,dz)+\Delta_n\int_{z\leq\upsilon_n}zm(dz)+\int_{t_k}^{t_{k+1}}\int_{z>\upsilon_n}zN(ds,dz)\notag\\
	&\qquad-\widetilde{\E}_{t_k,Y_{t_{k}}^{b_0}}^{b}\Big[\int_{t_k}^{t_{k+1}}\int_{z\leq\upsilon_n}z\widetilde{M}(ds,dz)+\Delta_n\int_{z\leq\upsilon_n}zm(dz)\notag\\
	&\qquad+\int_{t_k}^{t_{k+1}}\int_{z>\upsilon_n}zM(ds,dz)\big\vert X_{t_{k+1}}^{b}=Y_{t_{k+1}}^{b_0}\Big]\Big)^2\Big] \notag\\
	&\leq 3\left(D_{1,k,n}+D_{2,k,n}+D_{3,k,n}\right),\label{D}
	\end{align}
	where 
	\begin{align*} 
	D_{1,k,n}&=\widehat{\E}_{t_k,Y_{t_k}^{b_0}}^{b_0}\Big[\Big(\int_{t_k}^{t_{k+1}}\int_{z\leq\upsilon_n}z\widetilde{N}(ds,dz)\Big)^2\Big],\\
	D_{2,k,n}&=\widehat{\E}_{t_k,Y_{t_k}^{b_0}}^{b_0}\Big[\Big(\widetilde{\E}_{t_k,Y_{t_{k}}^{b_0}}^{b}\Big[\int_{t_k}^{t_{k+1}}\int_{z\leq\upsilon_n}z\widetilde{M}(ds,dz)\big\vert X_{t_{k+1}}^{b}=Y_{t_{k+1}}^{b_0}\Big]\Big)^2\Big],\\
	D_{3,k,n}&=\widehat{\E}_{t_k,Y_{t_k}^{b_0}}^{b_0}\Big[\Big(\int_{t_k}^{t_{k+1}}\int_{z>\upsilon_n}zN(ds,dz)\\
	&\qquad-\widetilde{\E}_{t_k,Y_{t_{k}}^{b_0}}^{b}\Big[\int_{t_k}^{t_{k+1}}\int_{z>\upsilon_n}zM(ds,dz)\big\vert X_{t_{k+1}}^{b}=Y_{t_{k+1}}^{b_0}\Big]\Big)^2\Big].
	\end{align*}
	First, using BDG's inequality, we get
	\begin{align} \label{D1}
	D_{1,k,n}\leq C\int_{t_k}^{t_{k+1}}\int_{z\leq\upsilon_n}z^2m(dz)ds= C\Delta_n\int_{z\leq\upsilon_n}z^2m(dz).
	\end{align}
	Next, using Jensen's inequality, \eqref{for1} of Lemma \ref{change}, \eqref{for3} of Lemma \ref{deviation1} with $q=q_1\in (1,2]$ and BDG's inequality, we have
	\begin{align}\label{D2}
	D_{2,k,n}&\leq\widehat{\E}_{t_k,Y_{t_k}^{b_0}}^{b_0}\Big[\widetilde{\E}_{t_k,Y_{t_{k}}^{b_0}}^{b}\Big[\Big(\int_{t_k}^{t_{k+1}}\int_{z\leq\upsilon_n}z\widetilde{M}(ds,dz)\Big)^2\big\vert X_{t_{k+1}}^{b}=Y_{t_{k+1}}^{b_0}\Big]\Big] \notag\\
	&\leq\widetilde{\E}_{t_k,Y_{t_{k}}^{b_0}}^{b}\Big[\Big(\int_{t_k}^{t_{k+1}}\int_{z\leq\upsilon_n}z\widetilde{M}(ds,dz)\Big)^2\Big]+C_1\vert u\vert\sqrt{\Delta_n}\varphi_{n\Delta_n}(b_0)e^{C_2u^2(\varphi_{n\Delta_n}(b_0))^2\Delta_nY_{t_k}^{b_0}}\notag\\
	&\qquad\times\Big(\widetilde{\E}_{t_k,Y_{t_{k}}^{b_0}}^{b}\Big[\Big\vert\int_{t_k}^{t_{k+1}}\int_{z\leq\upsilon_n}z\widetilde{M}(ds,dz)\Big\vert^{2q_1}\Big]\Big)^{\frac{1}{q_1}}\big(1+\sqrt{Y_{t_k}^{b_0}}+\sqrt{\Delta_n}\vert u\vert\varphi_{n\Delta_n}(b_0) Y_{t_k}^{b_0}\big)\notag\\
	&\leq C\Delta_n\int_{z\leq\upsilon_n}z^2m(dz)+C_1\varphi_{n\Delta_n}(b_0)\Delta_n^{\frac{1}{2}+\frac{1}{q_1}}\Big(\int_{z\leq\upsilon_n}z^{2q_1}m(dz)\Big)^{\frac{1}{q_1}}\notag\\ &\qquad\times e^{C_2u^2(\varphi_{n\Delta_n}(b_0))^2\Delta_nY_{t_k}^{b_0}}\big(1+\sqrt{Y_{t_k}^{b_0}}+\sqrt{\Delta_n}\vert u\vert\varphi_{n\Delta_n}(b_0) Y_{t_k}^{b_0}\big).
	\end{align}
	To treat $D_{3,k,n}$, for $k \in \{0,...,n-1\}$, we consider the events $\widehat{N}_{0,k}(\upsilon_n):=\{N_{t_{k+1}}^{\upsilon_n}-N_{t_{k}}^{\upsilon_n}=0\}$ which have no big jumps of $J^{\upsilon_n}$ in the interval $[t_k,t_{k+1})$ and $\widehat{N}_{\geq1,k}(\upsilon_n):=\{N_{t_{k+1}}^{\upsilon_n}-N_{t_{k}}^{\upsilon_n}\geq1\}$ which have one or
	more than one big jump of $J^{\upsilon_n}$ in the interval $[t_k,t_{k+1})$. Similarly, we consider the events $\widetilde{N}_{0,k}(\upsilon_n):=\{M_{t_{k+1}}^{\upsilon_n}-M_{t_{k}}^{\upsilon_n}=0\}$ which have no big jumps of $\widetilde{J}^{\upsilon_n}$ in $[t_k,t_{k+1})$ and $\widetilde{N}_{\geq1,k}(\upsilon_n):=\{M_{t_{k+1}}^{\upsilon_n}-M_{t_{k}}^{\upsilon_n}\geq1\}$ which have one or
	more than one big jump of $\widetilde{J}^{\upsilon_n}$ in $[t_k,t_{k+1})$. 
	
	Then, multiplying the random variable outside the conditional expectation of $D_{3,k,n}$ by ${\bf 1}_{\widehat{N}_{0,k}(\upsilon_n)} +{\bf 1}_{\widehat{N}_{\geq1,k}(\upsilon_n)}$, we get 
	\begin{equation}\label{D3}
	\begin{split}
	&D_{3,k,n}=\widehat{\E}_{t_k,Y_{t_k}^{b_0}}^{b_0}\Big[\big({\bf 1}_{\widehat{N}_{0,k}(\upsilon_n)}+{\bf 1}_{\widehat{N}_{\geq1,k}(\upsilon_n)}\big)\Big(\int_{t_k}^{t_{k+1}}\int_{z>\upsilon_n}zN(ds,dz)\\
	&\qquad-\widetilde{\E}_{t_k,Y_{t_{k}}^{b_0}}^{b}\Big[\int_{t_k}^{t_{k+1}}\int_{z>\upsilon_n}zM(ds,dz)\big\vert X_{t_{k+1}}^{b}=Y_{t_{k+1}}^{b_0}\Big]\Big)^2\Big]=M_{0,k,n}^b+M_{\geq1,k,n}^b,
	\end{split}
	\end{equation}
	where
	\begin{align*}
	M_{0,k,n}^b&=\widehat{\E}_{t_k,Y_{t_k}^{b_0}}^{b_0}\Big[{\bf 1}_{\widehat{N}_{0,k}(\upsilon_n)}\Big(\int_{t_k}^{t_{k+1}}\int_{z>\upsilon_n}zN(ds,dz)\\
	&\qquad-\widetilde{\E}_{t_k,Y_{t_{k}}^{b_0}}^{b}\Big[\int_{t_k}^{t_{k+1}}\int_{z>\upsilon_n}zM(ds,dz)\big\vert X_{t_{k+1}}^{b}=Y_{t_{k+1}}^{b_0}\Big]\Big)^2\Big],\\
	M_{\geq1,k,n}^b&=\widehat{\E}_{t_k,Y_{t_k}^{b_0}}^{b_0}\Big[{\bf 1}_{\widehat{N}_{\geq1,k}(\upsilon_n)}\Big(\int_{t_k}^{t_{k+1}}\int_{z>\upsilon_n}zN(ds,dz)\\
	&\qquad-\widetilde{\E}_{t_k,Y_{t_{k}}^{b_0}}^{b}\Big[\int_{t_k}^{t_{k+1}}\int_{z>\upsilon_n}zM(ds,dz)\big\vert X_{t_{k+1}}^{b}=Y_{t_{k+1}}^{b_0}\Big]\Big)^2\Big].
	\end{align*}	
	We start treating $M_{0,k,n}^b$. For this, multiplying the random variable inside the conditional expectation of $M_{0,k,n}^b$ by ${\bf 1}_{\widetilde{N}_{0,k}(\upsilon_n)} +{\bf 1}_{\widetilde{N}_{\geq1,k}(\upsilon_n)}$ and using equation \eqref{splitX}, we get 
	\begin{align}
	&M_{0,k,n}^b=\widehat{\E}_{t_k,Y_{t_k}^{b_0}}^{b_0}\Big[{\bf 1}_{\widehat{N}_{0,k}(\upsilon_n)}\Big(\widetilde{\E}_{t_k,Y_{t_{k}}^{b_0}}^{b}\Big[({\bf 1}_{\widetilde{N}_{0,k}(\upsilon_n)} +{\bf 1}_{\widetilde{N}_{\geq1,k}(\upsilon_n)}) \notag\\
	&\qquad\times\int_{t_k}^{t_{k+1}}\int_{z>\upsilon_n}zM(ds,dz)\big\vert X_{t_{k+1}}^{b}=Y_{t_{k+1}}^{b_0}\Big]\Big)^2\Big]\notag\\
	&=\widehat{\E}_{t_k,Y_{t_k}^{b_0}}^{b_0}\Big[{\bf 1}_{\widehat{N}_{0,k}(\upsilon_n)}\Big(\widetilde{\E}_{t_k,Y_{t_{k}}^{b_0}}^{b}\Big[{\bf 1}_{\widetilde{N}_{\geq1,k}(\upsilon_n)}\int_{t_k}^{t_{k+1}}\int_{z>\upsilon_n}zM(ds,dz)\big\vert X_{t_{k+1}}^{b}=Y_{t_{k+1}}^{b_0}\Big]\Big)^2\Big]\notag\\
	&=\widehat{\E}_{t_k,Y_{t_k}^{b_0}}^{b_0}\Big[{\bf 1}_{\widehat{N}_{0,k}(\upsilon_n)}\Big(\widetilde{\E}_{t_k,Y_{t_{k}}^{b_0}}^{b}\Big[{\bf 1}_{\widetilde{N}_{\geq1,k}(\upsilon_n)}\Big(X_{t_{k+1}}^{b}-X_{t_{k}}^{b}-\int_{t_k}^{t_{k+1}}(a-bX_s^{b})ds -\sigma\int_{t_k}^{t_{k+1}}\sqrt{X_s^{b}}dB_s\notag\\
	&\qquad-\int_{t_k}^{t_{k+1}}\int_{z\leq\upsilon_n}z\widetilde{M}(ds,dz)-\Delta_n\int_{z\leq\upsilon_n}zm(dz)\Big)\big\vert X_{t_{k+1}}^{b}=Y_{t_{k+1}}^{b_0}\Big]\Big)^2\Big] \notag\\
	&\leq 5\sum_{i=1}^{5}M_{0,i,k,n}^b,\label{M0}
	\end{align}
	where 
	\begin{align*}
	M_{0,1,k,n}^b&=\widehat{\E}_{t_k,Y_{t_k}^{b_0}}^{b_0}\Big[{\bf 1}_{\widehat{N}_{0,k}(\upsilon_n)}\Big(\widetilde{\E}_{t_k,Y_{t_{k}}^{b_0}}^{b}\Big[{\bf 1}_{\widetilde{N}_{\geq1,k}(\upsilon_n)}(X_{t_{k+1}}^{b}-X_{t_{k}}^{b})\big\vert X_{t_{k+1}}^{b}=Y_{t_{k+1}}^{b_0}\Big]\Big)^2\Big],\\
	M_{0,2,k,n}^b&=\widehat{\E}_{t_k,Y_{t_k}^{b_0}}^{b_0}\Big[{\bf 1}_{\widehat{N}_{0,k}(\upsilon_n)}\Big(\widetilde{\E}_{t_k,Y_{t_{k}}^{b_0}}^{b}\Big[{\bf 1}_{\widetilde{N}_{\geq1,k}(\upsilon_n)}\int_{t_k}^{t_{k+1}}(a-bX_s^{b})ds \big\vert X_{t_{k+1}}^{b}=Y_{t_{k+1}}^{b_0}\Big]\Big)^2\Big],\\
	M_{0,3,k,n}^b&=\widehat{\E}_{t_k,Y_{t_k}^{b_0}}^{b_0}\Big[{\bf 1}_{\widehat{N}_{0,k}(\upsilon_n)}\Big(\widetilde{\E}_{t_k,Y_{t_{k}}^{b_0}}^{b}\Big[{\bf 1}_{\widetilde{N}_{\geq1,k}(\upsilon_n)}\sigma\int_{t_k}^{t_{k+1}}\sqrt{X_s^{b}}dB_s\big\vert X_{t_{k+1}}^{b}=Y_{t_{k+1}}^{b_0}\Big]\Big)^2\Big],\\
	M_{0,4,k,n}^b&=\widehat{\E}_{t_k,Y_{t_k}^{b_0}}^{b_0}\Big[{\bf 1}_{\widehat{N}_{0,k}(\upsilon_n)}\Big(\widetilde{\E}_{t_k,Y_{t_{k}}^{b_0}}^{b}\Big[{\bf 1}_{\widetilde{N}_{\geq1,k}(\upsilon_n)}\int_{t_k}^{t_{k+1}}\int_{z\leq\upsilon_n}z\widetilde{M}(ds,dz)\big\vert X_{t_{k+1}}^{b}=Y_{t_{k+1}}^{b_0}\Big]\Big)^2\Big],\\
	M_{0,5,k,n}^b&=\widehat{\E}_{t_k,Y_{t_k}^{b_0}}^{b_0}\Big[{\bf 1}_{\widehat{N}_{0,k}(\upsilon_n)}\Big(\widetilde{\E}_{t_k,Y_{t_{k}}^{b_0}}^{b}\Big[{\bf 1}_{\widetilde{N}_{\geq1,k}(\upsilon_n)}\Delta_n\int_{z\leq\upsilon_n}zm(dz)\big\vert X_{t_{k+1}}^{b}=Y_{t_{k+1}}^{b_0}\Big]\Big)^2\Big].
	\end{align*}
	First, using equation \eqref{splitY} and the fact that there is no big jump of $J^{\upsilon_n}$ in $[t_k,t_{k+1})$, we get 
	\begin{align}
	&M_{0,1,k,n}^b=\widehat{\E}_{t_k,Y_{t_k}^{b_0}}^{b_0}\Big[{\bf 1}_{\widehat{N}_{0,k}(\upsilon_n)}\Big((Y_{t_{k+1}}^{b_0}-Y_{t_{k}}^{b_0})\widetilde{\E}_{t_k,Y_{t_{k}}^{b_0}}^{b}\Big[{\bf 1}_{\widetilde{N}_{\geq1,k}(\upsilon_n)}\big\vert X_{t_{k+1}}^{b}=Y_{t_{k+1}}^{b_0}\Big]\Big)^2\Big]\notag\\
	&=\widehat{\E}_{t_k,Y_{t_k}^{b_0}}^{b_0}\Big[{\bf 1}_{\widehat{N}_{0,k}(\upsilon_n)}\Big(\Big(\int_{t_k}^{t_{k+1}}(a-b_0Y_s^{b_0})ds+\sigma\int_{t_k}^{t_{k+1}}\sqrt{Y_s^{b_0}}dW_s+\Delta_n\int_{z\leq\upsilon_n}zm(dz)\notag\\
	&+\int_{t_k}^{t_{k+1}}\int_{z\leq\upsilon_n}z\widetilde{N}(ds,dz)\Big)\widetilde{\E}_{t_k,Y_{t_{k}}^{b_0}}^{b}\big[{\bf 1}_{\widetilde{N}_{\geq1,k}(\upsilon_n)}\vert X_{t_{k+1}}^{b}=Y_{t_{k+1}}^{b_0}\big]\Big)^2\Big] \notag\\
	&\leq 4\sum_{i=1}^{4}M_{0,1,i,k,n}^b,\label{M01}
	\end{align}
	where
	\begin{align*}
	&M_{0,1,1,k,n}^b=\widehat{\E}_{t_k,Y_{t_k}^{b_0}}^{b_0}\Big[{\bf 1}_{\widehat{N}_{0,k}(\upsilon_n)}\Big(\int_{t_k}^{t_{k+1}}(a-b_0Y_s^{b_0})ds\widetilde{\E}_{t_k,Y_{t_{k}}^{b_0}}^{b}\big[{\bf 1}_{\widetilde{N}_{\geq1,k}(\upsilon_n)}\vert X_{t_{k+1}}^{b}=Y_{t_{k+1}}^{b_0}\big]\Big)^2\Big],\\
	&M_{0,1,2,k,n}^b=\widehat{\E}_{t_k,Y_{t_k}^{b_0}}^{b_0}\Big[{\bf 1}_{\widehat{N}_{0,k}(\upsilon_n)}\Big(\sigma\int_{t_k}^{t_{k+1}}\sqrt{Y_s^{b_0}}dW_s\widetilde{\E}_{t_k,Y_{t_{k}}^{b_0}}^{b}\big[{\bf 1}_{\widetilde{N}_{\geq1,k}(\upsilon_n)}\vert X_{t_{k+1}}^{b}=Y_{t_{k+1}}^{b_0}\big]\Big)^2\Big],\\
	&M_{0,1,3,k,n}^b=\widehat{\E}_{t_k,Y_{t_k}^{b_0}}^{b_0}\Big[{\bf 1}_{\widehat{N}_{0,k}(\upsilon_n)}\Big(\int_{t_k}^{t_{k+1}}\int_{z\leq\upsilon_n}z\widetilde{N}(ds,dz)\widetilde{\E}_{t_k,Y_{t_{k}}^{b_0}}^{b}\big[{\bf 1}_{\widetilde{N}_{\geq1,k}(\upsilon_n)}\vert X_{t_{k+1}}^{b}=Y_{t_{k+1}}^{b_0}\big]\Big)^2\Big],\\
	&M_{0,1,4,k,n}^b=\widehat{\E}_{t_k,Y_{t_k}^{b_0}}^{b_0}\Big[{\bf 1}_{\widehat{N}_{0,k}(\upsilon_n)}\Big(\Delta_n\int_{z\leq\upsilon_n}zm(dz)\widetilde{\E}_{t_k,Y_{t_{k}}^{b_0}}^{b}\big[{\bf 1}_{\widetilde{N}_{\geq1,k}(\upsilon_n)}\vert X_{t_{k+1}}^{b}=Y_{t_{k+1}}^{b_0}\big]\Big)^2\Big].
	\end{align*}
	Using H\"older's inequality with $\frac{1}{p}+\frac{1}{q}=1$, Jensen's inequality, \eqref{for1} of Lemma \ref{change}, \eqref{for3} of Lemma \ref{deviation1} with $q=q_0>1$, we obtain
	\begin{align}
	&M_{0,1,1,k,n}^b\leq\big(\widehat{\E}_{t_k,Y_{t_k}^{b_0}}^{b_0}\big[\big\vert\int_{t_k}^{t_{k+1}}(a-b_0Y_s^{b_0})ds\big\vert^{2p}\big]\big)^{\frac{1}{p}}\big(\widehat{\E}_{t_k,Y_{t_k}^{b_0}}^{b_0}\big[\widetilde{\E}_{t_k,Y_{t_{k}}^{b_0}}^{b}\big[{\bf 1}_{\widetilde{N}_{\geq1,k}(\upsilon_n)}\vert X_{t_{k+1}}^{b}=Y_{t_{k+1}}^{b_0}\big]\big]\big)^{\frac{1}{q}}\notag\\
	&\leq\Big(\Delta_n^{2p-1}\int_{t_k}^{t_{k+1}}\widehat{\E}_{t_k,Y_{t_k}^{b_0}}^{b_0}\Big[\vert a-b_0Y_s^{b_0}\vert^{2p}\Big]ds\Big)^{\frac{1}{p}}\bigg(\widetilde{\P}_{t_k,Y_{t_{k}}^{b_0}}^{b}\big(\widetilde{N}_{\geq1,k}(\upsilon_n)\big)+C_1\vert u\vert\sqrt{\Delta_n}\varphi_{n\Delta_n}(b_0)\notag\\
	&\qquad\times\big(\widetilde{\P}_{t_k,Y_{t_{k}}^{b_0}}^{b}\big(\widetilde{N}_{\geq1,k}(\upsilon_n)\big)\big)^{\frac{1}{q_0}}e^{C_2u^2(\varphi_{n\Delta_n}(b_0))^2\Delta_nY_{t_k}^{b_0}} \big(1+\sqrt{Y_{t_k}^{b_0}}+\sqrt{\Delta_n}\vert u\vert\varphi_{n\Delta_n}(b_0) Y_{t_k}^{b_0}\big)\bigg)^{\frac{1}{q}}\notag\\
	&\leq C\Big(\Delta_n^{2p-1}\int_{t_k}^{t_{k+1}}\big(a^{2p}+\vert b_0\vert^{2p}\widehat{\E}_{t_k,Y_{t_k}^{b_0}}^{b_0}[\vert Y_s^{b_0}\vert^{2p}]\big)ds\Big)^{\frac{1}{p}}\bigg(\lambda_{\upsilon_n}\Delta_n+C_1\sqrt{\Delta_n}\varphi_{n\Delta_n}(b_0)\notag\\
	&\qquad\times\big(\lambda_{\upsilon_n}\Delta_n\big)^{\frac{1}{q_0}}e^{C_2u^2(\varphi_{n\Delta_n}(b_0))^2\Delta_nY_{t_k}^{b_0}}\big(1+\sqrt{Y_{t_k}^{b_0}}+\sqrt{\Delta_n}\vert u\vert\varphi_{n\Delta_n}(b_0) Y_{t_k}^{b_0}\big)\bigg)^{\frac{1}{q}}\notag\\		
	&\leq C(1+(Y_{t_k}^{b_0})^{2})\Delta_n^2\bigg(\left(\lambda_{\upsilon_n}\Delta_n\right)^{\frac{1}{q}}+(\sqrt{\Delta_n}\varphi_{n\Delta_n}(b_0))^{\frac{1}{q}}\big(\lambda_{\upsilon_n}\Delta_n\big)^{\frac{1}{q_0q}}\notag\\
	&\qquad \times e^{C_2u^2(\varphi_{n\Delta_n}(b_0))^2\Delta_nY_{t_k}^{b_0}}\big(1+\sqrt{Y_{t_k}^{b_0}}+\sqrt{\Delta_n}\vert u\vert\varphi_{n\Delta_n}(b_0) Y_{t_k}^{b_0}\big)^{\frac{1}{q}}\bigg).\label{M011}
	\end{align}
Here, we have used the fact that
\begin{align*} 
	\widetilde{\P}_{t_k,Y_{t_{k}}^{b_0}}^{b}\big(\widetilde{N}_{\geq1,k}(\upsilon_n)\big)=1-\widetilde{\P}_{t_k,Y_{t_{k}}^{b_0}}^{b}(\widetilde{N}_{0,k}(\upsilon_n))=1-e^{-\lambda_{\upsilon_n}\Delta_n}\leq \lambda_{\upsilon_n}\Delta_n.
\end{align*}
	Next, proceeding as for the term $M_{0,1,1,k,n}^b$, we get
	\begin{align}
	M_{0,1,2,k,n}^b&\leq C(1+Y_{t_k}^{b_0})\Delta_n\bigg(\left(\lambda_{\upsilon_n}\Delta_n\right)^{\frac{1}{q}}+(\sqrt{\Delta_n}\varphi_{n\Delta_n}(b_0))^{\frac{1}{q}}\big(\lambda_{\upsilon_n}\Delta_n\big)^{\frac{1}{q_0q}}\notag\\
	&\qquad \times e^{C_2u^2(\varphi_{n\Delta_n}(b_0))^2\Delta_nY_{t_k}^{b_0}}\big(1+\sqrt{Y_{t_k}^{b_0}}+\sqrt{\Delta_n}\vert u\vert\varphi_{n\Delta_n}(b_0) Y_{t_k}^{b_0}\big)^{\frac{1}{q}}\bigg).\label{M012}
	\end{align}
	Using BDG's inequality, we obtain
	\begin{align}
	M_{0,1,3,k,n}^b\leq\widehat{\E}_{t_k,Y_{t_k}^{b_0}}^{b_0}\Big[\Big(\int_{t_k}^{t_{k+1}}\int_{z\leq\upsilon_n}z\widetilde{N}(ds,dz)\Big)^2\Big]\leq C\Delta_n\int_{z\leq\upsilon_n}z^2m(dz).\label{M013}
	\end{align}
	Observe that
	\begin{align}
	M_{0,1,4,k,n}^b\leq\Big(\Delta_n\int_{z\leq\upsilon_n}zm(dz)\Big)^2.\label{M014}
	\end{align}
	Therefore, from \eqref{M01}-\eqref{M014}, we have shown that for $q>1$, $q_0>1$,
	\begin{align}
	M_{0,1,k,n}^b&\leq C(1+(Y_{t_k}^{b_0})^{2})\Delta_n\bigg(\left(\lambda_{\upsilon_n}\Delta_n\right)^{\frac{1}{q}}+(\sqrt{\Delta_n}\varphi_{n\Delta_n}(b_0))^{\frac{1}{q}}\big(\lambda_{\upsilon_n}\Delta_n\big)^{\frac{1}{q_0q}}\notag\\
	&\qquad \times e^{C_2u^2(\varphi_{n\Delta_n}(b_0))^2\Delta_nY_{t_k}^{b_0}} \big(1+\sqrt{Y_{t_k}^{b_0}}+\sqrt{\Delta_n}\vert u\vert\varphi_{n\Delta_n}(b_0) Y_{t_k}^{b_0}\big)^{\frac{1}{q}}\bigg)\notag\\
	&\qquad+C\Delta_n\bigg(\int_{z\leq\upsilon_n}z^2m(dz)+\Delta_n\Big(\int_{z\leq\upsilon_n}zm(dz)\Big)^2\bigg).\label{M01estimate}
	\end{align}
	Next, we treat $M_{0,2,k,n}^b$. Using Jensen's inequality, \eqref{for1} of Lemma \ref{change}, \eqref{for3} of Lemma \ref{deviation1} with $q=q_0>1$ and H\"older's inequality with $\frac{1}{p}+\frac{1}{q}=1$, we obtain
	\begin{align}
	&M_{0,2,k,n}^b\leq\widehat{\E}_{t_k,Y_{t_k}^{b_0}}^{b_0}\Big[\widetilde{\E}_{t_k,Y_{t_{k}}^{b_0}}^{b}\Big[{\bf 1}_{\widetilde{N}_{\geq1,k}(\upsilon_n)}\Big(\int_{t_k}^{t_{k+1}}(a-bX_s^{b})ds\Big)^2\big\vert X_{t_{k+1}}^{b}=Y_{t_{k+1}}^{b_0}\Big]\Big]\notag\\
	&=\widetilde{\E}_{t_k,Y_{t_{k}}^{b_0}}^{b}\Big[{\bf 1}_{\widetilde{N}_{\geq1,k}(\upsilon_n)}\Big(\int_{t_k}^{t_{k+1}}(a-bX_s^{b})ds\Big)^2\Big] +C_1\vert u\vert\sqrt{\Delta_n}\varphi_{n\Delta_n}(b_0)  e^{C_2u^2(\varphi_{n\Delta_n}(b_0))^2\Delta_nY_{t_k}^{b_0}}\notag\\ &\qquad\times \Big(\widetilde{\E}_{t_k,Y_{t_{k}}^{b_0}}^{b}\Big[{\bf 1}_{\widetilde{N}_{\geq1,k}(\upsilon_n)}\Big\vert\int_{t_k}^{t_{k+1}}(a-bX_s^{b})ds\Big\vert^{2q_0}\Big]\Big)^{\frac{1}{q_0}} \big(1+\sqrt{Y_{t_k}^{b_0}}+\sqrt{\Delta_n}\vert u\vert\varphi_{n\Delta_n}(b_0) Y_{t_k}^{b_0}\big)\notag\\
	&\leq \Big(\widetilde{\E}_{t_k,Y_{t_k}^{b_0}}^{b}\Big[\Big\vert\int_{t_k}^{t_{k+1}}(a-bX_s^{b})ds\Big\vert^{2p}\Big]\Big)^{\frac{1}{p}}\Big(\widetilde{\P}_{t_k,Y_{t_{k}}^{b_0}}^{b}\big(\widetilde{N}_{\geq1,k}(\upsilon_n)\big)\Big)^{\frac{1}{q}}+C_1\vert u\vert\sqrt{\Delta_n}\varphi_{n\Delta_n}(b_0)\notag\\ &\qquad\times \Big(\widetilde{\E}_{t_k,Y_{t_k}^{b_0}}^{b}\Big[\Big\vert\int_{t_k}^{t_{k+1}}(a-bX_s^{b})ds\Big\vert^{2q_0p}\Big]\Big)^{\frac{1}{q_0p}}\Big(\widetilde{\P}_{t_k,Y_{t_{k}}^{b_0}}^{b}\big(\widetilde{N}_{\geq1,k}(\upsilon_n)\big)\Big)^{\frac{1}{q_0q}} e^{C_2u^2(\varphi_{n\Delta_n}(b_0))^2\Delta_nY_{t_k}^{b_0}}\notag\\
	&\qquad\times \big(1+\sqrt{Y_{t_k}^{b_0}}+\sqrt{\Delta_n}\vert u\vert\varphi_{n\Delta_n}(b_0) Y_{t_k}^{b_0}\big)\notag\\
	&\leq C(1+(Y_{t_k}^{b_0})^{2})\Delta_n^2\bigg(\left(\lambda_{\upsilon_n}\Delta_n\right)^{\frac{1}{q}}+\sqrt{\Delta_n}\varphi_{n\Delta_n}(b_0) \left(\lambda_{\upsilon_n}\Delta_n\right)^{\frac{1}{q_0q}} e^{C_2u^2(\varphi_{n\Delta_n}(b_0))^2\Delta_nY_{t_k}^{b_0}}\notag\\
	&\qquad\times\big(1+\sqrt{Y_{t_k}^{b_0}}+\sqrt{\Delta_n}\vert u\vert\varphi_{n\Delta_n}(b_0) Y_{t_k}^{b_0}\big)\bigg).\label{M02}
	\end{align}
	Proceeding as for the term $M_{0,2,k,n}^b$, we get for $q>1$, $q_0>1$,
	\begin{align}
	M_{0,3,k,n}^b&\leq C(1+Y_{t_k}^{b_0})\Delta_n\bigg(\left(\lambda_{\upsilon_n}\Delta_n\right)^{\frac{1}{q}}+\sqrt{\Delta_n}\varphi_{n\Delta_n}(b_0)\left(\lambda_{\upsilon_n}\Delta_n\right)^{\frac{1}{q_0q}}\notag\\
	&\qquad\times e^{C_2u^2(\varphi_{n\Delta_n}(b_0))^2\Delta_nY_{t_k}^{b_0}}\big(1+\sqrt{Y_{t_k}^{b_0}}+\sqrt{\Delta_n}\vert u\vert\varphi_{n\Delta_n}(b_0) Y_{t_k}^{b_0}\big)\bigg).\label{M03}
	\end{align}
	Using Jensen's inequality and proceeding as for the term $D_{2,k,n}$, we get for $q_1\in (1,2]$,
	\begin{align}
	M_{0,4,k,n}^b&\leq\widehat{\E}_{t_k,Y_{t_k}^{b_0}}^{b_0}\Big[\widetilde{\E}_{t_k,Y_{t_{k}}^{b_0}}^{b}\Big[\Big(\int_{t_k}^{t_{k+1}}\int_{z\leq\upsilon_n}z\widetilde{M}(ds,dz)\Big)^2\big\vert X_{t_{k+1}}^{b}=Y_{t_{k+1}}^{b_0}\Big]\Big]\notag\\
	&\leq C\Delta_n\int_{z\leq\upsilon_n}z^2m(dz)+C_1\varphi_{n\Delta_n}(b_0)\Delta_n^{\frac{1}{2}+\frac{1}{q_1}}\Big(\int_{z\leq\upsilon_n}z^{2q_1}m(dz)\Big)^{\frac{1}{q_1}}\notag\\ &\qquad\times e^{C_2u^2(\varphi_{n\Delta_n}(b_0))^2\Delta_nY_{t_k}^{b_0}}\big(1+\sqrt{Y_{t_k}^{b_0}}+\sqrt{\Delta_n}\vert u\vert\varphi_{n\Delta_n}(b_0) Y_{t_k}^{b_0}\big).\label{M04}
	\end{align}
	Finally, observe that
	\begin{align}
	M_{0,5,k,n}^b\leq\Big(\Delta_n\int_{z\leq\upsilon_n}zm(dz)\Big)^2.\label{M05}
	\end{align}
	Thus, from \eqref{M0} and \eqref{M01estimate}-\eqref{M05}, we have shown that for $n$ large enough, $q_1\in (1,2]$, $q>1$, $q_0>1$,
	\begin{align} 
	M_{0,k,n}^b&\leq C(1+(Y_{t_k}^{b_0})^{2})\Delta_n\bigg(\left(\lambda_{\upsilon_n}\Delta_n\right)^{\frac{1}{q}}+(\sqrt{\Delta_n}\varphi_{n\Delta_n}(b_0))^{\frac{1}{q}}\big(\lambda_{\upsilon_n}\Delta_n\big)^{\frac{1}{q_0q}} e^{C_2u^2(\varphi_{n\Delta_n}(b_0))^2\Delta_nY_{t_k}^{b_0}}\notag\\
	&\qquad \times  \big(1+\sqrt{Y_{t_k}^{b_0}}+\sqrt{\Delta_n}\vert u\vert\varphi_{n\Delta_n}(b_0) Y_{t_k}^{b_0}\big) \bigg)+C\Delta_n\bigg(\int_{z\leq\upsilon_n}z^2m(dz)\notag\\
	&\qquad+\Delta_n\Big(\int_{z\leq\upsilon_n}zm(dz)\Big)^2\bigg)+C\varphi_{n\Delta_n}(b_0)\Delta_n^{\frac{1}{2}+\frac{1}{q_1}}\Big(\int_{z\leq\upsilon_n}z^{2q_1}m(dz)\Big)^{\frac{1}{q_1}} \notag\\ 
	&\qquad\times e^{C_2u^2(\varphi_{n\Delta_n}(b_0))^2\Delta_nY_{t_k}^{b_0}} \big(1+\sqrt{Y_{t_k}^{b_0}}+\sqrt{\Delta_n}\vert u\vert\varphi_{n\Delta_n}(b_0) Y_{t_k}^{b_0}\big). \label{c3m1}
	\end{align}		
	Finally, we treat $M_{\geq1,k,n}^b$. Multiplying the random variable inside the conditional expectation of $M_{\geq1,k,n}^b$ by ${\bf 1}_{\widetilde{N}_{0,k}(\upsilon_n)} +{\bf 1}_{\widetilde{N}_{\geq1,k}(\upsilon_n)}$ and using equations \eqref{splitY} and \eqref{splitX}, we get 
	\begin{align}
	M_{\geq1,k,n}^b&=\widehat{\E}_{t_k,Y_{t_k}^{b_0}}^{b_0}\Big[{\bf 1}_{\widehat{N}_{\geq1,k}(\upsilon_n)}\Big(\int_{t_k}^{t_{k+1}}\int_{z>\upsilon_n}zN(ds,dz)\notag\\
	&\quad-\widetilde{\E}_{t_k,Y_{t_{k}}^{b_0}}^{b}\Big[\big({\bf 1}_{\widetilde{N}_{0,k}(\upsilon_n)} +{\bf 1}_{\widetilde{N}_{\geq1,k}(\upsilon_n)}\big)\int_{t_k}^{t_{k+1}}\int_{z>\upsilon_n}zM(ds,dz)\big\vert X_{t_{k+1}}^{b}=Y_{t_{k+1}}^{b_0}\Big]\Big)^2\Big]\notag\\
	&=\widehat{\E}_{t_k,Y_{t_k}^{b_0}}^{b_0}\Big[{\bf 1}_{\widehat{N}_{\geq1,k}(\upsilon_n)}\Big(\int_{t_k}^{t_{k+1}}\int_{z>\upsilon_n}zN(ds,dz)\notag\\
	&\quad-\widetilde{\E}_{t_k,Y_{t_{k}}^{b_0}}^{b}\Big[{\bf 1}_{\widetilde{N}_{\geq1,k}(\upsilon_n)}\int_{t_k}^{t_{k+1}}\int_{z>\upsilon_n}zM(ds,dz)\big\vert X_{t_{k+1}}^{b}=Y_{t_{k+1}}^{b_0}\Big]\Big)^2\Big]\notag\\
	&=\widehat{\E}_{t_k,Y_{t_k}^{b_0}}^{b_0}\Big[{\bf 1}_{\widehat{N}_{\geq1,k}(\upsilon_n)}\Big(Y_{t_{k+1}}^{b_0}-Y_{t_{k}}^{b_0}-\int_{t_k}^{t_{k+1}}(a-b_0Y_s^{b_0})ds-\sigma\int_{t_k}^{t_{k+1}}\sqrt{Y_s^{b_0}}dW_s\notag\\
	&\quad-\int_{t_k}^{t_{k+1}}\int_{z\leq\upsilon_n}z\widetilde{N}(ds,dz)-\Delta_n\int_{z\leq\upsilon_n}zm(dz)-\widetilde{\E}_{t_k,Y_{t_{k}}^{b_0}}^{b}\Big[{\bf 1}_{\widetilde{N}_{\geq1,k}(\upsilon_n)}\Big(X_{t_{k+1}}^{b}-X_{t_{k}}^{b}\notag\\
	&\quad-\int_{t_k}^{t_{k+1}}(a-bX_s^{b})ds-\sigma\int_{t_k}^{t_{k+1}}\sqrt{X_s^{b}}dB_s-\int_{t_k}^{t_{k+1}}\int_{z\leq\upsilon_n}z\widetilde{M}(ds,dz)\notag\\
	&\quad-\Delta_n\int_{z\leq\upsilon_n}zm(dz)\Big)\big\vert X_{t_{k+1}}^{b}=Y_{t_{k+1}}^{b_0}\Big]\Big)^2\Big] \notag\\
	&\leq 9\sum_{i=1}^{9}M_{\geq1,i,k,n}^b,\label{M1}
	\end{align} 
	where 
	\begin{align*}
	M_{\geq1,1,k,n}^b&=\widehat{\E}_{t_k,Y_{t_k}^{b_0}}^{b_0}\Big[{\bf 1}_{\widehat{N}_{\geq1,k}(\upsilon_n)}\big(Y_{t_{k+1}}^{b_0}-Y_{t_{k}}^{b_0}-\widetilde{\E}_{t_k,Y_{t_{k}}^{b_0}}^{b}\big[{\bf 1}_{\widetilde{N}_{\geq1,k}(\upsilon_n)}(X_{t_{k+1}}^{b}-X_{t_{k}}^{b})\big\vert X_{t_{k+1}}^{b}=Y_{t_{k+1}}^{b_0}\big]\big)^2\Big],\\
	M_{\geq1,2,k,n}^b&=\widehat{\E}_{t_k,Y_{t_k}^{b_0}}^{b_0}\Big[{\bf 1}_{\widehat{N}_{\geq1,k}(\upsilon_n)}\Big(\int_{t_k}^{t_{k+1}}(a-b_0Y_s^{b_0})ds\Big)^2\Big],\\
	M_{\geq1,3,k,n}^b&=\widehat{\E}_{t_k,Y_{t_k}^{b_0}}^{b_0}\Big[{\bf 1}_{\widehat{N}_{\geq1,k}(\upsilon_n)}\Big(\sigma\int_{t_k}^{t_{k+1}}\sqrt{Y_s^{b_0}}dW_s\Big)^2\Big],\\
	M_{\geq1,4,k,n}^b&=\widehat{\E}_{t_k,Y_{t_k}^{b_0}}^{b_0}\Big[{\bf 1}_{\widehat{N}_{\geq1,k}(\upsilon_n)}\Big(\int_{t_k}^{t_{k+1}}\int_{z\leq\upsilon_n}z\widetilde{N}(ds,dz)\Big)^2\Big],\\
	M_{\geq1,5,k,n}^b&=\widehat{\E}_{t_k,Y_{t_k}^{b_0}}^{b_0}\Big[{\bf 1}_{\widehat{N}_{\geq1,k}(\upsilon_n)}\Big(\Delta_n\int_{z\leq\upsilon_n}zm(dz)\Big)^2\Big],\\
	M_{\geq1,6,k,n}^b&=\widehat{\E}_{t_k,Y_{t_k}^{b_0}}^{b_0}\Big[{\bf 1}_{\widehat{N}_{\geq1,k}(\upsilon_n)}\Big(\widetilde{\E}_{t_k,Y_{t_{k}}^{b_0}}^{b}\Big[{\bf 1}_{\widetilde{N}_{\geq1,k}(\upsilon_n)}\int_{t_k}^{t_{k+1}}(a-bX_s^{b})ds\big\vert X_{t_{k+1}}^{b}=Y_{t_{k+1}}^{b_0}\Big]\Big)^2\Big],\\
	M_{\geq1,7,k,n}^b&=\widehat{\E}_{t_k,Y_{t_k}^{b_0}}^{b_0}\Big[{\bf 1}_{\widehat{N}_{\geq1,k}(\upsilon_n)}\Big(\widetilde{\E}_{t_k,Y_{t_{k}}^{b_0}}^{b}\Big[{\bf 1}_{\widetilde{N}_{\geq1,k}(\upsilon_n)}\sigma\int_{t_k}^{t_{k+1}}\sqrt{X_s^{b}}dB_s\big\vert X_{t_{k+1}}^{b}=Y_{t_{k+1}}^{b_0}\Big]\Big)^2\Big],\\
	M_{\geq1,8,k,n}^b&=\widehat{\E}_{t_k,Y_{t_k}^{b_0}}^{b_0}\Big[{\bf 1}_{\widehat{N}_{\geq1,k}(\upsilon_n)}\Big(\widetilde{\E}_{t_k,Y_{t_{k}}^{b_0}}^{b}\Big[{\bf 1}_{\widetilde{N}_{\geq1,k}(\upsilon_n)}\int_{t_k}^{t_{k+1}}\int_{z\leq\upsilon_n}z\widetilde{M}(ds,dz)\big\vert X_{t_{k+1}}^{b}=Y_{t_{k+1}}^{b_0}\Big]\Big)^2\Big],\\
	M_{\geq1,9,k,n}^b&=\widehat{\E}_{t_k,Y_{t_k}^{b_0}}^{b_0}\Big[{\bf 1}_{\widehat{N}_{\geq1,k}(\upsilon_n)}\Big(\widetilde{\E}_{t_k,Y_{t_{k}}^{b_0}}^{b}\Big[{\bf 1}_{\widetilde{N}_{\geq1,k}(\upsilon_n)}\Delta_n\int_{z\leq\upsilon_n}zm(dz)\big\vert X_{t_{k+1}}^{b}=Y_{t_{k+1}}^{b_0}\Big]\Big)^2\Big].
	\end{align*}
	First, we treat the term $M_{\geq1,1,k,n}^b$. Using equation \eqref{splitX} and the fact that there is no big jump of $\widetilde{J}^{\upsilon_n}$ in $[t_k,t_{k+1})$, we get 
	\begin{align}
	M_{\geq1,1,k,n}^b&=\widehat{\E}_{t_k,Y_{t_k}^{b_0}}^{b_0}\Big[{\bf 1}_{\widehat{N}_{\geq1,k}(\upsilon_n)}\big(Y_{t_{k+1}}^{b_0}-Y_{t_{k}}^{b_0}-(Y_{t_{k+1}}^{b_0}-Y_{t_{k}}^{b_0})\widetilde{\E}_{t_k,Y_{t_{k}}^{b_0}}^{b}\big[{\bf 1}_{\widetilde{N}_{\geq1,k}(\upsilon_n)}\big\vert X_{t_{k+1}}^{b}=Y_{t_{k+1}}^{b_0}\big]\big)^2\Big]\notag\\
	&=\widehat{\E}_{t_k,Y_{t_k}^{b_0}}^{b_0}\Big[{\bf 1}_{\widehat{N}_{\geq1,k}(\upsilon_n)}\Big(\big(Y_{t_{k+1}}^{b_0}-Y_{t_{k}}^{b_0}\big)\Big(1-\widetilde{\E}_{t_k,Y_{t_{k}}^{b_0}}^{b}\Big[{\bf 1}_{\widetilde{N}_{\geq1,k}(\upsilon_n)}\big\vert X_{t_{k+1}}^{b}=Y_{t_{k+1}}^{b_0}\Big]\Big)\Big)^2\Big]\notag\\
	&=\widehat{\E}_{t_k,Y_{t_k}^{b_0}}^{b_0}\Big[{\bf 1}_{\widehat{N}_{\geq1,k}(\upsilon_n)}\Big(\big(Y_{t_{k+1}}^{b_0}-Y_{t_{k}}^{b_0}\big)\widetilde{\E}_{t_k,Y_{t_{k}}^{b_0}}^{b}\Big[1-{\bf 1}_{\widetilde{N}_{\geq1,k}(\upsilon_n)}\big\vert X_{t_{k+1}}^{b}=Y_{t_{k+1}}^{b_0}\Big]\Big)^2\Big]\notag\\
	&=\widehat{\E}_{t_k,Y_{t_k}^{b_0}}^{b_0}\Big[{\bf 1}_{\widehat{N}_{\geq1,k}(\upsilon_n)}\Big(\big(Y_{t_{k+1}}^{b_0}-Y_{t_{k}}^{b_0}\big)\widetilde{\E}_{t_k,Y_{t_{k}}^{b_0}}^{b}\Big[{\bf 1}_{\widetilde{N}_{0,k}(\upsilon_n)}\big\vert X_{t_{k+1}}^{b}=Y_{t_{k+1}}^{b_0}\Big]\Big)^2\Big]\notag\\
	&=\widehat{\E}_{t_k,Y_{t_k}^{b_0}}^{b_0}\Big[{\bf 1}_{\widehat{N}_{\geq1,k}(\upsilon_n)}\Big(\widetilde{\E}_{t_k,Y_{t_{k}}^{b_0}}^{b}\Big[{\bf 1}_{\widetilde{N}_{0,k}(\upsilon_n)}(X_{t_{k+1}}^{b}-X_{t_{k}}^{b})\big\vert X_{t_{k+1}}^{b}=Y_{t_{k+1}}^{b_0}\Big]\Big)^2\Big]\notag\\
	&=\widehat{\E}_{t_k,Y_{t_k}^{b_0}}^{b_0}\Big[{\bf 1}_{\widehat{N}_{\geq1,k}(\upsilon_n)}\Big(\widetilde{\E}_{t_k,Y_{t_{k}}^{b_0}}^{b}\Big[{\bf 1}_{\widetilde{N}_{0,k}(\upsilon_n)}\Big(\int_{t_k}^{t_{k+1}}(a-bX_s^{b})ds+\sigma\int_{t_k}^{t_{k+1}}\sqrt{X_s^{b}}dB_s\notag\\
	&\qquad+\int_{t_k}^{t_{k+1}}\int_{z\leq\upsilon_n}z\widetilde{M}(ds,dz)+\Delta_n\int_{z\leq\upsilon_n}zm(dz)\Big)\big\vert X_{t_{k+1}}^{b}=Y_{t_{k+1}}^{b_0}\Big]\Big)^2\Big]\notag\\
	&\leq 4\sum_{i=1}^{4}M_{\geq1,1,i,k,n}^b,\label{M11}
	\end{align}
	where
	\begin{align*}
	M_{\geq1,1,1,k,n}^b&=\widehat{\E}_{t_k,Y_{t_k}^{b_0}}^{b_0}\Big[{\bf 1}_{\widehat{N}_{\geq1,k}(\upsilon_n)}\Big(\widetilde{\E}_{t_k,Y_{t_{k}}^{b_0}}^{b}\Big[{\bf 1}_{\widetilde{N}_{0,k}(\upsilon_n)}\int_{t_k}^{t_{k+1}}(a-bX_s^{b})ds\big\vert X_{t_{k+1}}^{b}=Y_{t_{k+1}}^{b_0}\Big]\Big)^2\Big],\\
	M_{\geq1,1,2,k,n}^b&=\widehat{\E}_{t_k,Y_{t_k}^{b_0}}^{b_0}\Big[{\bf 1}_{\widehat{N}_{\geq1,k}(\upsilon_n)}\Big(\widetilde{\E}_{t_k,Y_{t_{k}}^{b_0}}^{b}\Big[{\bf 1}_{\widetilde{N}_{0,k}(\upsilon_n)}\sigma\int_{t_k}^{t_{k+1}}\sqrt{X_s^{b}}dB_s\big\vert X_{t_{k+1}}^{b}=Y_{t_{k+1}}^{b_0}\Big]\Big)^2\Big],\\
	M_{\geq1,1,3,k,n}^b&=\widehat{\E}_{t_k,Y_{t_k}^{b_0}}^{b_0}\Big[{\bf 1}_{\widehat{N}_{\geq1,k}(\upsilon_n)}\Big(\widetilde{\E}_{t_k,Y_{t_{k}}^{b_0}}^{b}\Big[{\bf 1}_{\widetilde{N}_{0,k}(\upsilon_n)}\int_{t_k}^{t_{k+1}}\int_{z\leq\upsilon_n}z\widetilde{M}(ds,dz)\big\vert X_{t_{k+1}}^{b}=Y_{t_{k+1}}^{b_0}\Big]\Big)^2\Big],\\
	M_{\geq1,1,4,k,n}^b&=\widehat{\E}_{t_k,Y_{t_k}^{b_0}}^{b_0}\Big[{\bf 1}_{\widehat{N}_{\geq1,k}(\upsilon_n)}\Big(\widetilde{\E}_{t_k,Y_{t_{k}}^{b_0}}^{b}\Big[{\bf 1}_{\widetilde{N}_{0,k}(\upsilon_n)}\Delta_n\int_{z\leq\upsilon_n}zm(dz)\big\vert X_{t_{k+1}}^{b}=Y_{t_{k+1}}^{b_0}\Big]\Big)^2\Big].
	\end{align*}
	Using H\"older's inequality with $\frac{1}{p}+\frac{1}{q}=1$, \eqref{for1} of Lemma \ref{change}, \eqref{for3} of Lemma \ref{deviation1} with $q=q_0>1$, we get
	\begin{align} 
	&M_{\geq1,1,1,k,n}^b\leq \big(\widehat{\E}_{t_k,Y_{t_k}^{b_0}}^{b_0}\big[\widetilde{\E}_{t_k,Y_{t_{k}}^{b_0}}^{b}\big[\big\vert\int_{t_k}^{t_{k+1}}(a-bX_s^{b})ds\big\vert^{2p}\big\vert X_{t_{k+1}}^{b}=Y_{t_{k+1}}^{b_0}\big]\big]\big)^{\frac{1}{p}}\big(\widehat{\P}_{t_k,Y_{t_{k}}^{b_0}}^{b}\big(\widehat{N}_{\geq1,k}(\upsilon_n)\big)\big)^{\frac{1}{q}}\notag\\
	&\leq \left(\lambda_{\upsilon_n}\Delta_n\right)^{\frac{1}{q}}\bigg(\widetilde{\E}_{t_k,Y_{t_{k}}^{b_0}}^{b}\Big[\Big\vert\int_{t_k}^{t_{k+1}}(a-bX_s^{b})ds\Big\vert^{2p}\Big]+C_1\vert u\vert\sqrt{\Delta_n}\varphi_{n\Delta_n}(b_0) e^{C_2u^2(\varphi_{n\Delta_n}(b_0))^2\Delta_nY_{t_k}^{b_0}}\notag \\
	&\qquad\times\Big(\widetilde{\E}_{t_k,Y_{t_{k}}^{b_0}}^{b}\Big[\Big\vert\int_{t_k}^{t_{k+1}}(a-bX_s^{b})ds\Big\vert^{2pq_0}\Big]\Big)^{\frac{1}{q_0}} \big(1+\sqrt{Y_{t_k}^{b_0}}+\sqrt{\Delta_n}\vert u\vert\varphi_{n\Delta_n}(b_0) Y_{t_k}^{b_0}\big)\bigg)^{\frac{1}{p}}\notag\\
	&\leq C(1+(Y_{t_k}^{b_0})^{2})\Delta_n^2\left(\lambda_{\upsilon_n}\Delta_n\right)^{\frac{1}{q}}\bigg(1+(\sqrt{\Delta_n}\varphi_{n\Delta_n}(b_0))^{\frac{1}{p}}\notag\\
	&\qquad \times e^{C_2u^2(\varphi_{n\Delta_n}(b_0))^2\Delta_nY_{t_k}^{b_0}} \big(1+\sqrt{Y_{t_k}^{b_0}}+\sqrt{\Delta_n}\vert u\vert\varphi_{n\Delta_n}(b_0) Y_{t_k}^{b_0}\big)^{\frac{1}{p}}\bigg).
	\end{align}
	Proceeding as for $M_{\geq1,1,1,k,n}^b$, we get 
	\begin{align} 
	M_{\geq1,1,2,k,n}^b&\leq  C(1+Y_{t_k}^{b_0})\Delta_n\left(\lambda_{\upsilon_n}\Delta_n\right)^{\frac{1}{q}}\bigg(1+(\sqrt{\Delta_n}\varphi_{n\Delta_n}(b_0))^{\frac{1}{p}}\notag\\
	&\qquad \times e^{C_2u^2(\varphi_{n\Delta_n}(b_0))^2\Delta_nY_{t_k}^{b_0}} \big(1+\sqrt{Y_{t_k}^{b_0}}+\sqrt{\Delta_n}\vert u\vert\varphi_{n\Delta_n}(b_0) Y_{t_k}^{b_0}\big)^{\frac{1}{p}}\bigg).
	\end{align}
	Proceeding as for $D_{2,k,n}$, we get for $q_1\in (1,2]$,
	\begin{align} 
	M_{\geq1,1,3,k,n}^b&\leq \widehat{\E}_{t_k,Y_{t_k}^{b_0}}^{b_0}\Big[\widetilde{\E}_{t_k,Y_{t_{k}}^{b_0}}^{b}\Big[\Big(\int_{t_k}^{t_{k+1}}\int_{z\leq\upsilon_n}z\widetilde{M}(ds,dz)\Big)^2\big\vert X_{t_{k+1}}^{b}=Y_{t_{k+1}}^{b_0}\Big]\Big]\notag\\
	&\leq C\Delta_n\int_{z\leq\upsilon_n}z^2m(dz)+C_1\varphi_{n\Delta_n}(b_0)\Delta_n^{\frac{1}{2}+\frac{1}{q_1}}\Big(\int_{z\leq\upsilon_n}z^{2q_1}m(dz)\Big)^{\frac{1}{q_1}}\notag\\ &\qquad\times e^{C_2u^2(\varphi_{n\Delta_n}(b_0))^2\Delta_nY_{t_k}^{b_0}}\big(1+\sqrt{Y_{t_k}^{b_0}}+\sqrt{\Delta_n}\vert u\vert\varphi_{n\Delta_n}(b_0) Y_{t_k}^{b_0}\big).
	\end{align} 	
	Now, observe that
	\begin{align} 
	M_{\geq1,1,4,k,n}^b\leq \Big(\Delta_n\int_{z\leq\upsilon_n}zm(dz)\Big)^2.\label{M111}
	\end{align} 
	Thus, from \eqref{M11}-\eqref{M111}, we have shown that for $p, q>1$ satisfying $\frac{1}{p}+\frac{1}{q}=1$ and $q_1\in (1,2]$,
	\begin{align}\label{M11estimate}
	M_{\geq1,1,k,n}^b&\leq C(1+(Y_{t_k}^{b_0})^2)\Delta_n\left(\lambda_{\upsilon_n}\Delta_n\right)^{\frac{1}{q}}\bigg(1+(\sqrt{\Delta_n}\varphi_{n\Delta_n}(b_0))^{\frac{1}{p}} e^{C_2u^2(\varphi_{n\Delta_n}(b_0))^2\Delta_nY_{t_k}^{b_0}}\notag\\
	&\qquad \times  \big(1+\sqrt{Y_{t_k}^{b_0}}+\sqrt{\Delta_n}\vert u\vert\varphi_{n\Delta_n}(b_0) Y_{t_k}^{b_0}\big)^{\frac{1}{p}}\bigg)+C\Delta_n\bigg(\int_{z\leq\upsilon_n}z^2m(dz)\notag\\
	&\qquad+\Delta_n\Big(\int_{z\leq\upsilon_n}zm(dz)\Big)^2\bigg)+C\varphi_{n\Delta_n}(b_0)\Delta_n^{\frac{1}{2}+\frac{1}{q_1}}\Big(\int_{z\leq\upsilon_n}z^{2q_1}m(dz)\Big)^{\frac{1}{q_1}} \notag\\ 
	&\qquad\times e^{C_2u^2(\varphi_{n\Delta_n}(b_0))^2\Delta_nY_{t_k}^{b_0}} \big(1+\sqrt{Y_{t_k}^{b_0}}+\sqrt{\Delta_n}\vert u\vert\varphi_{n\Delta_n}(b_0) Y_{t_k}^{b_0}\big).
	\end{align} 
	Using H\"older's inequality with $\frac{1}{p}+\frac{1}{q}=1$,
	\begin{align}
	M_{\geq1,2,k,n}^b&\leq C(1+(Y_{t_k}^{b_0})^{2})\Delta_n^2\left(\lambda_{\upsilon_n}\Delta_n\right)^{\frac{1}{q}},\label{M2}
	\end{align}
	and
	\begin{align}
	M_{\geq1,3,k,n}^b&\leq C(1+Y_{t_k}^{b_0})\Delta_n\left(\lambda_{\upsilon_n}\Delta_n\right)^{\frac{1}{q}}.\label{M3}
	\end{align}
	Notice that
	\begin{align}
	M_{\geq1,4,k,n}^b\leq C\Delta_n\int_{z\leq\upsilon_n}z^2m(dz).\label{M4}
	\end{align}
	Observe that
\begin{align}\label{M9}
	M_{\geq1,5,k,n}^b+M_{\geq1,9,k,n}^b\leq \Big(\Delta_n\int_{z\leq\upsilon_n}zm(dz)\Big)^2.
\end{align} 
	As for the term $M_{\geq1,1,1,k,n}^b$, we get for $p, q>1$ satisfying $\frac{1}{p}+\frac{1}{q}=1$,
	\begin{align}
	M_{\geq1,6,k,n}^b&\leq C(1+(Y_{t_k}^{b_0})^{2})\Delta_n^2\left(\lambda_{\upsilon_n}\Delta_n\right)^{\frac{1}{q}}\bigg(1+(\sqrt{\Delta_n}\varphi_{n\Delta_n}(b_0))^{\frac{1}{p}}\notag\\
	&\qquad \times e^{C_2u^2(\varphi_{n\Delta_n}(b_0))^2\Delta_nY_{t_k}^{b_0}} \big(1+\sqrt{Y_{t_k}^{b_0}}+\sqrt{\Delta_n}\vert u\vert\varphi_{n\Delta_n}(b_0) Y_{t_k}^{b_0}\big)^{\frac{1}{p}}\bigg), \label{M6}
	\end{align}
	and
	\begin{align}
	M_{\geq1,7,k,n}^b&\leq C(1+Y_{t_k}^{b_0})\Delta_n\left(\lambda_{\upsilon_n}\Delta_n\right)^{\frac{1}{q}}\bigg(1+(\sqrt{\Delta_n}\varphi_{n\Delta_n}(b_0))^{\frac{1}{p}}\notag\\
	&\qquad \times e^{C_2u^2(\varphi_{n\Delta_n}(b_0))^2\Delta_nY_{t_k}^{b_0}} \big(1+\sqrt{Y_{t_k}^{b_0}}+\sqrt{\Delta_n}\vert u\vert\varphi_{n\Delta_n}(b_0) Y_{t_k}^{b_0}\big)^{\frac{1}{p}}\bigg). \label{M7}
	\end{align}
	As for $D_{2,k,n}$, we get for $q_1\in (1,2]$,
	\begin{align} 
	M_{\geq1,8,k,n}^b&\leq \widehat{\E}_{t_k,Y_{t_k}^{b_0}}^{b_0}\Big[\widetilde{\E}_{t_k,Y_{t_{k}}^{b_0}}^{b}\Big[\Big(\int_{t_k}^{t_{k+1}}\int_{z\leq\upsilon_n}z\widetilde{M}(ds,dz)\Big)^2\big\vert X_{t_{k+1}}^{b}=Y_{t_{k+1}}^{b_0}\Big]\Big]\notag\\
	&\leq C\Delta_n\int_{z\leq\upsilon_n}z^2m(dz)+C_1\varphi_{n\Delta_n}(b_0)\Delta_n^{\frac{1}{2}+\frac{1}{q_1}}\Big(\int_{z\leq\upsilon_n}z^{2q_1}m(dz)\Big)^{\frac{1}{q_1}}\notag\\ &\qquad\times e^{C_2u^2(\varphi_{n\Delta_n}(b_0))^2\Delta_nY_{t_k}^{b_0}}\big(1+\sqrt{Y_{t_k}^{b_0}}+\sqrt{\Delta_n}\vert u\vert\varphi_{n\Delta_n}(b_0) Y_{t_k}^{b_0}\big). \label{M8}
	\end{align} 
	Consequently, from \eqref{M1} and \eqref{M11estimate}-\eqref{M8}, we conclude for $p, q>1$ satisfying $\frac{1}{p}+\frac{1}{q}=1$ and $q_1\in (1,2]$,
	\begin{align} 
	M_{\geq1,k,n}^b&\leq C(1+(Y_{t_k}^{b_0})^2)\Delta_n\left(\lambda_{\upsilon_n}\Delta_n\right)^{\frac{1}{q}}\bigg(1+(\sqrt{\Delta_n}\varphi_{n\Delta_n}(b_0))^{\frac{1}{p}} e^{C_2u^2(\varphi_{n\Delta_n}(b_0))^2\Delta_nY_{t_k}^{b_0}}\notag\\
	&\qquad \times  \big(1+\sqrt{Y_{t_k}^{b_0}}+\sqrt{\Delta_n}\vert u\vert\varphi_{n\Delta_n}(b_0) Y_{t_k}^{b_0}\big)^{\frac{1}{p}}\bigg)+C\Delta_n\bigg(\int_{z\leq\upsilon_n}z^2m(dz)\notag\\
	&\qquad+\Delta_n\Big(\int_{z\leq\upsilon_n}zm(dz)\Big)^2\bigg)+C\varphi_{n\Delta_n}(b_0)\Delta_n^{\frac{1}{2}+\frac{1}{q_1}}\Big(\int_{z\leq\upsilon_n}z^{2q_1}m(dz)\Big)^{\frac{1}{q_1}} \notag\\ 
	&\qquad\times e^{C_2u^2(\varphi_{n\Delta_n}(b_0))^2\Delta_nY_{t_k}^{b_0}} \big(1+\sqrt{Y_{t_k}^{b_0}}+\sqrt{\Delta_n}\vert u\vert\varphi_{n\Delta_n}(b_0) Y_{t_k}^{b_0}\big). \label{c3m2}
	\end{align}	
	Thus, from \eqref{D}, \eqref{D1}, \eqref{D2}, \eqref{D3}, \eqref{c3m1} and \eqref{c3m2}, the result follows.	
\end{proof}

\section{Appendix B: Useful results}
First, we recall a convergence result for triangular arrays of random variables. For each $n\in\mathbb{N}$, let $(\zeta_{k,n})_{k\geq 1}$ be a sequence of random variables defined on the filtered probability space $(\Omega, \mathcal{F}, \{\mathcal{F}_t\}_{t\in\R_+}, \P)$, and assume that $\zeta_{k,n}$ are $\mathcal{F}_{t_{k+1}}$-measurable for all $k$.
\begin{lemma}\label{zero} \textnormal{\cite[Lemma 3.4]{J12}} Assume that as $n  \rightarrow \infty$,  
	\begin{equation*} 
	\textnormal{(i)}\;  \sum_{k=0}^{n-1}\E\left[\zeta_{k,n}\vert \mathcal{F}_{t_k}\right] \overset{\P}{\longrightarrow} 0, \quad \text{ and } \quad \textnormal{(ii)} \,  \sum_{k=0}^{n-1}\E\left[\zeta_{k,n}^2\vert \mathcal{F}_{t_k} \right]\overset{\P}{\longrightarrow} 0.
	\end{equation*}
	Then 
	$
	\sum_{k=0}^{n-1}\zeta_{k,n}\overset{\P}{\longrightarrow} 0
	$ as $n  \rightarrow \infty$.
\end{lemma}
Next, we recall a so called stable central limit theorem for continuous local martingales. 
\begin{lemma}{\cite[Theorem 4.1]{Zan}}\label{THM_Zanten}
	Let $\bigl(\Omega, \mathcal{F}, \{\mathcal{F}_t\}_{t\in\R_+}, \P\bigr)$ be a filtered probability space satisfying the usual conditions.	Let $M=(M_t)_{t\in\R_+}$ be a square-integrable continuous local martingale w.r.t. the filtration $\{\mathcal{F}_t\}_{t\in\R_+}$ such that $\P(M_0 = 0) = 1$.
	Suppose that there exists a function $q : [t_0, \infty) \to \R$ with some $t_0 \in \R_+$ such that $q(t) \ne 0$ for all $t \in \R_+$, $\lim_{t\to\infty} q(t) = 0$ and
	\begin{equation}\label{Zanten1}
	q(t)^2 \langle M \rangle_t \overset{\P}{\longrightarrow} \eta^2 \qquad \text{as \ $t \to \infty$,}
	\end{equation}
	where $\eta$ is a random variable, and $(\left<M\right>_t)_{t\in\R_+}$ denotes the quadratic variation process of $M$.	Then, for each random variable $v$ defined on $(\Omega, \mathcal{F}, \P)$, we have
	\begin{equation}\label{Zanten2}
	\bigl(q(t) M_t, v\bigr)\overset{\mathcal{L}(\P)}{\longrightarrow} (\eta Z, v) \qquad \text{as \ $t \to \infty$,}
	\end{equation}
	where $Z$ is a standard normally distributed random variable independent of
	$(\eta, v)$.
	\ Moreover,
	\begin{equation}\label{Zanten3}
	\bigl(q(t) M_t, q(t)^2 \langle M \rangle_t\bigr) \overset{\mathcal{L}(\P)}{\longrightarrow} (\eta Z, \eta^2) \qquad
	\text{as \ $t \to \infty$.}
	\end{equation}
\end{lemma}

Note that \eqref{Zanten3} follows from \eqref{Zanten2} applied for $v = \eta^2$ and from
\eqref{Zanten1} by Theorem 2.7 (iv) of van der Vaart \cite{Vaart}.

We next present a comparison theorem given in Proposition A.1 in the ArXiv preprint of \cite{BBKP17}. 
\begin{lemma}\label{comparisontheorem} Let $a \in \R_+$, $b \in \R$, $\sigma \in \R_{++}$, and let $m$ be a L\'evy measure on $\R_{++}$ satisfying condition {\bf(A1)}. Let $\eta_0$ and $\overline{\eta}_0$ be random variables independent of $W$ and $J$ satisfying $\widehat{\P}(\eta_0\in \R_{+})=1$ and $\widehat{\P}(\overline{\eta}_0\in \R_{+})=1$. Let $(Y_t^b)_{t\in\R_+}$ be a pathwise unique strong solution of the SDE \eqref{eq1} such that $\widehat{\P}(Y_0^b = \eta_0) = 1$. Let $(\overline{Y}_t^b)_{t\in\R_+}$ be a pathwise unique strong solution of the SDE
	\begin{equation}\label{eq0}
		d\overline{Y}_t^b=\left(a-b\overline{Y}_t^b\right)dt +\sigma\sqrt{\overline{Y}_t^b}dW_t,
	\end{equation}	
	such that $\widehat{\P}(\overline{Y}_0^b = \overline{\eta}_0) = 1$. Then $\widehat{\P}(\eta_0 \geq \overline{\eta}_0)=1$ implies $\widehat{\P}(\text{$Y_t^b \geq \overline{Y}_t^b$ for all $t \in \R_+$}) = 1$.
\end{lemma}

\begin{theorem}\label{ruleproNualart}\textnormal{(Proposition 1.3.3 of Nualart \cite{N})} Let $F\in \mathbb{D}^{1,2}$ and $u$ be in the domain of $\delta$ such that $Fu\in L^2(\Omega;H)$. Then $Fu$ belongs to the domain of $\delta$ and the following equality is true 
	\begin{equation}\label{ruleproduct}
		\delta(Fu)=F\delta(u)-\left<DF,u\right>_H,
	\end{equation} 
	provided the right-hand side of \eqref{ruleproduct}	is square integrable.
\end{theorem}

\begin{theorem}\label{existenceMallideri} \textnormal{(Lemma 2.1 of Le{\'o}n, Navarro and Nualart \cite{LNN03})} Let $\varphi\in \mathcal{C}^1(\mathbb{R})$ be a continuously differentiable function and let $F\in \mathbb{D}^{1,2}$. Then $\varphi(F)\in \mathbb{D}^{1,2}$ if and only if $\varphi(F)\in L^2(\Omega)$ and $\varphi'(F)DF\in L^2(\Omega\times [0,T])$, and under these hypotheses 
	$$
	D(\varphi(F))=\varphi'(F)DF.
	$$
\end{theorem}

\end{document}